\documentclass[11pt]{amsart}

\usepackage{amsfonts,amsmath,latexsym,amssymb,verbatim,amsbsy,amsthm}
\usepackage{array}
\usepackage{graphicx}
\usepackage{caption}
\usepackage{float}
\usepackage{wrapfig}
\usepackage[shortlabels]{enumitem}
\usepackage{comment}
\usepackage{xfrac}

\usepackage[top=1in, bottom=1in, left=1in, right=1in]{geometry}

\usepackage[dvipsnames]{xcolor}

\usepackage[colorlinks=true, pdfstartview=FitV, linkcolor=RoyalBlue,citecolor=ForestGreen, urlcolor=blue]{hyperref}

\theoremstyle{plain}
\newtheorem{THEOREM}{Theorem}[section]

\newtheorem{Lemma}[THEOREM]{Lemma}

\newtheorem{theorem}[THEOREM]{Theorem}
\newtheorem{corollary}[THEOREM]{Corollary}
\newtheorem{lemma}[THEOREM]{Lemma}

\theoremstyle{definition}

\theoremstyle{remark}

\newtheorem{remark}[THEOREM]{Remark}

\DeclareMathOperator{\supp}{supp} %
\def \a {\alpha}

\def \d {\delta}

\def \e {\varepsilon}

\def \l {\lambda}

\def \bu {{\bf u}}

\def \bx {{\bf x}}
\def \by {{\bf y}}

\def \bx {{\bf x}}

\def \by {{\bf y}}

\def \cL {\mathcal{L}}

\def \cN {\mathcal{N}}

\def \cR {\mathcal{R}}

\def \ff {\mathsf{f}}

\usepackage{amsmath}

\newcommand{\N}{\ensuremath{\mathbb{N}}}   %%% naturals
\newcommand{\R}{\ensuremath{\mathbb{R}}}   %%% reals
\newcommand{\T}{\ensuremath{\mathbb{T}}}   %%% torus

\newcommand{\pa}{\partial}

\def \dx  {\, \mbox{d}x}

\def \dy  {\, \mbox{d}y}

\newcommand{\ep}{\varepsilon}
\def \S {C^{1+\alpha}_{0,\,\text{even}}}

\begin{document}

\title[2D EULER: POISEUILLE FLOW]{TRAVELING WAVES  NEAR POISEUILLE FLOW \\ FOR THE 2D EULER EQUATION }

\author{\'Angel Castro and Daniel Lear}

\address{Instituto de Ciencias Matematicas, Madrid.}
\email{angel\_castro@icmat.es}

\address{Universidad de Cantabria, Santander.}
\email{daniel.lear@unican.es}

\date{\today}



%\thanks{\textbf{Acknowledgment.}  AC is supported by the Spanish Ministry of Science and Innovation, through
%the “Severo Ochoa Programme for Centres of Excellence in R\&D (CEX2019-000904-S)”,
%and by grants  Europa Excelencia program ERC2018-092824 and  RED2018-102650-T funded by MCIN/AEI/10.13039/501100011033. AC and DL are supported by grants MTM2017-89976-P and  PID2020-114703GB-I00 funded by MCIN}

\begin{abstract}
In this paper we reveal the existence of a large family of new, nontrivial and Lipschitz traveling waves for the 2D Euler equation at an arbitrarily small distance from the Poiseuille flow in $H^s$, with $s<\sfrac{3}{2}$, at the level of the vorticity.
\end{abstract}

\maketitle

%\addtocontents{toc}{\protect\setcounter{tocdepth}{1}}

\tableofcontents

\section{Introduction and main result}
In this paper we consider the 2D Euler equation for an incompressible and inviscid fluid
\begin{align}
\pa_t \bu +(\bu\cdot\nabla) \bu =&-\nabla p,\label{euleru}\\
\nabla \cdot \bu =&0,\label{incompressibility}
\end{align}
in a periodic channel $\T\times [-1, 1]$ and with slip boundary condition
\begin{align}\label{slip}
%u_2(x,\pm 1,t)=0.
u_2\big|_{y=\pm 1}=0.
\end{align}
We also fix the circulation of the velocity, which is conserved. Indeed, we will take
\begin{align}\label{circulation}
\int_{-1}^1 u_1(x,y,t)dy=\frac{2}{3}, \qquad \forall t\geq 0.
\end{align}
In this way, the Poiseuille flow, which is given by
\begin{align}\label{poiseuille}
\bu_P(x,y)=(y^2,0),
\end{align}
solves \eqref{euleru}, \eqref{incompressibility}, \eqref{slip} and \eqref{circulation} (with pressure $p=0$).\medskip

The main objective of this paper is to study the evolution of 2D Euler equation near Poiseuille flow in a periodic channel with slip boundary condition.
We show that the dynamics of 2D Euler at low regularity is quite rich. In fact, we are able to show the existence of traveling waves at a distance as close as we want to Poiseuille measured in the $H^{3/2^-}(\T\times[-1,1])$ topology, at the level  of the  vorticity.

Investigating the behavior of Euler equations near shear flows, such as Poiseuille flow, transcends theoretical curiosity and is of paramount importance. Despite its simplicity, Poiseuille flow exhibits rich dynamics and has broad practical implications. For example, in biomedical engineering, it is essential for understanding blood flow in arteries and for designing medical devices such as catheters and drug delivery systems, see \cite{fung,nichols}.

For the sake of clarity, we will now give an elementary statement of our main result.
\begin{theorem}\label{thmbasic} For  any $0\leq s<\sfrac{3}{2}$ and $\epsilon>0$,  the 2D Euler system \eqref{euleru}, \eqref{incompressibility}, \eqref{slip} and \eqref{circulation} admits a nontrivial traveling wave solution at distance $\epsilon$ of the Poiseuille flow. That is, the travelling wave satisfies the smallness condition  $$\|\omega+2y\|_{H^s(\T\times[-1,1])}< \epsilon,$$
with $\omega\in W^{1,\infty}(\T\times[-1,1])$ and whose speed is $O(1)$ with respect to $\epsilon$.
\end{theorem}

\begin{remark}
Here and throughout the rest of the paper, ``nontrivial solution'' means that the velocity is not a shear flow, i.e., that the dependence of $u_1$ on the horizontal variable $x$ is nontrivial.
\end{remark}

This result will follow from the combination of Theorems \ref{thm:detallado} and \ref{t:distance}   where the properties of these solutions will be stated in more detail.
We hope that the techniques introduced here will be sufficiently robust and flexible for studying the existence of traveling wave solutions near other more general shear flows.

\subsection{Motivation and background}

The study of the  Euler equation in 2D remains a vibrant and active field of research. Unlike its 3D counterpart, the global existence of smooth solutions is a classic result dating back to the 1930's, see \cite{H,W}. However, understanding the qualitative behavior of these global solutions presents formidable difficulties and challenges.

Understanding the long time behaviour of these solutions is especially difficult due to the absence of a global relaxation mechanism. In light of these challenges, it is natural to investigate possible ``end state'' configurations to better understand the long-time dynamics. These configurations serve as important benchmarks and may even emerge as dominant aspects of the dynamics of 2D Euler.\\

There have been some attempts to construct a theory for the generic long-term behavior of solutions to the 2D Euler equation. Numerical simulations suggest that, in the generic case, the solutions, while exhibiting complicated fine-scale behavior, may exhibit certain structures on the unit scale. Examples of such ``end states'' configurations include stationary, time-periodic, quasi-periodic, and perhaps chaotic states.

In this direction, there are a couple of far-reaching conjectures regarding the long-term behavior of the 2D Euler equation on bounded planar domains $M \subset \mathbb{R}^2$:

\begin{itemize}
	\item \textbf{Conjecture } (\v{S}verák (2011), \cite{S}).
Generic initial data $\omega_0\in L^{\infty}(M)$ gives rise to inviscid incompressible motions whose vorticity orbits $\{\omega(t)\}_{t\in\R}$ are not precompact in $L^2(M)$.

	\item \textbf{Conjecture } (Shnirelman (2013), \cite{Sh})
For any initial data $\omega_0\in L^{\infty}(M),$ the collection of $L^2(M)$ weak limits of the orbit $\{\omega(t)\}_{t\in\R}$  consists of vorticities which  generate $L^2(M)$ precompact orbits under 2D Euler evolution.
\end{itemize}

This would help explain the local chaos versus global structure phenomenon. We refer the reader to \cite{BL,DD,DEJ,EMS} and references therein for important recent works on these and related problems, and Section 3 of the recent book \cite{DE} for an excellent survey. A rigorous proof of these conjectures is beyond the scope of current PDE techniques. A more realistic and achievable approach is to consider the 2D Euler equations in physically relevant perturbative regimes, which is the problem we consider in this paper below.\\

A particularly important class of solutions of the Euler equations is given by the steady states, i.e. time-independent flow configurations. Thus, a fundamental question that naturally arises is how the solutions near such steady states behave and, in particular, whether they are stable in an appropriate sense, which must be defined very carefully.

%Among the steady states of the Euler equations, a class of particular importance arises from shear flows, exemplified by well-known configurations such as Couette, Poiseuille, or Kolmogorov flows. Since the foundational investigations of Kelvin [?] and Reynolds [?] in the 1880's, shear flows have been important in both theory and applications of fluid dynamics, and are commonly regarded as the natural state of a fluid in non-turbulent situations.\\

The study of the stability properties of these steady states is a classical topic and a fundamental problem in hydrodynamics. In the 1960's, Arnold \cite{Arnold} proved a general stability criteria, using the energy Casimir method, but this  does not give asymptotic information on the global dynamics.

The full nonlinear asymptotic stability problem has only been investigated recently, starting with the breakthrough work of Bedrossian-Masmoudi \cite{BM}, who proved \textit{inviscid damping} and nonlinear stability
in the Gevrey spaces $\mathcal{G}^s$, with $s >\sfrac{1}{2}$, for perturbations of the Couette flow on $\T\times\R$.

%This \textit{inviscid damping} phenomenon share analogies with the \textit{Landau damping} in the kinetic theory of plasma physics, see the paper of Mohout–Villani \cite{} and references therein.

%The same result does not hold for $s < \sfrac{1}{2}$, as it was shown in \cite{DM} by Deng-Masmoudi, see also \cite{DZ2}.
This result does not hold for $s < \sfrac{1}{2}$, as demonstrated in \cite{DM} by Deng-Masmoudi. See also \cite{DZ2}.

%See also the very recent results \cite{CWZZ,Zhao} for the inhomogeneous case.

In very recent and independent works Ionescu-Jia \cite{IJ_otro,IJ_acta} and Masmoudi-Zhao \cite{MZ} proved that nonlinear asymptotic stability holds true also for perturbations around the Couette flow and more general monotone shear flows in the finite periodic channel $\T\times [0,1]$.

Motivated by the results mentioned above, the linearized equations around more general shear flows have been intensively investigated in the last few years. See, for example, \cite{BZV,BMlinear,DZ1,DZ2,GNRS,Jia1,Jia2,Wei-Zhang-Zhao,Wei-Zhang-Zhao_1,Wei-Zhang-Zhao_2,Zillinger1,Zillinger2} and their references for a complete, but not exhaustive, list of references.\\

By all the above, an important direction to explore further is how the regularity of initial data may or may not affect the dynamics. In this direction we refer to the work of Lin–Zeng \cite{LZ}, where they proved that nonlinear inviscid damping is not true for perturbations of the Couette flow in $H^s$ with $s < \sfrac{3}{2}$. In addition, the non existence of nontrivial traveling waves close to the Couette in $H^s$, with $s > \tfrac{3}{2}$, is also proved.

More specifically, they discovered Kelvin’s cat’s eyes steady states nearby to the Couette shear flow. In fact, the dynamics at this level of regularity exhibit even greater complexity and there are also travelling waves with speed of order O(1) and zero circulation, as shown by the authors in \cite{CL}.

For other related equations like 2D Euler equation with Coriolis force or on a rotating sphere we refer to \cite{LWZZ,WZZhu} and \cite{CP,N} respectively.\\

Recently, there has been a growing interest in the study of the existence or absence of invariant structures such as steady-states and traveling waves for the incompressible 2D Euler equations near other shear flows. Here, it is important to mention the result of Coti Zelati–Elgindi–Widmayer \cite{ZEW}. They prove the existence of analytic nontrivial steady-states near Kolmogorov flow  $\mathbf{u}_K=(\sin(y),0)$ in the domain $\T\times[0,2\pi]$. In this result the distance is measured in $C^w$ (analytic functions). On the contrary, in the domain $\T\times [0,2\pi\delta]$, with $\delta\notin \mathbb{N}$, they show a rigidity result: nonexistence of  nontrivial traveling waves near Kolmogorov flow in the $H^2$ topology. Finally, they also prove the nonexistence of nontrivial traveling waves near the Poiseuille flow in the $H^{5+}$ topology in $\T\times [-1,1]$. This last result complements Theorem \ref{thmbasic}, but it still leaves a gap for $\sfrac{3}{2}\leq s \leq 5$.\\

\begin{comment}
In addition,
the non existence of nontrivial traveling waves close to the  Poiseuille and Kolmogorov flow in $H^{5^+}$ for a channel $\T\times[-1,1]$ and $H^3$ for a  rectangular torus $\T_\d^2$ is also proved.
\end{comment}

For the Navier-Stokes equation, when the viscosity $\nu > 0$ is small but nonzero, another significant physical phenomenon emerges, known as \textit{enhanced dissipation}, which contributes to stabilizing the flow. Essentially, the background flow mixes the vorticity field, amplifying the viscous effect in averaging the vorticity function. This leads to a quicker decay of the vorticity compared to when the viscous effect acts alone. Using enhanced dissipation allows for establishing improved nonlinear asymptotic stability results for perturbations of the order $O(\nu^\gamma)$, with a suitable $\gamma > 0$. Determining the smallest $\gamma$ for which nonlinear stability remains valid is a crucial challenge and represents an active research area, addressing the ``transition threshold conjecture''. For comprehensive insights into enhanced dissipation and transition threshold problems, we direct readers to recent works \cite{AH,BGM,BMV,BZ,CLWZ,CWZ,G,GW,LWZ,LX,MZ1,MZ2,WZZ,WZ} and the references therein. In our context, it is pertinent to point out the importance of several recent works \cite{CDLZ,CWZ1,CZEW,DZ,DL1,WX} focused on Poiseuille flow.

\subsection{Tools and ideas behind the proof}
In this section, we explain the key ingredients of the proof and highlight the main differences and new challenges that arise in the case of Poiseuille flow with respect to Couette and Taylor-Couette \cite{CL,CL1}.

\textit{Step 1: Formulation of the problem via vorticity contour lines.}
As in our previous works \cite{CL,CL1} the starting point will be the use of the transport character of the vorticity equation to translate the problem to the study of vorticity  contour lines. That is, we will assume that the level sets of the vorticity are given by the family of graphs $(x,y+f(\bx,t))$, i.e.,
\[
\omega(x,y+f(\bx,t),t)=\varpi(y), \qquad \forall t\geq 0,
\]
for certain profile $\varpi$ which will be fixed in an appropriate manner later. Then, all the information of the problem is encoded in the time-evolution of $f(\bx,t)$ and in  $\varpi$. After some algebraic manipulations the problem reduces to the equation
\[
\pa_t f(\bx,t)=\pa_x\overline{\Psi}_{\varpi}[f](\bx,t), \quad \text{over  $\T\times\stackrel{\circ}\supp\left(\varpi'\right)$},
\]
with boundary conditions $f|_{y=\pm 1}=0$ and where $\pa_x\overline{\Psi}_{\varpi}[\cdot]$ is a nonlocal and nonlinear functional.

To search for traveling wave solutions we will take, with a little abuse of notation, the ansatz $f(\bx,t):=f(x+\l t,y)$ which will lead to a time-independent equation for $(\l, f(\bx))$ of the form
\begin{equation}\label{introFunctional}
F_\varpi[\l,f]=0, \quad \text{over  $\T\times\stackrel{\circ}\supp\left(\varpi'\right)$}.
\end{equation}

Bifurcation theory, particularly through the Crandall-Rabinowitz theorem, will be a essential tool in our demonstration. This theory has proven to be invaluable in establishing the existence of solutions to a wide range of equations within the field of fluid mechanics.\medskip

\textit{Step 2: Bifurcation theory and the profile function $\varpi$.}  The above functional equation \eqref{introFunctional} will be solved by using the Crandall-Rabinowitz theorem as one can  bifurcate from zero for some values of $\l$. The main task of the manuscript will be to check all the required conditions of C-R theorem.

Prior to that, it is important to emphasize that this approach has enabled us to reduce the problem just to the support of the vorticity gradient, where all the dynamics occur.

An important difference, which makes this problem challenging and necessitates novel approaches for its resolution, lies in the fact that for both Couette  and  Taylor-Couette flow, their respective vorticities remain constant. Consequently, their gradients vanish across the entire working domain. This circumstance allows us to seek perturbed profiles whose gradient is localized in two small regions, each comparable in size to the distance between the time periodic solution and the shear/radial flow.

 This is in  stark contrast  with the case at hand, where the gradient of the associated vorticity is a non-zero constant. This fact implies we are forced to seek new perturbed profiles whose gradient is nontrivial in almost  the entirety of the working domain. As we will see in Section \ref{s:analysisLOP}, this issue makes the analysis of the speed and direction of bifurcation completely different.

We advance that our perturbed profile $\varpi\in W^{1,\infty}([-1,1])$ will be given by an odd function defined through
\[
\varpi'(y)=(-2)\begin{cases}
    1 \qquad \phantom{.}0\leq y < a_\ep,\\
    0 \qquad a_\ep \leq y \leq b_\ep,\\
    1 \qquad b_\ep < y \leq 1,
\end{cases}
\]
with
\[
a_\e \equiv \frac{1-\ep}{2}, \qquad b_\e\equiv \frac{1+\ep}{2}.
\]

Notice that there is a large part where $\varpi'\neq 0$ and a small part, of order $\ep$, where $\varpi'=0$. Both of them will play an essential role in the proof of theorem \ref{thmbasic}.\medskip

\textit{Step 3: Regularity of the functional for a  Lipschitz profile $\varpi$.}

A difficulty to be faced is related to the non explicit form of the Green’s function associated to the Laplacian in a channel $\T\times[-1,1]$. The situation in bounded domains turns out to be more delicate due to the presence of boundaries. This has an impact on the study of the regularity of the functional that will describe the traveling wave solutions. This type of difficulty has appeared previously in the literature, see for example \cite{CL1} for the Euler equation in an annular domain or \cite{HXX} for the generalized SQG equation in the disc. Here, we will adopt a similar approach to that used in \cite{CL1}.

Another new technical difference from our previous results is that  we will need to achieve the required regularity of the functional for a Lipschitz profile $\varpi$,  in contrast to the smooth profiles used previously in \cite{CL,CL1}. Roughly speaking, we conduct a detailed study through the $L^p$ theory of  the elliptic problem associated with the stream function.

Here, it is important to note that the fact that $\varpi \in W^{1,\infty}([-1,1])$ implies that the traveling waves obtained in the paper belong to $W^{1,\infty}(\T \times [-1,1])$, but they  are not in $C^1(\T \times [-1,1])$.  We hope that the same result holds with a smooth regularization of the profile $\varpi$, which would allow us to obtain smooth traveling waves. However, the analysis of the linear operator, needed in the C-R theorem and carried out in Section \ref{s:analysisLOP}, would be much more involved with a smooth profile. Because of that, here we work with a Lipschitz profile and pay the price of some complications in the proof of the differentiability of the functional.\medskip

\textit{Step 4: Analysis of the linear operator.}
To implement C-R theorem and generate a nontrivial solution to the equation \eqref{introFunctional}, one needs to check the required spectral properties for the linearized operator at $f=0$.

This linear operator  is given by $\cL^\l \equiv D_f F_\varpi[0,\l]$ and  for any  function
\[
h(x,y)=\sum_{n\geq 1}h_n(y)\cos(nx),
\]
we have that
\[
\cL^\lambda h(x,y) =\sum_{n\geq 1} \cL^\lambda_{n}h_n (y)\cos(nx), \quad \text{over  $\T\times\stackrel{\circ}\supp\left(\varpi'\right)$},
\]
where each operator $\cL^\lambda_{n}$ is of integral type. More precisely, after employing the symmetries of the system, we obtain
\begin{align*}
\cL_{n}^\lambda h_n(y)=(\lambda-\Psi'_0(y))h_n(y)-\int_{0}^1\varpi'(z)G_n(y,z)h_n(z)dz, \qquad y\in [0,1]\cap\stackrel{\circ}\supp\left(\varpi'\right),
\end{align*}
with
\[
\Psi'_0(y)=\int_{0}^y\varpi(\bar{y})d\bar{y}-\int_{0}^1\int_{0}^{\bar{y}}\varpi(z)dzd\bar{y}-\frac{1}{3},
\]
and where $G_n(\cdot,\cdot)$ is just the Green function associated to the operator $\pa_y^2 -n^2$ on $[0,1]$. That is,
\begin{align*}
G_n(y,z)=\frac{1}{n\sinh(n)}\left\{\begin{array}{cc}\sinh(n(1-y))\sinh(nz) &\quad  z<y,\\ \sinh(ny)\sinh(n(1-z)) &\quad  z>y.\end{array}\right.
\end{align*}

The main core of the paper consists of  the study of the spectral properties of $\mathcal{L}^\lambda$. In particular, a large part will be dedicated to obtain an element in its kernel. Here, in this road map of the proof, we give particular attention to this issue since it is the base also to obtain the remaining necessary conditions to apply the C-R theorem.

\textit{Step 4.1:  Choice of the mode, manipulations and rescaling.} In order to  find an element in the kernel, a usual trick is to focus our attention in just one mode. For simplicity, in this problem we fix $m=1$ (that is, $h\equiv h_1$ and $h_n\equiv 0,  \text{ for all } n\neq m$ with $m=1$) and using the fact that $\varpi'(y)=-2\chi_{[0,a_\ep)\cup(b_\ep,1]}(y)$ we have reduced the problem to find $(\l,h)$ such that
\begin{equation}\label{eqintegralintro}
(\lambda-\Psi'_0(y))h(y)+2\int_{[0,a_\ep)\cup(b_\ep,1]}G(y,z)h(z)dz=0, \qquad y\in [0,a_\ep)\cup(b_\ep,1],
\end{equation}
with the boundary conditions
\[
h(0)=0=h(1).
\]
%\begin{remark}\label{ABtildeintro}
%Since \eqref{eqintegralintro} is a linear equation we have that the values of $h(a_\ep)$ and $h(b_\ep)$ are uniquely determined up to multiplication by a constant. For simplicity, in the rest, we set that $\tilde{A}\equiv h(a_\ep)$ and $\tilde{B}\equiv h(b_\ep)$ for some $\tilde{A},\tilde{B}\in\R.$
%\end{remark}

\begin{remark}
Here, it is important to note that the associated integral equation for both Couette and Taylor-Couette flow have an integration domain of size of order $O(\ep)$. This fact play a fundamental role in the proofs in \cite{CL,CL1}. However, in this setting we have a domain of integration of order $O(1)$ which requires new ideas and use a completely different approach to proceed.
\end{remark}

The main goal is to look for a solution of \eqref{eqintegralintro} with speed $\l$ that satisfies $\Psi_0'(b_\e)< \l < \Psi_0'(a_\e)$. This means that $\l=\Psi_0'(y^\ast)$ for some $y^*\in(a_\e,b_\e)$. This way we  avoid any zero of $(\lambda-\Psi'_0(y))$ in $[0,a_\ep]\cup[b_\ep,1]$.  In fact,
\[
\lambda-\Psi'_0(y)=\begin{cases}
    (y-\mu)(y+\mu), \qquad 0\leq y < a_\e,\\
    (y-\nu)(y-\rho), \qquad b_\e< y\leq 1,
\end{cases}
\]
for some $\mu=\mu(\lambda), \nu(\mu), \rho(\mu) \in \R$ satisfying
\begin{equation}\label{mu>a_intro}
\mu>a_\ep,\qquad \text{and} \qquad \rho<\nu< b_\ep.
\end{equation}

Next we introduce the function
\[
H(y)\equiv\begin{cases}
h(y)(y-\mu)(y+\mu),\qquad 0\leq y<a_\e,\\
h(y)(y-\nu)(y-\rho),\qquad b_\e<y\leq 1,
\end{cases}
\]
which allow us to write \eqref{eqintegralintro} as
\begin{align}\label{compactintro}
H(y)+2\int_{0}^1 G(y,z)\left(\frac{\chi_{[0,a_\ep)}(z)}{z^2-\mu^2}+\frac{\chi_{(b_\ep,1]}(z)}{(z-\nu)(z-\rho)}\right)H(z)dz=0, \qquad y\in[0,a_\ep)\cup(b_\ep,1].
\end{align}

Let us rescale equation \eqref{compactintro} in a suitable manner to fix the domain. To achieve this, we introduce the following auxiliary rescaled variable $x$, defined as
\[
x=\begin{cases}
\frac{y}{1-\ep},\qquad \text{for $0\leq y<a_\ep$}, \\
\frac{y-\ep}{1-\ep}, \qquad \text{for $b_\ep<y\leq 1$},
\end{cases}
\]
and the rescaled auxiliary function
\[
F(x)=\begin{cases}
H((1-\ep)x),\qquad &\text{for $0\leq x<\frac{1}{2}$},\\
H((1-\ep)x+\ep), \qquad &\text{for $\frac{1}{2}<x\leq 1$}.
\end{cases}
\]
\begin{remark}
Here, we emphasize that  $F(x)$ does not have to be continuous at the point $x=1/2$.
\end{remark}

This procedure allow us to pass from a problem \eqref{compactintro} in a domain $[0,a_\ep)\cup(b_\ep,1]$ depending of the parameter $\ep$ to a fixed domain $[0,1/2)\cup(1/2,1]$ independent of it.
In this way, we have translated our problem to obtain $$F(x)=F^L(x)\chi_{[0,1/2)}(x)+F^R(x)\chi_{(1/2,1]}(x),$$ where $F^{L}$ and $F^R$ solve respectively
\begin{multline}\label{inteqleft_intro}
   F^L(x) 
    + \int_0^{1/2} \tilde{G}_\ep(x,z)\frac{F^L(z)}{z^2-\tilde{\mu}^2}dz
    +\int_{1/2}^1 \tilde{G}_\ep(x,z)\frac{F^R(z)}{(z-\tilde{\nu})(z-\tilde{\rho})}dz=0, \qquad 0\leq x < 1/2,
\end{multline}
and
\begin{multline}\label{inteqright_intro}
     F^R(x)+
    \int_0^{1/2} \tilde{G}_\ep(x,z)\frac{F^L(z)}{z^2-\tilde{\mu}^2}dz
    +\int_{1/2}^1 \tilde{G}_\ep(x,z)\frac{F^R(z)}{(z-\tilde{\nu})(z-\tilde{\rho})}dz=0, \qquad 1/2< x\leq 1,
\end{multline}
with the boundary conditions
\begin{equation}\label{BCforFLRintro}
F^L(0)=0=F^R(1),
\end{equation}
%and
%\begin{equation}\label{BCforFLRcentrointro}
%\lim_{x\to 1/2^-}F^L(x)=\tilde{A}(a_\ep-\mu)(a_\ep+\mu)\equiv A, \qquad \lim_{x\to 1/2^+}F^R(x)=\tilde{B}(b_\ep-\nu)(b_\ep-\rho)\equiv B,
%\end{equation}
where
\begin{equation}\label{tildevariablesintro}
\tilde{\mu}=\frac{\mu}{1-\ep}, \quad  \text{ and } \quad \tilde{\nu}=\frac{\nu-\ep}{1-\ep}, \,\, \tilde{\rho}=\frac{\rho-\ep}{1-\ep}.
\end{equation}
In the end,  the problem reduces to obtain $(F,\tilde{\mu})$ solving \eqref{inteqleft_intro}-\eqref{BCforFLRintro}.\medskip

\textit{Step 4.2: Integral equation to ODE transformation.}
It is straightforward to check that if the pair $F\equiv (F^L,F^R)$ solves \eqref{inteqleft_intro}-\eqref{inteqright_intro}, then it also solves the following system of ordinary differential equations
\begin{align}
\left(f^L\right)''(x)-\left(\frac{2}{x^2-\tilde{\mu}^2}+(1-\ep)^2\right)f^L(x)&=0, \qquad \text{for }0\leq x< 1/2, \label{odeleft_intro}\\
f^L(0)&=0, \label{bcleft_intro}
\end{align}
and
\begin{align}
\left(f^R\right)''(x)-\left(\frac{2}{(x-\tilde{\nu})(x-\tilde{\rho})}+(1-\ep)^2\right)f^R(x)&=0, \qquad \text{for } 1/2 < x \leq 1, \label{oderight_intro}\\
f^R(1)&=0. \label{bcright_intro}
\end{align}
Due to \eqref{mu>a_intro} and \eqref{tildevariablesintro} we must impose
\begin{align}\label{rhonumu}\tilde{\rho}<\tilde{\nu}< 1/2 < \tilde{\mu}.\end{align}
Thus, the only regular-singular points $x=\tilde{\mu}$ and $x=\tilde{\nu}$ are outside the domain where we are working. In any case, using the Frobenius series method we can obtain two linearly independent solutions of each ODE equation. That is, we are able to
solve \eqref{odeleft_intro} over  $[0,\tilde{\mu})$ and \eqref{oderight_intro}  over $(\tilde{\nu},1]$.

 Since we are working with a linear second-order ODE we will include an ``extra'' boundary condition given by
\begin{equation}\label{bccentro_intro}
\lim_{x\to 1/2^-} f^L(x)=1=\lim_{x\to 1/2^+} f^R(x).
\end{equation}
Using a linear combination of the two linearly independent solutions obtained earlier for each ODE, we can also obtain a solution for each ODE that satisfies \eqref{bcleft_intro}, \eqref{bcright_intro} and \eqref{bccentro_intro}.

\begin{remark}
It is important to emphasize that the reciprocal is not true in general. That is, the solution of the ODE system \eqref{odeleft_intro}-\eqref{bcleft_intro}-\eqref{bccentro_intro} and \eqref{oderight_intro}-\eqref{bcright_intro}-\eqref{bccentro_intro} is not in general a solution of the original integral system \eqref{inteqleft_intro}-\eqref{inteqright_intro}. We must do a suitable choice of the values at points $x=1/2^{\pm}$.
\end{remark}

To recover the solution of the original integral system \eqref{inteqleft_intro}-\eqref{inteqright_intro}, we consider the candidate
\begin{equation}\label{candidateintro}
F(x)=A f^L(x)\chi_{[0,1/2)}(x)+ B f^R(x)\chi_{(1/2,1]}(x).
\end{equation}

\textit{Step 4.3: Reducing the problem to an equation for $\tilde{\mu}$.}
Using our candidate \eqref{candidateintro}, it can be  verified that
$(A,B)$ needs to satisfy
\begin{align}\label{matrixintro}
\left(\begin{array}{cc}1+\sinh(b_\ep)I_1 & \sinh(a_\ep)I_2\\ \sinh(a_\ep)I_1  &  1+\sinh(b_\ep)I_2 \end{array}\right)\left(\begin{array}{cc} A \\ B\end{array}\right)=0,
\end{align}
with
\[
I_1\equiv \int_{0}^\frac{1}{2}\tilde{G}_\ep(1/2^-,z)\frac{f^L(z)}{z^2-\tilde{\mu}^2}dz, \qquad I_2\equiv \int_{\frac{1}{2}}^1 \tilde{G}_\ep(1/2^+,z)\frac{f^R(z)}{(z-\tilde{\nu})(z-\tilde{\rho})}dz.
\]
Thus, in order to get a nontrivial solution $(A,B)\neq 0$ it must hold
\begin{align}\label{detsolintro}
1
+\sinh(b_\ep)\left(I_1+I_2\right)
+\left(\sinh^2(b_\ep)-\sinh^2(a_\ep)\right)I_1 I_2 =0,
\end{align}
which is actually an equation for $\tilde{\mu}$. Thus, we have finally reduced the problem to find $\tilde{\mu}$ satisfying inequalities \eqref{rhonumu}  and   solving  equation \eqref{detsolintro}. It is important to note that $\tilde{\mu}$ appears implicitly in the expressions of $f^L$ and $f^R$ as well.\medskip

\textit{Step 4.4: Analyzing the limiting problem.}
A crucial step to obtain $\tilde{\mu}$  solving the equation \eqref{detsolintro} and \eqref{rhonumu} will be to study the limiting problem that appears when we make the parameter $\ep$ to tend to zero.

We need to introduce an ansatz for $\tilde{\mu}$. That is,
\begin{align}\label{ansatzintro}
\tilde{\mu}= \frac{1}{2}+\left(\frac{1}{2}+\mu_1\right)\ep +\tilde{\mu}_2(\ep)\ep^2\log^2(\ep),
\end{align}
where $-\sfrac{1}{2} < \mu_1 < +\sfrac{1}{2}$ and $\tilde{\mu}_2(\ep)=O(1)$ in terms of $\ep.$ The fact that $\mu_1\in(-\sfrac{1}{2},+\sfrac{1}{2})$ implies that $\mu>a_\ep$ and $\rho<\nu<b_\ep$ for $\ep$ small enough.\\

The idea is as follows:

Suppose there exist $(F_\ep,\tilde{\mu}_\ep)$,  that solve our integral system  and satisfy \eqref{rhonumu} for all $\ep>0$ small enough. In particular, these solutions also solve the associated system of ODEs. Then, it is natural to ask what problem satisfies the limit when $\ep\to 0^+$. The function $F_0$, that formally appears, solves
\begin{align}\label{eqf0intro}
F_0''(x)-\left(\frac{2}{x^2-(1/2)^2}+1\right)F_0(x)=0, \quad \text{for }0\leq x<1/2 \text{ and } 1/2 <x\leq 1,
\end{align}
with
\[
F_0(0)=0=F_0(1).
\]

As we did before, the above system can also be analyzed and solved using the Frobenius series method over $[0,1/2)\cup(1/2,1]$, since the only regular-singular point is $x=1/2$.

In addition, if we assume $$F_0(1/2^-)=F_0(1/2^+)=1$$ we can pass to the limit in \eqref{detsolintro} to get
\begin{align*}
&1+C_1\sinh(1/2)\int_{0}^\frac{1}{2}\frac{\sinh(x)F_0(x)-\sinh(1/2)}{x^2-1/4}dx\\
&+C_1\sinh(1/2)
\int_{1/2}^1\frac{\sinh(1-x)F_0(x)-\sinh(1/2)}{x^2-1/4}dx
\nonumber\\ &+C_1\sinh^2(1/2)\log\left(\frac{1+2\mu_1}{1-2\mu_1}\right)-C_1\sinh^2(1/2)\log(3)=0.
\end{align*}

This expression can be understood as an equation for $\mu_1$, with a unique solution in $(-1/2,1/2)$.

The real proof consists of using this object $(F_0,\mu_1)$ to show the existence of a solution of \eqref{detsolintro}, $\tilde{\mu}_\ep$, for $\ep>0$ small enough. Indeed, \eqref{detsolintro}, with the ansatz \eqref{ansatzintro} can be written as
\[
\tilde{\mu}_2=\mathcal{G}[F_0,\mu_1,\ep]+o_\ep(1)\mathcal{F}[\tilde{\mu}_2,f[\tilde{\mu}_2]].
\]

This equation for $\tilde{\mu}_2$ is obtained by comparing $f[\tilde{\mu}_2]$ in \eqref{odeleft_intro}-\eqref{bcright_intro} and $F_0$. And it can be solved by Brouwer fixed point theorem.

\medskip

\textit{Step 5: The remaining conditions of C-R theorem and the distance to the Poiseuille flow.}
In the last part of the paper we verified the  remaining necessary conditions regarding the codimension of $\mathcal{L}^\lambda$ and the transversality property in C-R theorem. Showing that $\mathcal{L}^{\lambda}$ has codimension one still needs a new argument and the proof differs  from \cite{CL,CL1}. The transversality property is, instead, easy to obtain. Finally, we check that the distance between the traveling waves solutions and the Poiseuille flow is small in $H^{\frac{3}{2}-}$.

\subsection{Organization}
The rest of the paper is organized as follows. In Section \ref{s:formulationproblem}, we will derive the equation for the vorticity contour lines using the Biot-Savart law and set the profile function $\varpi$. In Section \ref{s:bifurcation},  we introduce and review some background on bifurcation theory and  C-R theorem. In Section \ref{s:spaces}, we define the spaces we will work with to apply the C-R theorem and study the regularity of the nonlinear functional. Sections \ref{s:analysisLOP}-\ref{s:transversality} constitute the core of the manuscript, where we conduct the spectral study and verify that all the hypotheses of the C-R theorem are satisfied. Finally, in Section \ref{s:mainthm}, we quantify the distance between our solution and the Poiseuille flow.\medskip

\section{Formulation of the problem}\label{s:formulationproblem}

We consider the 2D Euler equations \eqref{euleru}, \eqref{incompressibility}, \eqref{slip} and \eqref{circulation},  in the vorticity formulation, i.e.,
\begin{align}\label{eulerw}
\pa_t \omega+ \bu\cdot \nabla \omega=0,
\end{align}
where the velocity is rescued from the vorticity through the stream function $\Psi$, i.e.,
\begin{equation}\label{stream}
\Delta \Psi= \omega, \qquad \Psi\big|_{y=\pm 1}=\mp\frac{1}{3},
\end{equation}
and
\begin{align}\label{velostream}
\bu =\nabla^\perp \Psi \equiv \left(-\pa_y\Psi,\, \pa_x\Psi\right).
\end{align}
Notice that it is straightforward to check that $\omega(\bx,t)=-2 y$ solves \eqref{eulerw}-\eqref{stream} with
$\bu(\bx,t)=\bu_P(\bx)$.

\subsection{The equation for the traveling waves}
In order to study \eqref{eulerw}, \eqref{stream} and \eqref{velostream}, we look at the level sets of the vorticity $\omega$ inside of the support of $\nabla \omega$. We will assume that these level sets can be parameterized as a family of graphs, i.e.,
\begin{align*}
\omega(x,y+f(x,y,t),t)=\varpi(y),
\end{align*}
for some Lipschitz profile $\varpi$ (independent of time).

By straightforward computations, we get
\begin{align*}
\nabla \omega(x,y+f(\bx,t),t)=\frac{\left(-\pa_xf(\bx,t), 1\right)}{1+\pa_yf(\bx,t)}\varpi'(y).
\end{align*}
Then, using  \eqref{velostream} into \eqref{eulerw}, we can check that
\begin{align}\label{ecuacionlevel1}
\left(\pa_t f(\bx,t)-\pa_x \overline{\Psi}(\bx,t)\right)\frac{\varpi'(y)}{1+\pa_y f(\bx,t)}=0,
\end{align}
where
\begin{align*}
\overline{\Psi}(\bx,t)\equiv\Psi(x,y+f(\bx,t),t).
\end{align*}

In the following, we will use the notation
\begin{align*}
\Phi_f(x,y)\equiv (x,y+f(x,y)),
\end{align*}
and we also define the auxiliary domain
\begin{align*}
\Omega_f\equiv\{ (\bar{x},\bar{y})\in \T \times [-1,1]\, :\, (\bar{x},\bar{y})=\Phi_f(x,y)  \text{ with } x\in \T \text{ and } y\in \supp(\varpi')\}.
\end{align*}
Now we observe that
\begin{align*}
\Omega_f=\supp\left(\nabla \omega\right),
\end{align*}
and we emphasize that we just have to solve the equation
\begin{align}\label{ecuacionlevel2}
\pa_t f(\bx,t)=\pa_x\overline{\Psi}(\bx,t) \quad \text{in $\T\times\stackrel{\circ}\supp\left(\varpi'\right)$},
\end{align}
with boundary condition
\begin{align}\label{boundarycondf}
%f(x,\pm 1,t)=0,
f\big|_{y=\pm 1}=0,
\end{align}
as we can not move the top and bottom of the channel.

Since we are looking for traveling waves solutions, we will introduce the ansatz
$$f(x,y,t)\equiv f(x+\lambda t, y)$$ with a slight abuse of notation. Then, our problem \eqref{ecuacionlevel2} becomes
\begin{align*}
\lambda \pa_x f(\bx) =\pa_x \overline{\Psi}_f(\bx) \quad \text{in $\T\times\stackrel{\circ}\supp(\varpi')$},
\end{align*}
with the same boundary condition \eqref{boundarycondf} and with $\overline{\Psi}_f=\Psi_f\circ \Phi_f$. That is,
\begin{align*}
\overline{\Psi}_f(\bx)=\Psi_f(x,y+f(x,y)),
\end{align*}
with
\begin{equation*}
\left\{
\begin{aligned}
\Delta \Psi_f=& \omega_f\quad \text{in $\T\times [-1,1],$}\\
\Psi_f\big|_{y=\pm 1}=& \mp \frac{1}{3},
\end{aligned}
\right.
\end{equation*}
%\begin{align*}
%\Delta \Psi_f=& \omega_f\quad %\text{in $\T\times [-1,1],$}\\
%\Psi_f\big|_{y=\pm 1}=& \mp \frac{1}{3},
%\end{align*}
where
\begin{equation*}
\omega_f(\bx)=\left\{
\begin{aligned}
\varpi\circ\left(\Phi_{f}^{-1}\right)_2(\bx)\quad &\qquad\text{in $\Omega_f,$}\\
\text{constant} &\qquad \text{in $\T\times [-1,1]\setminus \Omega_f.$}
\end{aligned}
\right.
\end{equation*}
%where, for $\bx\in \Omega_f$,
%\begin{align*}
%\omega(\bx)= \varpi\circ\left(\Phi_{f}^{-1}\right)_2(\bx)
%\end{align*}
%and for the $\bx\in \T\times [-1,1]\setminus \Omega_f$, $\omega$ is extended in a continuous way.

\subsubsection{The profile function $\varpi$}
In this section, we set the profile function $\varpi(y)$ over $y\in [-1,1]$. In first place, we take $\varpi(y)$ an odd function of $y$.
Secondly, we take auxiliary parameters $a,b$ (which will be fixed later) such that
 $a>0$, $b<1$ and $a<b$. Thus, we define
\begin{align*}
\varpi'(y)\equiv -2\cdot\left\{\begin{array}{cc} 1 \quad& 0\leq y<a,\\
0 \quad& a\leq y \leq b, \\
1 \quad& b<y\leq 1.
\end{array}\right.
\end{align*}
As a direct computation, the profile $\varpi$ will be then defined by
\begin{align}\label{profile}
\varpi(y)=\int_{0}^y\varpi'(\bar{y})d\bar{y}=\left\{\begin{array}{cc} -2y \quad & 0\leq y<a,\\
-2a \quad & a\leq y \leq b, \\
-2(y-b)-2a \quad & b<y\leq 1.
\end{array}\right.
\end{align}
In addition, we define an auxiliary velocity given by
\begin{align}\label{velocidad}
u(y) \equiv \int_{0}^y\varpi(\bar{y})d\bar{y}=\left\{\begin{array}{cc}-y^2 \quad& 0\leq y<a,\\
-2ay +a^2 \quad& a\leq y \leq b,\\
-(y-b)^2-2ay+a^2 \quad& b<y\leq 1.
\end{array}\right.
\end{align}

With this choice of profile $\varpi$ given by \eqref{profile}, we arrive to the equation
\begin{align}\label{equationlevel3}
\lambda f(\bx)=\overline{\Psi}_f(\bx)-\frac{1}{2\pi}\int_{\T}\overline{\Psi}_f(\bar{x},y)d \bar{x},\quad
\text{in $\T\times [-1,-b)\cup(-a,a)\cup(b,1]$},
\end{align}
where we have chosen the constant of integration so that the right-hand side of \eqref{equationlevel3} is zero mean.
We also note that this choice is compatible with the boundary condition $f\big|_{y=\pm 1}=0.$

\subsubsection{Symmetries of the equation} In this section we see that if $f(x,y)$ is a solution of \eqref{equationlevel3} then $g(x,y)\equiv -f(x,-y)$ is also a solution of
\begin{align*}
\lambda g(\bx)=\overline{\Psi}_g(\bx)-\frac{1}{2\pi}\int_{\T}\overline{\Psi}_g(\bar{x},y)d \bar{x},\quad
\text{in $\T\times [-1,-b)\cup(-a,a)\cup(b,1]$},
\end{align*}
with  $\Psi_g(x,y)\equiv -\Psi_f(x,-y)$.

The proof is reduced to proving that $\Psi_g$ does indeed solve the system
\begin{equation*}
\left\{
\begin{aligned}
\Delta \Psi_g=& \omega_g\quad \text{in $\T\times [-1,1],$}\\
\Psi_g\big|_{y=\pm 1}=& \mp \frac{1}{3}.
\end{aligned}
\right.
\end{equation*}
The boundary condition follows immediately by the definition of $\Psi_g$. For the other, we note that
$$\Delta \Psi_g(x,y)=-\Delta \left(\Psi_f(x,-y)\right)=-\Delta\Psi_f (x,-y)=-\omega_f(x,-y).$$
In addition, since $(\Phi_g)_2(x,y)=-(\Phi_f)_2(x,-y)$ we have that
\[
\omega_f(x,-y)=\varpi\circ (\Phi_f^{-1})_2(x,-y)=-\varpi\circ (\Phi_g^{-1})_2(x,y)=-\omega_g(x,y).
\]
Then, combining both, we have proved our desired goal.

\subsection{The final equation}
The above  motivates us to look for odd solutions in $y$ of \eqref{equationlevel3}. This allows us to write the problem just in $\T\times [0,a)\cup (b,1]$. In fact, we get
\begin{equation}\label{ecuf1}
\lambda f(\bx) = \overline{\Psi}_f(\bx)-\frac{1}{2\pi}\int_{\T}\overline{\Psi}_f(\bar{x},y)d\bar{x} \quad \text{in $\T\times [0,a)\cup (b,1],$}
\end{equation}
together with the boundary conditions
\begin{align}
f(x,0)=&0,\label{ecuf2}\\
f(x,1)=&0,\label{ecuf3}
\end{align}
with $\overline{\Psi}_f(x,y)=\Psi_f(x,y+f(x,y))$ and
\begin{equation}\label{ecupsi1}
\left\{
\begin{aligned}
\Delta \Psi_f&= \omega_f\quad \text{in $\T\times [0,1],$}\\
\Psi_f\big|_{y=0}&= 0,\\
\Psi_f\big|_{y=1}&= - \tfrac{1}{3}.
\end{aligned}
\right.
\end{equation}

%\begin{align}\label{ecupsi1}
%\Delta \Psi_f=&\omega_f \quad  \text{in $\T\times (0,1)$,}\\
%\Psi_f(x,0)=&0,\label{ecupsi2}\\
%\Psi_f(x,1)=&-\frac{1}{3}\label{ecupsi3},
%\end{align}

From now on we will fix the parameters $a$ and $b$. We will take
\begin{align*}
a\equiv&\frac{1}{2}-\frac{\ep}{2},\\
b\equiv&\frac{1}{2}+\frac{\ep}{2},
\end{align*}
where $0<\ep\ll 1$. Thus
\begin{align}\label{omegaf}
\omega_f(x,y)=\left\{\begin{array}{cl} -2\left(\Phi_{f}^{-1}\right)_2(x,y) & x\in \T,\quad 0\leq y<a+f(x,a),\\
-2a &  x\in \T, \quad  a+f(x,a)\leq y \leq b+f(x,b),\\ -2\left(\left(\Phi_{f}^{-1}\right)_2(x,y)-b\right)-2a \quad & x\in \T, \quad b+f(x,b)<y\leq 1.\end{array}\right.
\end{align}

\subsection{From the level set equation to the Euler equation}

Let us assume that we have a solution $f\in C^{1+\alpha}$ of \eqref{ecuf1}-\eqref{ecuf3}. Since $\varpi\in W^{1,\infty}$, the vorticity $\omega_f$ given by \eqref{omegaf} also belongs to $W^{1,\infty}$. In this section we prove that such  vorticity solves Euler equations in a weak sense.

We are looking for solutions $\omega \in W^{1,\infty}(\T\times [-1,1]\times[0,T])$ satisfying
\begin{equation}\label{weakeuler}
\int_{0}^T\int_{\T\times (-1,1)} \left(\omega(x,y,t)\pa_t\eta(x,y,t)+\omega(x,y,t)\bu(x,y,t)\cdot\nabla\eta(x,y,t)\right)dxdydt=0,
\end{equation}
%\begin{align}\label{weakeuler}
%&\int_{0}^T\int_{\T\times (-1,1)} \left(\omega(x,y,t)\pa_t\eta(x,y,t)+\omega(x,y,t)\bu(x,y,t)\cdot\nabla\eta(x,y,t)\right)dxdydt=0,\\
%&\forall \eta(x,y,t)\in C^\infty_{c}\left(\T\times(-1,1)\times (0,T)\right),\nonumber
%\end{align}
for all $\eta(x,y,t)\in C^\infty_{c}\left(\T\times(-1,1)\times (0,T)\right)$, together with \eqref{stream} and \eqref{velostream}.

First of all, we show that if $\underline{\omega}\in W^{1,\infty}(\T\times[-1,1])$ satisfies
\begin{equation}\label{weaktravel}
\int_{\T\times (-1,1)} \left(\lambda\underline{\omega}(x,y)\pa_x\rho(x,y)+\underline{\omega}(x,y)\bu(x,y)\cdot\nabla\rho(x,y)\right)dxdy=0,
\end{equation}
for all $\rho(x,y)\in C_c^\infty(\T\times (-1,1))$
%\begin{align}\label{weaktravel}
%&\int_{\T\times (-1,1)} \left(\lambda\underline{\omega}(x,y)\pa_x\rho(x,y)+\underline{\omega}(x,y)\bu(x,y)\cdot\nabla\rho(x,y)\right)dxdy=0,\\
%&\forall \rho(x,y)\in C_c^\infty(\T\times (-1,1)),\nonumber
%\end{align}
and for some $\lambda\in \R$,
then $\omega(x,y,t)=\underline{\omega}(x+\lambda t,y)$ satisfies \eqref{weakeuler}. Indeed, it is just enough to show that
\begin{align*}
&\int_{0}^T\int_{\T\times (-1,1)} \underline{\omega}(x+\lambda t,y) \pa_t\eta(x,y,t)dxdydt=
\int_{0}^T\int_{\T\times (-1,1)} \underline{\omega}(x,y) \pa_t\eta(x-\lambda t,y,t)dxdydt\nonumber\\
&=\int_{0}^T\int_{\T\times (-1,1)} \underline{\omega}(x,y) \left(\frac{d}{dt} \eta(x-\lambda t, y,t)+\lambda \pa_x\eta(x-\lambda t, y ,t) \right)dxdydt\nonumber\\&=\int_{0}^T\int_{\T\times (-1,1)} \underline{\omega}(x,y) \lambda \pa_x\eta(x-\lambda t, y ,t) dxdydt.\nonumber
\end{align*}
%\myr{Mejor poner $\bu=\nabla^\perp\Delta^{-1}\omega$ en \eqref{weakeuler} y \eqref{weaktravel}. Otra opcion es definir $\underline{\bu}=\bu(x+\lambda t,y)$. Yo pondria mas detalles, del tipo tomar $\rho(x,y)=\eta(x-\lambda t,y,t)$ para un $t$ fijo.}

Now, to obtain a solution of \eqref{weaktravel}, we start by checking that if $f(\bx)\in C^{1+\alpha}$, has small norm (such that $\Phi_f^{-1}$ is well defined) and  satisfies \eqref{ecuf1} and \eqref{ecuf2}-\eqref{ecuf3}, where $\overline{\Psi_f}=\Psi_f\circ\Phi_f$ with $\Psi_f$ satisfying \eqref{ecupsi1} and with $\omega_f$ given by \eqref{omegaf}, then $\omega_f\in W^{1,\infty}(\T\times[-1,1])$ satisfies \eqref{weaktravel}.

To prove this, the first thing we need to do is to extend $f$ from $\T\times [0,a)\cup (b,1]$ to $\T\times [-1,-b)\cup (-a,a)\cup (b,1]$ and $\Psi_f$ from $\T\times [0,1]$ to $\T\times [-1,1]$ as odd functions on $y$. Thus, these extensions, which, with a little abuse of notation, we call again $f$ and $\Psi_f$ solve respectively \eqref{ecuf1} but now in $\T\times [-1,-b)\cup (-a,a)\cup (b,1]$, and \eqref{ecupsi1} in $\T\times [-1,1]$, being $\omega_f$ in $\T\times[-1,0]$ the odd extension of $\omega_f$ in $\T\times [0,1]$. In addition, we have that
\begin{align}
f(x,-1)&=0,\label{ecuf20}\\
\Psi_f(x,-1)&=\frac{1}{3}\label{ecupsi20}.
\end{align}

In order to check \eqref{weaktravel}, we begin splitting the first term as follows
\begin{align*}
&\int_{\T\times (-1,1)}\lambda \omega_f(x,y)\pa_x\rho(x,y)dxdy\\&=\int_{\T}\left(\int_{(0,a+f(x,a))}...\, dy+\int_{(a+f(x,a),b+f(x,b))}...\, dy+\int_{(b+f(x,b),1)}...\, dy +\int_{(-1,0)}...\,dy\right)dx.
\end{align*}
In order to do the manuscript more readable, we will focus only on the region $y\in(0,1)$ since the other part follows by analogous computations.

Let us start with the first term, after applying integration by parts and some direct computations we arrive to
\begin{align}\label{weaktrozo1}
\int_{\T\times (0,a+f(x,a))}\lambda \omega_f(x,y)\pa_x\rho(x,y)dydx&=-\int_{\T\times(0,a+f(x,a))}\lambda \pa_x\omega_f(x,y)\rho(x,y)dydx\\
&\quad +\int_{\T}\lambda\omega_f(x,a+f(x,a))\rho(x,a+f(x,a))(-\pa_xf(x,a))dx\nonumber\\
&=\int_{\T}\lambda (-2a)\rho(x,a+f(x,a))(-\pa_xf(x,a))dx\nonumber\\
&\quad +\int_{\T\times (0,a)}\lambda\pa_x f(x,y) \varpi'(y)\rho(x,y+f(x,y))dydx.\nonumber
\end{align}
%\begin{align*}
%&\int_{\T\times (0,a+f(x,a))}\lambda \omega_f(x,y)\pa_x\rho(x,y)dydx\\&=\int_{\T}\lambda\omega_f(x,a+f(x,a))\rho(x,a+f(x,a))(-\pa_xf(x,a))dx\\
%&-\int_{\T\times(0,a+f(x,a))}\lambda \pa_x\omega_f(x,y)\rho(x,y)dydx\\
%&=\int_{\T}\lambda (-2a)\rho(x,a+f(x,a))(-\pa_xf(x,a))dx\\
%&+\int_{\T\times (0,a)}\lambda\pa_x f(x,y) \varpi'(y)\rho(x,y+f(x,y))dydx.
%\end{align*}
Proceeding similarly, we obtain that
\begin{align}\label{weaktrozo2}
&\int_{\T\times (a+f(x,a),b+f(x,b))}\lambda\omega_f(x,y)\pa_x \rho(x,y)dydx=\int_{\T\times (a+f(x,a),b+f(x,b))}\lambda(-2a)\pa_x \rho(x,y)dydx\\
&=\int_{\T}\lambda(-2a)\rho(x,a+f(x,a))\pa_x f(x,a)dx+\int_{\T}\lambda(-2a)\rho(x,b+f(x,b))(-\pa_x f(x,b))dx.\nonumber
\end{align}
Making the same argument once again, we get
\begin{align}\label{weaktrozo3}
\int_{\T\times (b+f(x,b),1)}\lambda\omega_f(x,y)\pa_x \rho(x,y)dydx&=-\int_{\T\times (b+f(x,b),1)}\lambda\pa_x\omega_f(x,y)\rho(x,y)dydx\\
&\quad +\int_{\T}\lambda\omega_f(x,b+f(x,b))\rho(x,b+f(x,b))\pa_xf(x,b)dx\nonumber\\
&=\int_{\T}\lambda(-2a)\rho(x,b+f(x,b))\pa_xf(x,b)dx\nonumber\\
&\quad +\int_{\T\times (b,1)}\lambda\pa_x f(x,y) \varpi'(y)\rho(x,y+f(x,y))dydx.\nonumber
\end{align}
%\begin{align*}
%&\int_{\T\times (b+f(x,b),1)}\lambda\omega_f(x,y)\pa_x \rho(x,y)dydx=-\int_{\T\times (b+f(x,b),1)}\lambda\pa_x\omega_f(x,y)\rho(x,y)dydx\\
%&+\int_{\T}(-2a)\rho(x,b+f(x,b))\pa_xf(x,b)dx+\int_{\T}\omega_f(x,1)\rho(x,1)\pa_x f(x,1)dx
%\end{align*}
Thus, adding \eqref{weaktrozo1}, \eqref{weaktrozo2} and \eqref{weaktrozo3} we finally obtain
\begin{align}\label{weaktrozofinal1}
\int_{\T\times (0,1)}\lambda \omega_f(x,y)\pa_x\rho(x,y)dydx=\int_{\T\times (0,a)\cup (b,1)}\lambda \pa_xf(x,y)\varpi'(y)\rho(x,y+f(x,y))dydx.
\end{align}
Similar calculations result in the following
\begin{multline}\label{weaktrozofinal2}
\int_{\T\times (0,1)} \omega_f(x,y)\bu(x,y)\cdot\nabla\rho(x,y)dxdy\\
=-\int_{\T\times(0,a)\cup (b,1)}\bu(x,y+f(x,y))\cdot (-\pa_xf(x,y), 1) \varpi'(y) \rho(x,y+f(x,y))dydx.
\end{multline}
%\begin{align*}
%-\int_{\T\times(0,a)\cup (b,1)}\bu(x,y+f(x,y))\cdot (-\pa_xf, 1) \varpi'(y) \rho(x,y+f(x,y))dydx.\end{align*}
And therefore, combining \eqref{weaktrozofinal1} and \eqref{weaktrozofinal2} we find that
\begin{align*}
&\int_{\T\times (0,1)}\left(\omega_f(x,y)\pa_x\rho(x,y)+\omega_f(x,y)\bu(x,y)\cdot\nabla\rho(x,y)\right)dydx\\
&=\int_{\T\times(0,a)\cup (b,1)}\varpi'(y)\left(\lambda\pa_x f(x,y)-\bu(x,y+f(x,y))\cdot(-\pa_xf(x,y),1)\right)\rho(x,y+f(x,y))dydx\\
&=0.
\end{align*}

\section{Bifurcation theory and Crandall-Rabinowitz
}\label{s:bifurcation}
The resolution of a nonlinear equation of the type \eqref{ecuf1}  can be studied through bifurcation theory. Here we will apply
the well-known Crandall-Rabinowitz theorem whose proof can be found in \cite{CR}.

\begin{theorem}\label{th:CR} Let $X, Y$ be two Banach spaces, $V$ a neighborhood of $0$ in $X$ and let
$
F : V\times\R   \to Y
$
with the following  properties:
\begin{enumerate}
\item $F [0,\lambda] = 0$ for any $\lambda\in \R$.
\item The partial derivatives $D_\l F$, $D_{f}F$ and $D^2_{\l,f}F$ exist and are continuous.
\item There exists $\lambda^\ast$ such that if $\mathcal{L}^{^\ast}\equiv D_{f} F[0,\l^\ast]$ then  $\cN(\mathcal{L}^\ast)$ and $Y/\cR(\mathcal{L}^\ast)$ are one-dimensional.
\item {\it Transversality assumption}: $D^2_{\l,f}F[0,\l^\ast]h^\ast \not\in \cR(\mathcal{L}^\ast)$, where
$$
\cN(\mathcal{L}^\ast) = \text{span}\{h^\ast\}.
$$
\end{enumerate}
If $Z$ is any complement of $\cN(\mathcal{L}^\ast)$ in $X$, then there is a neighborhood $U$ of $(0,\l^\ast)$ in $X\times \R$, an interval $(-\sigma_0,\sigma_0)$, and continuous functions $\varphi: (-\sigma_0,\sigma_0) \to \R$, $\psi: (-\sigma_0,\sigma_0) \to Z$ such that $\varphi(0) = 0$, $\psi(0) = 0$ and
\begin{align*}
F^{-1}(0)\cap U=&\Big\{\big(\lambda^\ast+\varphi(\sigma), \sigma h^\ast+\sigma \psi(\sigma)\big)\,;\,\vert \sigma\vert<\sigma_0\Big\} \cup \Big\{(\lambda,0)\,;\, (\lambda,0)\in U\Big\}.
\end{align*}

\end{theorem}

The bulk of the work consists of proving all the assumptions of Theorem \ref{th:CR}. This will be done in detail in the
the following sections.

\section{Functional setting and regularity}\label{s:spaces}

To apply the C-R theorem we first need to fix the function spaces. We have to look for Banach spaces X and Y such that $F : X\times \R \rightarrow Y$ is well defined and satisfies the required assumptions.

First, we introduce the H\"{o}lder spaces

\begin{align*}
C^{1+\alpha}_{0}\equiv\left\{ f\in C^{1+\alpha}\left(\T\times [0,a)\cup (b,1]\right)\,:\, f(x,0)=f(x,1)=0,\, \int_{\T}f(\bar{x},y)d\bar{x}=0\right\},
\end{align*}
and
\begin{align*}
\S\equiv\left\{ f\in C^{1+\alpha}_0\,:\, f(x,y)=f(-x,y)\right\}.
\end{align*}

In the rest of the paper we will consider that $1/2\leq \alpha <1$. This fact will not be important until section \ref{s:mainthm} where it will be used.

\subsection{Analysis of the operator $\overline{\Psi}$}
After introduce the above Banach spaces, we are going to check that functional
\begin{equation}\label{auxFuncbarpsi}
\begin{array}{>{\displaystyle}r @{} >{{}}c<{{}} @{} >{\displaystyle}l}
          \overline{\Psi}_{(\cdot)}\colon B_{\ep/4}\subset   C^{1+\alpha}(\T\times[0,a)\cup(b,1]) &\rightarrow& C^{1+\alpha}(\T\times[0,a)\cup(b,1]) \\
          f &\mapsto& \overline{\Psi}_{f}
         \end{array}
\end{equation}
%\begin{align}
%\overline{\Psi}_{(\cdot)}\, : B_{\ep/4}\subset  & C^{1+\alpha}\to C^{1+\alpha}\\
%&\quad\quad f \to \overline{\Psi}_f,\nonumber
%\end{align}
is well-defined and it is  continuous from $B_{\ep/4}$ to $C^{1+\alpha}(\T\times[0,a)\cup(b,1])$, where
$$B_{\ep/4}\equiv\{ f\in \S\,:\, ||f||_{C^{1+\alpha}}<\ep/4\}.$$

To check this, we first note that,  for $f\in B_{\ep/4}$ taking $\ep$ small enough we directly have that
\[
a+f(x,a)<b+f(x,b), \qquad \forall x\in\T,
\]
and
\[
1+\pa_yf(x,y)>0, \qquad \forall (x,y)\in\T\times[0,a)\cup(b,1].
\]
%\begin{align*}a+f(x,a)&<b+f(x,b), \qquad \forall x\in\T,\\
%1+\pa_yf(x,y)&>0, \qquad \forall (x,y)\in\T\times[0,a)\cup(b,1].
%\end{align*}
In addition, we have
\begin{align*}
&\Phi_f\,:\, \T\times [0,a)\to \Omega_{\text{in}},\\
&\Phi_f\,:\, \T\times (b,1]\to \Omega_{\text{out}},
\end{align*}
where the sets
\begin{align*}
\Omega_{\text{in}}\equiv&\{(x,y)\in \T\times [0,1] \,: \, 0\leq y< a+f(x,a)\},\\
\Omega_{\text{out}}\equiv&\{(x,y)\in \T\times [0,1] \,: \,  b+f(x,b)<y\leq 1\},
\end{align*}
satisfy $$\Omega_{\text{in}}\cap\Omega_{\text{out}}=\emptyset.$$

Moreover, there exist a well-defined inverse  $\Phi_f^{-1}\in C^{1+\alpha}(\T\times[0,a)\cup (b,1])$ that maps $\Omega_{\text{in}}$ into $\T\times [0,a)$ and $\Omega_{\text{out}}$ into $\T\times(b,1]$. So, as an immediate consequence, we can define $\omega_f$ through \eqref{omegaf}.
More specifically, since $\varpi\in W^{1,\infty}([0,1])$ we directly have that $\omega_f\in W^{1,\infty}(\T\times [0,1])$. Then, in particular, we obtain $\omega_f\in C^\beta(\T\times [0,1])$ for all $0<\beta<1$. Now, applying elliptic regularity theory into \eqref{ecupsi1} we obtain $\Psi_f\in C^{2+\beta}(\T\times [0,1])$ and finally $\overline{\Psi}_f=\Psi_f\circ\Phi_f\in C^{1+\alpha}(\T\times [0,a)\cup (b,1]).$
%We get that, since  $\varpi\in W^{1,\infty}((0,1))$, $\omega_f\in W^{1,\infty}(\T\times (0,1))$. In particular, $\omega_f\in C^\beta(\T\times [0,1])$, for all $0<\beta<1$. Then $\Psi_f\in C^{2+\beta}(\T\times (0,1))$ and $\overline{\Psi}_f\in C^{1+\alpha}(\T\times (0,a)\cup (b,1))$.

The continuity of $\overline{\Psi}_f$ will be studied in the next section where we will even proof that it is differentiable.

\subsection{The continuity of the functional}
After introducing and studying the functional \eqref{auxFuncbarpsi}, we are now ready to analyze
\[
 \begin{array}{>{\displaystyle}r @{} >{{}}c<{{}} @{} >{\displaystyle}l}
\tilde{F}_{(\cdot)}\equiv \overline{\Psi}_{(\cdot)}-\frac{1}{2\pi}\int_{\T}\overline{\Psi}_{(\cdot)}(\bar{x},y)d\bar{x}\colon B_{\ep/4} &\rightarrow& \S \\
          f &\mapsto& \overline{\Psi}_f-\frac{1}{2\pi}\int_{\T}\overline{\Psi}_f(\bar{x},y)d\bar{x}
         \end{array}
\]
The fact that $\tilde{F}_{(\cdot)}$ is well-defined from $B_{\ep/4}$ to $C^{1+\a}(\T\times[0,a)\cup(b,1])$ follows straightforward from the above. To conclude that, in fact, is well-defined from $\S$ to itself we  need to check
\begin{enumerate}

    \item $\int_\T \tilde{F}_{(\cdot)}(\bar{x},y)d\bar{x}=0,$

    \item $\tilde{F}_{(\cdot)}(x,0)=\tilde{F}_{(\cdot)}(x,1)=0$,

    \item $\tilde{F}_{(\cdot)}(x,y)=\tilde{F}_{(\cdot)}(-x,y).$
\end{enumerate}
%In addition we will have that
%\begin{align*}
%\tilde{F}_{(\cdot)}\equiv \overline{\Psi}_{(\cdot)}-\frac{1}{2\pi}\int_{\T}\overline{\Psi}_{(\cdot)}(\bar{x},y)d\bar{x}\, : \,& \S \to \S\\
%& f \to \overline{\Psi}_f-\frac{1}{2\pi}\int_{\T}\overline{\Psi}_f(\bar{x},y)d\bar{x}.
%\end{align*}
That the mean of $\tilde{F}_f$ is equal to zero follows by its definition. In addition, $\tilde{F}_f(x,0)=0$ since $\overline{\Psi}_f(x,0)=\Psi_f(x,0)=0$. Also, $\tilde{F}_f(x,1)=0$ since $\overline{\Psi}_f(x,1)=\Psi_f(x,-1)=-1/3$ and then $$\frac{1}{2\pi}\int_{\T}\overline{\Psi}(\bar{x},1)d\bar{x}=-\frac{1}{3}.$$ Finally, we notice that, if $f(x,y)=f(-x,y)$ then $\omega_f(x,y)=\omega_f(-x,y)$. Thus $\Psi_f(x,y)=\Psi_f(-x,y)$ and consequently $\overline{\Psi}_f(x,y)=\overline{\Psi}_f(-x,y)$.\\

In summary, we have proved the following Lemmas:

\begin{Lemma}\label{regpsi} Let $f\in B_{\ep/4}$ and $\Psi_f$ solving \eqref{ecupsi1}, then $\Psi_f\in C^{2+\beta}(\T\times [-1,1])$ for all $0<\beta<1$.
\end{Lemma}

\begin{Lemma} Let $F[f,\lambda](x,y)\equiv \lambda f(x,y)-\tilde{F}_f(x,y).$ That is,
\begin{align*}
F[f,\lambda](x,y) = \lambda f(x,y) -\overline{\Psi}_f(x,y)+\frac{1}{2\pi}\int_{\T}\overline{\Psi}_f(\bar{x},y)d\bar{x}.
\end{align*}
Then
\begin{align*}
F\,:\, B_{\ep/4}\times \R\to \S
\end{align*}
is well-defined and it is continuous from  $B_{\ep/4}$ to $\S$.

\end{Lemma}

\subsection{The differential of the functional}

Now, we compute the differential of $\overline{\Psi}_f$ with respect to $f$ at the point $\eta\in B_{\ep/4}$.
We start computing $D_f\overline{\Psi}_\eta h$ for $\eta\in B_{\ep/4}$ and $h\in \S$. That is,
%To start, we compute for $\eta\in B_{\frac{\ep}{4}}$ and $h\in \S$,
\begin{align}\label{dpsibarra}
D_f\overline{\Psi}_\eta h(x,y) =h(x,y)\pa_y \Psi(x,y+\eta(x,y))+D_f\Psi_\eta h(x,y+\eta(x,y)).
\end{align}
From the Lemma \ref{regpsi} we know that $\pa_y\Psi(x,y)\in C^{1+\beta}(\T\times [0,1])$ for all $0<\beta<1$, and then $\pa_y \Psi(x,y+\eta(x,y))\in C^{1+\alpha}(\T\times [0,a)\cup (b,1])$. Then, since the first term of the above expression  is just a derivative in the classical sense, we only have to prove that there exists the differential $D_f\Psi_\eta h$, which will be presented in the following Lemma.

\begin{Lemma}\label{l:gateaux} Let $\eta\in B_{\ep/4}$ and $h\in \S$ and consider
\begin{align*}
\Psi_{(\cdot)}\, : \, B_{\ep/4}\subset\S \to C^{1+\alpha}(\T\times [0,1]).
\end{align*}
Then, the Gateaux derivative of $\Psi_f$ with respect to $f$ at the point $\eta$ in the direction $h$, i.e. $D_f\Psi_\eta h$, is given by solving the problem
%\begin{align*}
%\Delta D_f\Psi_\eta h = &d\omega_{\eta}\quad \text{ in $\T\times [0,1]$}\\
%D_f\Psi_\eta h (x, 1)=D_f\Psi_\eta h (x, 0)=&0,
%\end{align*}
\begin{equation}\label{eq:DfPsieta}
\left\{
\begin{aligned}
\Delta D_f\Psi_\eta h = &d\omega_{\eta},\quad \text{ in $\T\times [0,1]$},\\
D_f\Psi_\eta h\big|_{y=0}=&0,\\
D_f\Psi_\eta h \big|_{y=1}=&0,
\end{aligned}
\right.
\end{equation}
where
\begin{align*}
d\omega_\eta=\left\{\begin{array}{cl} -\left(\frac{\varpi'h}{1+\pa_y \eta}\right)\circ \Phi_\eta^{-1} &\qquad   0\leq y<a+\eta(x,a),\\
0 &\qquad  a+\eta(x,a)\leq y\leq b+\eta(x,b),\\
 -\left(\frac{\varpi'h}{1+\pa_y \eta}\right)\circ \Phi_\eta^{-1} &\qquad  b+\eta(x,b)<y\leq 1.\end{array}\right.
\end{align*}
\end{Lemma}
\begin{proof}
First of all, we define $\varphi$ through the problem
\begin{equation*}
\left\{
\begin{aligned}
\Delta \varphi = &d\omega_{\eta},\quad \text{ in $\T\times [0,1]$},\\
\varphi\big|_{y=0}=&0,\\
\varphi \big|_{y=1}=&0,
\end{aligned}
\right.
\end{equation*}
%\begin{align*}
%\Delta \varphi =& d\omega_\eta, \quad \text{ in $\T\times (0,1)$}\\
%\varphi(x,1)=&0,\\
%\varphi(x,0)=&0.
%\end{align*}
To conclude that $\varphi= D_f\Psi_\eta$, we have to show that
\begin{align}\label{limPsi}
\lim_{t\to 0} \left|\left|\frac{\Psi_{\eta+th}-\Psi_{\eta}}{t}-\varphi \right|\right|_{C^{1+\alpha}(\T\times[0,1])}=0.
\end{align}
Before starting with the main part of the proof, we note that taking $t$ small enough we get
\[
\frac{\Psi_{\eta+th}-\Psi_{\eta}}{t}-\varphi\in C^{1+\a}(\T\times[0,1]),
\]
This is an immediate consequence of the fact that by definition $d\omega_\eta\in L^p(\T\times [0,1])$ for all $1\leq p \leq \infty.$ Then, by elliptic regularity theory $\varphi\in W^{2,p}(\T\times [0,1])$ for $1<p<\infty$ and by Morrey's inequality we get $\varphi\in C^{1+\beta}(\T\times [0,1])$ for any $0<\beta<1$.

Now, we are ready to prove \eqref{limPsi}.
By definition we have that
\begin{equation*}
\left\{
\begin{aligned}
\Delta \left(\frac{\Psi_{\eta+th}-\Psi_{\eta}}{t}-\varphi\right)=&\frac{\omega_{\eta+th}-\omega_{\eta}}{t}-d\omega_\eta, \quad \text{ in $\T\times [0,1]$},\\
\left.\frac{\Psi_{\eta+th}-\Psi_{\eta}}{t}-\varphi\right|_{y=0,\,1}=&0,
\end{aligned}
\right.
\end{equation*}
and by elliptic regularity theory we get
\begin{align*}
\left|\left|\frac{\Psi_{\eta+th}-\Psi_{\eta}}{t}-\varphi \right|\right|_{C^{1+\alpha}}\leq \left|\left|\frac{\Psi_{\eta+th}-\Psi_{\eta}}{t}-\varphi \right|\right|_{W^{2,p}}\leq
\left|\left|\frac{\omega_{\eta+th}-\omega_{\eta}}{t}-d\omega_\eta \right|\right|_{L^p},
\end{align*}
where in the first step we have used Morrey's inequality
with $p=\frac{2}{1-\alpha}$. This means that \eqref{limPsi} reduces to proving that
\begin{align}\label{limomega}
\lim_{t\to 0}\left|\left|\frac{\omega_{\eta+th}-\omega_{\eta}}{t}-d\omega_\eta\right|\right|_{L^p(\T\times [0,1])}=0, \qquad \text{for }  2<p<\infty.
\end{align}

Let us compute the difference $\omega_{\eta+th}-\omega_{\eta}$. To do that we define the domains (for $t$ small enough)
\begin{align*}
\Omega_1=&\left\{x\in \T,\,\, 0\leq y<\text{min} (a+\eta(x,a), a+\eta(x,a)+th(x,a)) \right\},\\
\nonumber\Omega_2=&\left\{x\in \T,\,\, \text{min} (a+\eta(x,a), a+\eta(x,a)+th(x,a))\leq y<\text{max} (a+\eta(x,a), a+\eta(x,a)+th(x,a)) \right\},\\
\nonumber\Omega_3=&\left\{x\in \T,\,\, \text{max} (a+\eta(x,a), a+\eta(x,a)+th(x,a))\leq y\leq \text{min} (b+\eta(x,b), b+\eta(x,b)+th(x,b)) \right\},\\
\nonumber\Omega_4=&\left\{x\in \T,\,\, \text{min} (b+\eta(x,b), b+\eta(x,b)+th(x,b))<y\leq\text{max} (b+\eta(x,b), b+\eta(x,b)+th(x,b)) \right\},\\
\nonumber\Omega_5=&\left\{x\in \T,\,\, \text{max} (b+\eta(x,b), b+\eta(x,b)+th(x,b))<y\leq 1\right\}.
\end{align*}
such that
\begin{equation}\label{omega1}
\bigcup_{i=1}^5 \Omega_i=\T\times[0,1].
\end{equation}
Thus, from \eqref{omegaf},
\begin{align*}
\omega_{\eta+th}-\omega_\eta=\left\{\begin{array}{cc}-2\left(\Phi_{\eta+th}^{-1}\right)_2+2\left(\Phi_{\eta}^{-1}\right)_2  &\qquad (x,y)\in \Omega_1,\\ \text{ either} -2\left(\Phi^{-1}_{\eta+t h}\right)_2+2a \text{ or } -2a +2 \left(\Phi^{-1}_{\eta}\right)_2 &\qquad (x,y)\in \Omega_2, \\ 0 &\qquad (x,y)\in \Omega_3,\\
\text{ either} -2\left(\left(\Phi^{-1}_{\eta+t h}\right)_2-b\right) \text{ or } 2 \left(\left(\Phi^{-1}_{\eta}\right)_2-b\right) &\qquad (x,y)\in \Omega_4,\\
-2\left(\Phi^{-1}_{\eta+t h}\right)_2+2\left(\Phi^{-1}_{\eta}\right)_2 &\qquad (x,y)\in \Omega_5.\end{array}\right.
\end{align*}

In $\Omega_1$ we have that
\begin{align*}
\frac{\omega_{\eta+t h}-\omega_{\eta}}{t}-d\omega_{\eta} = -2\frac{\left(\Phi_{\eta+th}^{-1}\right)_2-\left(\Phi_{\eta}^{-1}\right)_2}{t}-\frac{2 h}{1+\pa_y\eta}\circ \Phi_\eta^{-1}.
\end{align*}
Moreover, the first factor of the above expression can be written in a more convenient way as
\begin{align*}\left(\Phi_{\eta+th}^{-1}\right)_2-\left(\Phi_{\eta}^{-1}\right)_2=\int_{0}^1\frac{d\left(\Phi^{-1}_{\eta+sth}\right)_2}{ds}ds.
\end{align*}
In addition, from
$$\left(\Phi_{\eta+sth}^{-1}\right)_2\circ\Phi_{\eta+sth}=y,$$
we have that
\begin{align*}
\frac{d\left(\Phi_{\eta+sth}^{-1}\right)_2}{ds}\circ\Phi_{\eta+sth} + \frac{d\Phi_{\eta+sth}}{ds}\cdot \left(\nabla \left(\Phi^{-1}_{\eta+sth}\right)_2\right)\circ \Phi_{\eta+sth}=0.
\end{align*}
Then
\begin{align*}
\frac{d\left(\Phi_{\eta+sth}^{-1}\right)_2}{ds}&=
-th\circ\Phi^{-1}_{\eta+tsh}\left(\pa_y\left(\Phi_{\eta+sth}^{-1}\right)_2\right)\\
&=-t\frac{ h }{1+\pa_y\eta+st \pa_y h }\circ\Phi^{-1}_{\eta+sth}.
\end{align*}
We obtain that
\begin{align*}
\frac{\omega_{\eta+th}-\omega_{\eta}}{t}-d\omega_\eta=2\int_{0}^1 \left(\frac{h}{1+\pa_y\eta+st\pa_yh }\circ \Phi^{-1}_{\eta+st h}-\frac{h}{1+\pa_y\eta}\circ\Phi^{-1}_{\eta}\right)ds.
\end{align*}
In order to estimate this last term, we split it in the following way
\begin{align*}
&\frac{h}{1+\pa_y\eta+st\pa_yh }\circ \Phi^{-1}_{\eta+st h}-\frac{h}{1+\pa_y\eta}\circ\Phi^{-1}_{\eta}=
\frac{h\circ \Phi^{-1}_{\eta+st h} - h \circ\Phi^{-1}_\eta}{1+(\pa_y\eta+st\pa_yh)\circ\Phi^{-1}_{\eta+sth}}
\\
&+ h\circ\Phi^{-1}_{\eta}\left(\frac{1}{1+\pa_y\eta+st\pa_yh }\circ \Phi^{-1}_{\eta+st h}-\frac{1}{1+\pa_y\eta}\circ\Phi^{-1}_{\eta}\right)\\
& \equiv I_1 + I_2.
\end{align*}
We can bound
\begin{align*}
|I_1|&\leq  C |\Phi^{-1}_{\eta+sth}-\Phi^{-1}_{\eta}|,\\
|I_2|&\leq  C \left(|\Phi^{-1}_{\eta+sth}-\Phi^{-1}_{\eta}|^\alpha+ t\right).
\end{align*}
\begin{remark}
In the rest the value of the constant $C$ can change from line to line and depends on $\ep$.
\end{remark}
The calculation for $I_1$ is trivially satisfied. For $I_2$ we only have to add and subtract the appropriate terms to arrive at a more convenient expression as follows
\begin{align*}
&\frac{1}{1+\pa_y\eta+st\pa_yh }\circ \Phi^{-1}_{\eta+st h}-\frac{1}{1+\pa_y\eta}\circ\Phi^{-1}_{\eta} \pm \frac{1}{1+\pa_y\eta}\circ \Phi^{-1}_{\eta+st h}&\\
&=\frac{1}{1+\pa_y\eta}\circ\Phi^{-1}_{\eta+st h}-\frac{1}{1+\pa_y\eta}\circ\Phi^{-1}_{\eta}\\
&-t\frac{s\pa_y h}{(1+\pa_y\eta+st\pa_yh)(1+\pa_y\eta)}\circ \Phi^{-1}_{\eta+st h}.
\end{align*}
After that, since $|\Phi_{\eta+sth}^{-1}-\Phi_{\eta}^{-1}|=|(\Phi_{\eta+sth}^{-1})_2-(\Phi^{-1}_{\eta})_2|\leq Ct$ we can conclude that
\begin{align*}
\left|\left|\frac{\omega_{\eta+t h}-\omega_{\eta}}{t}-d\omega_\eta \right|\right|_{L^\infty(\Omega_1)}\leq C t^\alpha.
\end{align*}

In $\Omega_2$ we have to distinguish between the following two cases:\\
\underline{Case $\eta(x,a)<\eta(x,a)+t h (x,a)$:} Note that in this case we have
\begin{align*}
\frac{\omega_{\eta+t h}-\omega_{\eta}}{t}-d\omega_\eta=\frac{-2\left(\Phi^{-1}_{\eta+th}\right)_2+2a}{t}.
\end{align*}
Since $a=\left(\Phi^{-1}_{\eta+th}\right)_2(x,a+\eta(x,a)+th(x,a))$ we get
\begin{align*}
\left|a-\left(\Phi_{\eta+th}^{-1}\right)_2(x,y)\right|\leq C |a+\eta(x,a)+th(x,a)-y|.
\end{align*}
Then, we find that
\begin{align*}
\left|\left| \frac{\omega_{\eta+t h}-\omega_{\eta}}{t}-d\omega_\eta\right|\right|_{L^p}^p =&C t^{-p}\int_{\T}\int_{a+\eta(x,y)}^{a+\eta(x,a)+th(x,a)}|a+\eta(x,a)+th(x,a)-y|^p dy dx\\
&\leq \frac{C}{p+1}t.
\end{align*}

\noindent
\underline{Case $\eta(x,a)+th(x,a)<\eta(x,a)$:} In this case we have
\begin{align*}
\frac{\omega_{\eta+t h}-\omega_{\eta}}{t}-d\omega_\eta=\frac{-2a+2\left(\Phi^{-1}_{\eta}\right)_2}{t} +\frac{\varpi' h}{1+\pa_y\eta}\circ \Phi^{-1}_\eta.
\end{align*}
With the first term $\frac{-2a+2\left(\Phi^{-1}_{\eta}\right)_2}{t}$ we proceed as before, but using that $a=(\Phi^{-1}_\eta)_2(x,a+\eta(x,a))$. The term $\frac{\varpi'h}{1+\pa_y\eta}\circ{\Phi}_\eta^{-1}$ can be bounded in $L^\infty$ such that
\begin{align*}
\left|\left| \frac{\varpi'}{1+\pa_y\eta}\circ{\Phi}_\eta^{-1} \right|\right|_{L^p}^p\leq C \int _{\T}\int^{a+\eta(x,a)}_{a+\eta(x,a)+t h(x,a)}dy dx \leq C t.
\end{align*}
Combining both cases, we can conclude that
\begin{align*}
\left|\left|\frac{\omega_{\eta+t h}-\omega_{\eta}}{t}-d\omega_\eta \right|\right|_{L^p(\Omega_2)}\leq C t^{\frac{1}{p}}, \qquad \text{for } 1<p<\infty.
\end{align*}

In $\Omega_3$ there is nothing to do. In $\Omega_4$ we proceed as we did with $\Omega_2$ and in $\Omega_5$ as we did with $\Omega_1$. Joining all these cases, we finally achieve that
\begin{align*}
\left|\left|\frac{\omega_{\eta+th}-\omega_{\eta}}{t}-d\omega_\eta \right|\right|_{L^p(\T\times[0,1])}\leq C \left(t^\alpha+t^\frac{1}{p}\right), \qquad \text{for } 1<p<\infty.
\end{align*}
Finally, taking the limit of the previous expression, we can conclude that \eqref{limomega} is satisfied.
\end{proof}

Once we have seen the existence of the Gateaux derivative, we show below that it is continuous.
%We next prove that $D_f\Psi_\eta h $ is continuous.

\begin{Lemma} \label{continuitydPsi}
Let $\eta_1,\eta_2\in B_{\ep/4}$ and $h\in \S$. Then, the following estimate holds
\begin{align*}
&||D_f\Psi_{\eta_1}h-D_f\Psi_{\eta_2}h||_{C^{1+\alpha}(\T\times [0,1])}\\ &\leq C||h||_{C^{1+\alpha}(\T\times [0,a)\cup (b,1])}\left(||\eta_1-\eta_2||^\frac{1}{p}_{C^{1+\alpha}(\T\times [0,a)\cup (b,1])}+||\eta_1-\eta_2||^\alpha_{C^{1+\alpha}(\T\times [0,a)\cup (b,1])}\right),
\end{align*}
for any  $p\in(2,\infty).$
\end{Lemma}
\begin{proof}
Proceeding as before, using Morrey's inequality and elliptic regularity theory, we have that
\begin{align}\label{difftodoeta12}
||D_f\Psi_{\eta_1}h-D_f\Psi_{\eta_2}h||_{C^{1+\alpha}(\T\times [0,1])}&\leq||D_f\Psi_{\eta_1}h-D_f\Psi_{\eta_2}h||_{W^{2,p}(\T\times [0,1])}\\
&\leq ||d\omega_{\eta_1}-d\omega_{\eta_2}||_{L^p(\T\times [0,1])}.\nonumber
\end{align}
Let $\Omega_i$ be the analogue of the domains defined in \eqref{omega1} but now for $\eta_1$ and $\eta_2$.  We obtain
\begin{align*}
d\omega_{\eta_1}-d\omega_{\eta_2}=\left\{\begin{array}{cc}\frac{2h}{1+\pa_y\eta_1}\circ \Phi^{-1}_{\eta_1}-\frac{2h}{1+\pa_y\eta_2}\circ \Phi^{-1}_{\eta_2} &\qquad (x,y)\in \Omega_1,\\
\text{ either } \frac{2h}{1+\pa_y\eta_1}\circ \Phi^{-1}_{\eta_1}\text{ or } -\frac{2h}{1+\pa_y\eta_2}\circ \Phi^{-1}_{\eta_2} &\qquad (x,y)\in \Omega_2,\\
0 &\qquad (x,y)\in \Omega_3, \\
\text{ either } \frac{2h}{1+\pa_y\eta_1}\circ \Phi^{-1}_{\eta_1}\text{ or } -\frac{2h}{1+\pa_y\eta_2}\circ \Phi^{-1}_{\eta_2}&\qquad (x,y)\in \Omega_4,\\
\frac{2h}{1+\pa_y\eta_1}\circ \Phi^{-1}_{\eta_1}-\frac{2h}{1+\pa_y\eta_2}\circ \Phi^{-1}_{\eta_2} & \qquad (x,y)\in \Omega_5.
\end{array}\right.
\end{align*}
In $\Omega_2$ we proceed as follow. In the set $\tilde{I}=\{ x\in \T\,:\, \eta_1(x,a)<\eta_2(x,a)\}$ we have that
\begin{align*}
\int_{\tilde{I}}\int_{a+\eta_1(x,a)}^{a+\eta_2(x,a)}\left|\frac{2h}{1+\pa_y\eta_2}\circ \Phi_{\eta_2}^{-1}\right|^p dydx
&\leq C || \eta_2(\cdot,a)-\eta_1(\cdot,a)||_{L^\infty(\tilde{I})}\\
&\leq C ||\eta_2-\eta_1||_{C^{1+\alpha}(\T\times[0,a)\cup (b,1])}.
\end{align*}
In the set $\T\setminus \tilde{I}$ we can produce a similar inequality. In $\Omega_4$ the same argument yields
\begin{equation}\label{diffOme24}
||d\omega_{\eta_1}-d\omega_{\eta_2}||_{L^p (\T\times \Omega_2\cup \Omega_4)}\leq C||h||_{L^\infty(\T\times [0,a)\cup(b,1])}||\eta_1-\eta_2||^{\frac{1}{p}}_{C^{1+\alpha}(\T\times [0,a)\cup (b,1])}.
\end{equation}
Since $d\omega_{\eta_1}-d\omega_{\eta_2}=0$ in $\Omega_3$ we just have to deal with this difference in $\Omega_1$ and $\Omega_5$. The estimate follow the same steps in both domains. In $\Omega_1$,
\begin{align*}
\frac{d\omega_{\eta_1}-d\omega_{\eta_2}}{2}=&\frac{h}{1+\pa_y\eta_1}\circ \Phi^{-1}_{\eta_1}-\frac{h}{1+\pa_y\eta_2}\circ \Phi^{-1}_{\eta_2}
\\=& \frac{h\circ \Phi^{-1}_{\eta_1}-h\circ \Phi^{-1}_{\eta_2}}{1+\pa_y\eta_1\circ \Phi^{-1}_{\eta_1}}+ h\circ \Phi_{\eta_2}^{-1}\left(\frac{1}{1+\pa_y\eta_1\circ\Phi_{\eta_1}^{-1}}-\frac{1}{1+\pa_y\eta_2\circ\Phi^{-1}_{\eta_2}}\right)\\
\equiv& J_1+J_2.
\end{align*}
Let us start with $J_1$. We can directly  estimate
\begin{align*}
&\left|\frac{h\circ \Phi^{-1}_{\eta_1}-h\circ \Phi^{-1}_{\eta_2}}{1+\pa_y\eta_1\circ \Phi^{-1}_{\eta_1}}\right|\leq
 C ||h||_{C^1}|\left(\Phi^{-1}_{\eta_1}\right)_2-\left(\Phi^{-1}_{\eta_2}\right)_2|.
\end{align*}
In addition, we can write
\begin{equation}\label{aux:diffeta12}
\left(\Phi^{-1}_{\eta_1}\right)_2-\left(\Phi^{-1}_{\eta_2}\right)_2=\int_{0}^1 \frac{d}{ds}\left(\Phi^{-1}_{s\eta_1+(1-s)\eta_2}\right)_2ds.
\end{equation}
Since $$\left(\Phi^{-1}_{s\eta_1+(1-s)\eta_2}\right)_2\circ \Phi_{s\eta_1+(1-s)\eta_2}=y,$$
we have that
\begin{align*}
\frac{d\left(\Phi^{-1}_{s\eta_1+(1-s)\eta_2}\right)_2}{ds}\circ \Phi_{s\eta_1+(1-s)\eta_2}+\frac{d\Phi_{s\eta_1+(1-s)\eta_2}}{ds}\cdot \nabla \left(\Phi^{-1}_{s\eta_1+(1-s)\eta_2}\right)_2\circ \Phi_{s\eta_1+(1-s)\eta_2}=0.
\end{align*}
Then
\begin{align*}
\frac{d\left(\Phi^{-1}_{s\eta_1+(1-s)\eta_2}\right)_2}{ds}&=-(\eta_1-\eta_2)\circ \Phi^{-1}_{s\eta_1+(1-s)\eta_2} \left(\pa_y\left(\Phi^{-1}_{s\eta_1+(1-s)\eta_2}\right)_2\right)\\
&=-\frac{(\eta_1-\eta_2)}{1+s\pa_y\eta_1+(1-s)\pa_y\eta_2}\circ\Phi^{-1}_{s\eta_1+(1-s)\eta_2}.
\end{align*}
Therefore, returning to \eqref{aux:diffeta12} we get
\begin{align*}
\left| \left(\Phi^{-1}_{\eta_1}\right)_2- \left(\Phi^{-1}_{\eta_2}\right)_2\right|\leq C ||\eta_1-\eta_2||_{C^{1+\alpha}(\T\times [0,a)\cup(b,1])},
\end{align*}
and consequently we obtain that
\[
|J_1|\leq C ||\eta_1-\eta_2||_{C^{1+\alpha}(\T\times [0,a)\cup(b,1])}.
\]
Let us continue with $J_2$. We note that
\begin{align*}
|J_2|=&\left|h\circ\Phi_{\eta_2}^{-1}\left(\frac{1}{1+\pa_y\eta_1\circ\Phi_{\eta_1}^{-1}}
-\frac{1}{1+\pa_y\eta_2\circ\Phi^{-1}_{\eta_2}}\right)\right|
\\&\leq C \left|\pa_y \eta_1\circ\Phi^{-1}_{\eta_1}-\pa_y\eta_2\circ\Phi^{-1}_{\eta_2}\right|\\
&\leq C\left|\pa_y\eta_1\circ \Phi^{-1}_{\eta_1}-\pa_y \eta_1\circ \Phi^{-1}_{\eta_2}\right|+
C\left|\pa_y \eta_{1}\circ\Phi^{-1}_{\eta_2}-\pa_y\eta_2\circ \Phi^{-1}_{\eta_2}\right|\\
&\leq C ||\eta_1||_{C^{1+\alpha}(\T\times [0,a)\cup (b,1])}|\Phi_{\eta_1}^{-1}-\Phi^{-1}_{\eta_2}|^\alpha
+C||\eta_1-\eta_2||_{C^{1+\alpha}(\T\times [0,a)\cup (b,1])}.
\end{align*}
Since $|\Phi_{\eta_1}^{-1}-\Phi^{-1}_{\eta_2}|=|\left(\Phi_{\eta_1}^{-1}\right)_2-\left(\Phi^{-1}_{\eta_2}\right)_2| \leq C ||\eta_1-\eta_2||_{C^{1+\alpha}(\T\times [0,a)\cup(b,1])}$, the above argument yields
\begin{equation}\label{diffOme15}
||d\omega_{\eta_1}-d\omega_{\eta_2}||_{L^p (\T\times \Omega_1\cup \Omega_5)}\leq C||h||_{C^1(\T\times [0,a)\cup(b,1])}||\eta_1-\eta_2||^{\alpha}_{C^{1+\alpha}(\T\times [0,a)\cup (b,1])}.
\end{equation}
Finally, combining \eqref{difftodoeta12} together with  \eqref{diffOme24}-\eqref{diffOme15}, we can conclude our goal.
\end{proof}

Now, putting together Lemma \ref{l:gateaux} and Lemma \ref{continuitydPsi} we obtain the following result.
\begin{corollary}The functional
\begin{align*}
\Psi_{(\cdot)} \, : \, B_{\ep/4}\subset \S\to C^{1+\alpha}(\T\times[0,1])
\end{align*}
is Frechet differentiable and its derivative is continuous.
\end{corollary}

It is remain to prove that the differential of $\overline{\Psi}$ is continuous. Until now we have proven that this differential is given by \eqref{dpsibarra}. In order to prove that it is continuous we will need the following auxiliary Lemma.
\begin{Lemma} \label{aux}Let $f\in L^p(\T\times[0,1])$.  For all  $\tilde{\ep}>0$ there exists $\delta>0$ (that depends on $f$) such that if $$||\eta_1-\eta_2||_{C^{1+\alpha}(\T\times [0,a)\cup (b,1])}\leq \delta,$$ then
$$||f\circ \Phi_{\eta_1}-f\circ \Phi_{\eta_2}||_{L^p(\T\times [0,a)\cup (b,1])}\leq \tilde{\ep}.$$
\end{Lemma}
\begin{proof}
We can approximate $f$ by  piecewise constant functions $f_N=\sum_{n=1}^N c_n \chi_{I_n}$, where $\{I_n\}_{n=1}^N$ is a collection of rectangles that cover $\T\times [0,a)\cup (b,1]$, in such a way that
\begin{align*}
\lim_{N\to \infty} ||f-f_N||_{L^p(\T\times [0,a)\cup(b,1])}=0.
\end{align*}

\begin{align*}
||f_N||_{L^p}^p&=\int_I \left|\sum_{n=1}^N c_n \chi_{I_n}\right|^p dxdy \\
&=\sum_{k=1}^N\int_{I_k}\left|\sum_{n=1}^N c_n \chi_{I_n}\right|^p dxdy\\
&=\sum_{k=1}^N |c_k|^p\int_{I_k}dxdy=\sum_{k=1}^N|c_k|^p |I_k|.
\end{align*}

Adding and subtracting appropriate terms, we split the term of interest as follows
\begin{align*}
|| f\circ\Phi_{\eta_1}-f\circ\Phi_{\eta_2}||_{L^p(\T\times [0,a)\cup (b,1])}\leq & \sum^2_{i=1}||f\circ\Phi_{\eta_i}-f_N\circ \Phi_{\eta_i}||_{L^p(\T\times [0,a)\cup (b,1])}\\&+ ||f_N\circ\Phi_{\eta_1}-f_N\circ\Phi_{\eta_2}||_{L^p(\T\times [0,a)\cup (b,1])}.
\end{align*}
Here, we can bound
\begin{align*}
\int_{\T\times[0,a)\cup(b,1]}\left|f\circ \Phi_{\eta_i}-f_N\circ\Phi_{\eta_i}\right|^pdydx&=\int_{\T\times[0,a+\eta_i(x,a))\cup(b+\eta_i(x,b),1]}\left|f-f_N\right|^p
\frac{1}{1+\pa_y\eta_i\circ\Phi^{-1}_{\eta_i}}dydx\\
&\leq C ||f-f_N||^p_{L^p(\T\times [0,1])}.
\end{align*}
Now, we take $N$ large enough to have that $||f-f_N||_{L^p(\T\times [0,1])}\leq \tilde{\ep}/(100C)$. It remains to estimate
\begin{align*}
&||f_N\circ\Phi_{\eta_1}-f_N\circ\Phi_{\eta_2}||^p_{L^p(\T\times [0,a)\cup (b,1])}=\int_{\T\times [0,a)\cup (b,1]} \left|\sum_{n=1}^N c_n (\chi_{I_n}\circ\Phi_{\eta_1}-\chi_{I_n}\circ\Phi_{\eta_2})\right|^pdy dx\\
& \leq \int_{\Omega_{\text{in}}\cup \Omega_{\text{out}}} \left|\sum_{n=1}^N c_n (\chi_{I_n}-\chi_{I_n}\circ\Phi_{\eta_2}\circ \Phi^{-1}_{\eta_1})\right|^p J_{\Phi_{\eta_1}} dy dx\\
&\leq C\int_{\Omega_{\text{in}}\cup \Omega_{\text{out}}} \left|\sum_{n=1}^N c_n (\chi_{I_n}-\chi_{I_n}\circ\Phi_{\eta_2}\circ \Phi^{-1}_{\eta_1})\right|^p  dy dx.
\end{align*}

We deal with the integral over $\Omega_{\text{in}}$ and $ \Omega_{\text{out}}$ separately. Here, we recall that
\begin{align*}
    \Omega_{\text{in}}&=\{x\in \T,\,\, 0\leq y<a+\eta_1(x,a)\},\\
    \Omega_{\text{out}}&=\{x\in \T,\,\, b+\eta_1(x,b)<y\leq 1\}.
\end{align*}
We have that
\begin{align*}
&\int_{\Omega_{\text{in}}} \left|\sum_{n=1}^N c_n (\chi_{I_n}-\chi_{I_n}\circ\Phi_{\eta_2}\circ \Phi^{-1}_{\eta_1})\right|^p  dy dx\\
&=\sum_{k=1}^N \int_{\Omega_{\text{in}}\cap I_k}\left|\sum_{n=1}^N c_n (\chi_{I_n}-\chi_{I_n}\circ\Phi_{\eta_2}\circ \Phi^{-1}_{\eta_1})\right|^p  dy dx\\
&=\sum_{k\in S_{in}}  \int_{\Omega_{\text{in}}\cap I_k}\left|\sum_{n=1}^N c_n (\chi_{I_n}-\chi_{I_n}\circ\Phi_{\eta_2}\circ \Phi^{-1}_{\eta_1})\right|^p  dy dx,
\end{align*}
where
$$S_{in}=\{ n=1,..., N \,\, :\,\, I_n \cap \Omega_{\text{in}}\neq \emptyset\}.$$
In the integral
\begin{align*}
\int_{\Omega_{\text{in}}\cap I_k}\left|\sum_{n=1}^N c_n (\chi_{I_n}-\chi_{I_n}\circ\Phi_{\eta_2}\circ \Phi^{-1}_{\eta_1})\right|^p  dy dx,
\end{align*}
we just have the contribution of, at most, nine terms, for $\delta$ small enough. Indeed,
\begin{align*}
&\int_{\Omega_{\text{in}}\cap I_k}\left|\sum_{n=1}^N c_n (\chi_{I_n}-\chi_{I_n}\circ\Phi_{\eta_2}\circ \Phi^{-1}_{\eta_1})\right|^p  dy dx\\
&=\int_{\Omega_{\text{in}}\cap I_k}\left| c_k (\chi_{I_k}-\chi_{I_k}\circ\Phi_{\eta_2}\circ \Phi^{-1}_{\eta_1})
-\sum_{i=1}^8c_{k(i)} \chi_{I_k(i)}\circ\Phi_{\eta_2}\circ \Phi^{-1}_{\eta_1}\right|^p  dy dx,
\end{align*}
where $\{I_k(i)\}_{i=1}^{8}$ are the neighbor rectangles  of $I_k$. It is clear that each rectangle $I_k$ has at most $8$ neighboring rectangles, in case of having less we simply assume that $I_k(i)=\emptyset$ for the remaining ones.

Thus,
\begin{align*}
&\int_{\Omega_{\text{in}}\cap I_k}\left|\sum_{n=1}^N c_n (\chi_{I_n}-\chi_{I_n}\circ\Phi_{\eta_2}\circ \Phi^{-1}_{\eta_1})\right|^p  dy dx\\
&\leq C\int_ {\Omega_{\text{in}}\cap I_k} \left| c_k (\chi_{I_k}-\chi_{I_k}\circ\Phi_{\eta_2}\circ \Phi^{-1}_{\eta_1})
\right|^p+\sum_{i=1}^8 \left|c_{k(i)}\chi_{I_k(i)}\circ\Phi_{\eta_2}\circ \Phi^{-1}_{\eta_1}\right|^p  dy dx.
\end{align*}
Let $$\rho=\min_{k\in S_{in}} |I_k\cap \Omega_{\text{in}}|.$$

On the one hand, the term
$$|\chi_{I_k}-\chi_{I_k}\circ\Phi_{\eta_2}\circ \Phi^{-1}_{\eta_1}|,$$ over $I_k\cap \Omega_{\text{in}}$,  is just different from zero in the set $\Phi_{\eta_2}\circ \Phi^{-1}_{\eta_1}(I_k)\setminus I_k$. The measure of this term can be made as small as we want taking $\delta$ small enough. Let us say that

$$|\Phi_{\eta_2}\circ \Phi^{-1}_{\eta_1}(I_k)\setminus I_k|\leq \tilde{\ep} \rho,$$
for every $k\in S_{in}$.

On the other hand, the terms $\chi_{I_k(i)}\circ\Phi_{\eta_2}\circ \Phi^{-1}_{\eta_1}$, over $I_k\cap \Omega_{\text{in}}$, are just different from zero in the set $\Phi_{\eta_2}\circ\Phi^{-1}_{\eta_1}(I_{k}\cap \Omega_{\text{in}})\cap I_{k(i)}$. The measure of this terms can be made as small as we want taking $\delta$ small enough. We will take
$$|\Phi_{\eta_2}\circ\Phi^{-1}_{\eta_1}(I_{k}\cap \Omega_{\text{in}})\cap I_{k(i)}|\leq \tilde{\ep}\rho,$$
for every $k\in S_{in}$.

Therefore
\begin{align*}
&\int_{\Omega_{\text{in}}\cap I_k}\left|\sum_{n=1}^N c_n (\chi_{I_n}-\chi_{I_n}\circ\Phi_{\eta_2}\circ \Phi^{-1}_{\eta_1})\right|^p  dy dx\\
&\leq C \left(|c_k|^p+\sum_{i=1}^8 |c_{k(i)}|^p \right)\tilde{\ep}\rho\\
&\leq C\tilde{\ep}\left(|c_k|^p|I_k|+\sum_{i=1}^8 |c_{k(i)}|^p |I_{k(i)}|\right).
\end{align*}
Thus
\begin{align*}
\sum_{k\in S_{in}}  \int_{\Omega_{\text{in}}\cap I_k}\left|\sum_{n=1}^N c_n (\chi_{I_n}-\chi_{I_n}\circ\Phi_{\eta_2}\circ \Phi^{-1}_{\eta_1})\right|^p  dy dx\leq C \tilde{\ep}\sum_{k\in S_{in}}|c_k|^p |I_k| \leq  C\tilde{\ep}||f_N||_{L^p}^p.
\end{align*}
We achieve the conclusion of the Lemma by making a similar argument in $\Omega_{\text{out}}$.

\end{proof}

\begin{Lemma} Let $h\in \S$. The differential
\begin{align*}
D_f\overline{\Psi}_{(\cdot)}h\,:\, B_{\ep/4} \to C^{1+\alpha}(\T\times [0,a)\cup (b,1])
\end{align*}
is continuous in $B_{\ep/4}$.
\end{Lemma}
\begin{proof}
According to \eqref{dpsibarra}, we have that
\begin{align*}
D_f\overline{\Psi}_\eta h = \left(\pa_y\Psi_\eta\circ \Phi_{\eta}\right)\cdot h+D_f\Psi_{\eta}h \circ\Phi_{\eta}.
\end{align*}
We will show that, fixed $\eta_2\in B_{\ep/4}$, for all $\tilde{\ep}>0$ there exists $\delta$ such that
\begin{align*}
||D_f\overline{\Psi}_{\eta_1} h-D_f\overline{\Psi}_{\eta_2}h||_{C^\alpha(\T\times[0,a)\cup (b,1])}\leq \tilde{\ep}
\end{align*}
if
\begin{align*}
||\eta_1-\eta_2||_{C^{1+\alpha}}<\delta.
\end{align*}
We will divide the proof into two parts. Firstly we will prove that
\begin{align}\label{l:parte1}
||\pa_y\Psi_{\eta_1}\circ\Phi_{\eta_1}-\pa_y\Psi_{\eta_2}\circ \Phi_{\eta_2}||_{C^{1+\alpha}}<\tilde{\ep}.
\end{align}
The most delicate term in this norm is
\begin{align*}
||\nabla \left(\pa_y\Psi_{\eta_1}\circ\Phi_{\eta_1}-\pa_y\Psi_{\eta_2}\circ \Phi_{\eta_2}\right)||
_{C^\alpha(\T\times[0,a)\cup (b,1])}.
\end{align*}
Since $\nabla \left(\pa_y\Psi_{\eta_i}\circ\Phi_{\eta_i}\right)=\nabla \Phi_{\eta_i}\cdot\nabla \pa_y \Psi_{\eta_i}\circ\Phi_{\eta_i}$ we can split
\begin{align*}
&\nabla \left(\pa_y\Psi_{\eta_1}\circ\Phi_{\eta_1}-\pa_y\Psi_{\eta_2}\circ \Phi_{\eta_2}\right)\\&
= (\nabla \Phi_{\eta_1}-\nabla\Phi_{\eta_2})\cdot \nabla \pa_y\Psi_{\eta_1}\circ \Phi_{\eta_1}+ \nabla\Phi_{\eta_2}\cdot(\nabla \pa_y \Psi_{\eta_1}\circ\Phi_{\eta_1}-\nabla\pa_y\Psi_{\eta_2}\circ\Phi_{\eta_2}).
\end{align*}
Then, using the fact that $C^\alpha$ is a Banach algebra, we find that
\begin{align*}
&||\nabla \left(\pa_y\Psi_{\eta_1}\circ\Phi_{\eta_1}-\pa_y\Psi_{\eta_2}\circ \Phi_{\eta_2}\right)||_{C^\alpha}\\
&\leq ||\eta_1-\eta_2||_{C^{1+\alpha}} ||\nabla \pa_y\Psi_{\eta_1}\circ \Phi_{\eta_1}||_{C^\alpha}+
||\nabla\Phi_{\eta_2}||_{C^\alpha}||\nabla \pa_y \Psi_{\eta_1}\circ\Phi_{\eta_1}-\nabla\pa_y\Psi_{\eta_2}\circ\Phi_{\eta_2}||_{C^\alpha}.
\end{align*}
The first term of the above expression can be made as small as we want since $\Psi_{\eta_1}\in C^{2+\beta}(\T\times[0,1])$ for all $\beta \in(0,1)$ and by hypothesis $||\eta_1-\eta_2||_{C^{1+\alpha}}< \delta$ with $\delta$ a sufficiently small free parameter.

Therefore, we only need to focus our attention on the other term.
In addition, the last factor of the above expression can be split as
\begin{align}\label{aux:delicate1+2}
&||\nabla \pa_y \Psi_{\eta_1}\circ\Phi_{\eta_1}-\nabla\pa_y\Psi_{\eta_2}\circ\Phi_{\eta_2}||_{C^\alpha}\\
&\leq ||\nabla\pa_y \Psi_{\eta_1}\circ \Phi_{\eta_1}-\nabla\pa_y \Psi_{\eta_2}\circ \Phi_{\eta_1}||_{C^\alpha}
+||\nabla\pa_y \Psi_{\eta_2}\circ \Phi_{\eta_1}-\nabla\pa_y \Psi_{\eta_2}\circ \Phi_{\eta_2}||_{C^\alpha}.\nonumber
\end{align}
For the last term of the above expression \eqref{aux:delicate1+2}, we  apply Morrey's inequality with $p=2/(1-\alpha)$.
\[
||\nabla\pa_y \Psi_{\eta_2}\circ \Phi_{\eta_1}-\nabla\pa_y \Psi_{\eta_2}\circ \Phi_{\eta_2}||_{C^\alpha} \leq ||\nabla\pa_y \Psi_{\eta_2}\circ \Phi_{\eta_1}-\nabla\pa_y \Psi_{\eta_2}\circ \Phi_{\eta_2}||_{W^{1,p}}.
\]
Thus, the delicate term in this norm is
\begin{align*}
&||\nabla^2\pa_y\Psi_{\eta_2}\circ \Phi_{\eta_1}\cdot \nabla \Phi_{\eta_1} - \nabla^2\pa_y\Psi_{\eta_2}\circ \Phi_{\eta_2}\cdot \nabla \Phi_{\eta_2}||_{L^p}\\
&\leq ||\nabla^2\pa_y \Psi_{\eta_2}\circ\Phi_{\eta_1}\cdot \nabla (\Phi_{\eta_1}-\Phi_{\eta_2})||_{L^p}  +|| \nabla^2  \pa_y  \Psi_{\eta_2} \circ (\Phi_{\eta_1}-\Phi_{\eta_2})\cdot \nabla\Phi_{\eta_2}||_{L^p}\\
&\leq ||\nabla^2\pa_y \Psi_{\eta_2}\circ\Phi_{\eta_1}||_{L^p}|| \nabla (\Phi_{\eta_1}-\Phi_{\eta_2})||_{L^\infty} + || \nabla^2  \pa_y  \Psi_{\eta_2} \circ \Phi_{\eta_1}-\nabla^2  \pa_y  \Psi_{\eta_2} \circ\Phi_{\eta_2}||_{L^p} || \nabla\Phi_{\eta_2}||_{L^\infty}.
\end{align*}
The first factor can be handle very easily. Note that
\[
 ||\nabla^2\pa_y \Psi_{\eta_2}\circ\Phi_{\eta_1}||_{L^p(\T\times[0,a)\cup(b,1])}|| \nabla (\Phi_{\eta_1}-\Phi_{\eta_2})||_{L^\infty}\leq ||\nabla^2\pa_y \Psi_{\eta_2}||_{L^p(\T\times[0,1])} ||\eta_1-\eta_2||_{C^1}.
\]
For the other one, we just need to apply Lemma \ref{aux} with $f\equiv\nabla^2\pa_y \Psi_{\eta_2}$, since $\Psi_{\eta_i}\in W^{3,p}(\T\times[0,1])$  and consequently $f\equiv\nabla^2\pa_y \Psi_{\eta_2}\in L^p(\T\times[0,1])$, to get the desired result.

For the first term of \eqref{aux:delicate1+2}, we have that
\begin{multline*}
||\nabla\pa_y \Psi_{\eta_1}\circ \Phi_{\eta_1}-\nabla\pa_y \Psi_{\eta_2}\circ \Phi_{\eta_1}||_{C^\alpha(\T\times[0,a)\cup(b,1])}\\
\leq ||\nabla\Phi_{\eta_1}||^\alpha_{L^\infty(\T\times[0,a)\cup(b,1])}||\nabla\pa_y \Psi_{\eta_1}-\nabla\pa_y \Psi_{\eta_2}||_{C^\alpha(\T\times[0,1])}.
\end{multline*}
Consequently, we have
\begin{align*}
||\nabla\pa_y \Psi_{\eta_1}\circ \Phi_{\eta_1}-\nabla\pa_y \Psi_{\eta_2}\circ \Phi_{\eta_1}||_{C^\alpha(\T\times[0,a)\cup(b,1])}&\leq ||\nabla\pa_y \left(\Psi_{\eta_1} -\Psi_{\eta_2}\right)||_{C^\alpha(\T\times[0,1]])}\\
&\leq ||\nabla\pa_y \left(\Psi_{\eta_1} -\Psi_{\eta_2}\right)||_{W^{1,p}}\\
&\leq ||\Psi_{\eta_1} -\Psi_{\eta_2}||_{W^{3,p}}\leq ||\omega_{\eta_1}-\omega_{\eta_2}||_{W^{1,p}},
\end{align*}
%\begin{align*}
%\frac{|F_1\circ\Phi-F_2\circ\Phi|}{|x-y|^\alpha}
%=\frac{|F_1\circ\Phi-F_2\circ\Phi|}{|\Phi(x)-\Phi(y)|^\alpha}\frac{|\Phi(x)-\Phi(y)|^\alpha}{|x-y|}
%\leq ||\nabla\phi||^\alpha_{L^\infty} ||F_1-F_2||_{C^\alpha(\T\times [0,1])}
%\end{align*}
where we have used Morrey's inequality with $p=2/(1-\alpha)$ and elliptic regularity theory.

To do small this last term, we just need to start proceeding as in the Lemma \ref{l:gateaux} to control $||\omega_{\eta_1}-\omega_{\eta_2}||_{L^p}$. For the other term, i.e. $||\omega_{\eta_1}-\omega_{\eta_2}||_{\dot{W}^{1,p}}=|| \nabla w_{\eta_1}-\nabla w_{\eta_2}||_{L^p}$ we start noting that
\[
\nabla w_{\eta_i}=\begin{cases}
\frac{(-\pa_x \eta_i,1)}{1+\pa_y \eta_i}\varpi' \circ \Phi_{\eta_i}^{-1}, \qquad &\text{over }  \Omega_{\eta_i},\\
0, \qquad &\text{otherwise},
\end{cases}
\]
with
\begin{align*}
\Omega_{\eta_i}\equiv\{ (\bar{x},\bar{y})\in \T \times [0,1]\, :\, (\bar{x},\bar{y})=\Phi_{\eta_i}(x,y)  \text{ with } x\in \T \text{ and } y\in \supp(\varpi')\}.
\end{align*}
Thus
\begin{align*}
\nabla\omega_{\eta_1}-\nabla\omega_{\eta_2}=\left\{
\begin{array}{cc}
\frac{(-\pa_x \eta_1,1)}{1+\pa_y \eta_1}\varpi' \circ \Phi_{\eta_1}^{-1}-\frac{(-\pa_x \eta_2,1)}{1+\pa_y \eta_2}\varpi' \circ \Phi_{\eta_2}^{-1}, \qquad &\text{over } \Omega_{\eta_1}\cap\Omega_{\eta_2},\\
\frac{(-\pa_x \eta_1,1)}{1+\pa_y \eta_1}\varpi' \circ \Phi_{\eta_1}^{-1}, \qquad &\text{over }  \Omega_{\eta_1}\cap\Omega_{\eta_2}^c,\\
-\frac{(-\pa_x \eta_2,1)}{1+\pa_y \eta_2}\varpi' \circ \Phi_{\eta_2}^{-1}, \qquad &\text{over }  \Omega_{\eta_1}^c\cap\Omega_{\eta_2},\\
0,\qquad &\text{over }  \Omega_{\eta_1}^c\cap\Omega_{\eta_2}^c.
\end{array}\right.
\end{align*}
The terms $\frac{(-\pa_x \eta_i,1)}{1+\pa_y \eta_i}\varpi' \circ \Phi_{\eta_i}^{-1}$ for $i=1,2$ can be bounded in $L^\infty$ such that
\begin{align*}
\left|\left| \frac{(-\pa_x \eta_i,1)}{1+\pa_y \eta_i}\varpi' \circ \Phi_{\eta_i}^{-1} \right|\right|_{L^p}^p\leq C \iint_{\left(\Omega_{\eta_1}^c\cap\Omega_{\eta_2}\right)\cup \left(\Omega_{\eta_1}\cap\Omega_{\eta_2}^c\right)}dy dx \leq C ||\eta_1-\eta_2||_{C^{1+\a}}.
\end{align*}
Over the region $\Omega_{\eta_1}\cap\Omega_{\eta_2}$ we have
\begin{multline*}
\frac{(-\pa_x \eta_1,1)}{1+\pa_y \eta_1}\varpi' \circ \Phi_{\eta_1}^{-1}-\frac{(-\pa_x \eta_2,1)}{1+\pa_y \eta_2}\varpi' \circ \Phi_{\eta_2}^{-1}\\
=\left(\frac{(-\pa_x \eta_1,1)}{1+\pa_y \eta_1}-\frac{(-\pa_x \eta_2,1)}{1+\pa_y \eta_2}\right)\varpi' \circ \Phi_{\eta_1}^{-1}+\frac{(-\pa_x \eta_2,1)}{1+\pa_y \eta_2}\left(\varpi' \circ \Phi_{\eta_1}^{-1}-\varpi' \circ \Phi_{\eta_2}^{-1}\right).
\end{multline*}
We will now study each of these terms in more detail. For the first, by adding and subtracting appropriate terms we have
\[
\frac{(-\pa_x \eta_1,1)}{1+\pa_y \eta_1}-\frac{(-\pa_x \eta_2,1)}{1+\pa_y \eta_2}=(-\pa_x \eta_1,1)\frac{\pa_y (\eta_2-\eta_1)}{(1+\pa_y\eta_1)(1+\pa_y\eta_2)}+\frac{(-\pa_x(\eta_1-\eta_2),0)}{1+\pa_y\eta_2}
\]
Since by hypothesis $\eta_1, \eta_2\in B_{\ep/4}$, in particular both belongs to $\S$ and consequently we get
\[
\left|\left|\left(\frac{(-\pa_x \eta_1,1)}{1+\pa_y \eta_1}-\frac{(-\pa_x \eta_2,1)}{1+\pa_y \eta_2}\right)\varpi' \circ \Phi_{\eta_1}^{-1}\right|\right|_{L^\infty}\leq C ||\eta_1-\eta_2||_{C^{1+\a}}.
\]
For the second term, we just need to note that $\varpi' \circ \Phi_{\eta_1}^{-1}-\varpi' \circ \Phi_{\eta_2}^{-1}=0$ over $\Omega_{\eta_1}\cap\Omega_{\eta_2}$. Then, combining the above argument yields
\[
|| \nabla w_{\eta_1}-\nabla w_{\eta_2}||_{L^p}\leq C ||\eta_1-\eta_2||_{C^{1+\a}}.
\]
Going back in our argument, we conclude that we have proven \eqref{l:parte1}.
Finally, we prove that
\begin{multline*}
|| D_f\Psi_{\eta_1}h \circ\Phi_{\eta_1}-D_f\Psi_{\eta_2}h \circ\Phi_{\eta_2}||_{C^\alpha(\T\times[0,a)\cup(b,1])}\\
\leq ||D_f\Psi_{\eta_1}h \circ(\Phi_{\eta_1}-\Phi_{\eta_2})||_{C^\alpha(\T\times[0,a)\cup(b,1])} + || (D_f \Psi_{\eta_1}h-D_f\Psi_{\eta_2}h)\circ\Phi_{\eta_2}||_{C^\alpha(\T\times[0,a)\cup(b,1])}\\
\leq ||D_f\Psi_{\eta_1}h||_{C^{1+\a}(\T\times[0,1])} ||\Phi_{\eta_1}-\Phi_{\eta_2}||_{C^\alpha(\T\times[0,a)\cup(b,1])}\\
+ ||\nabla\Phi_{\eta_2}||^\alpha_{L^{\infty}(\T\times[0,a)\cup(b,1])} ||D_f\Psi_{\eta_1}h -D_f\Psi_{\eta_2}h||_{C^\a(\T\times[0,1])}.
\end{multline*}
To prove smallness it suffices to apply the Lemma \ref{continuitydPsi}.
\end{proof}

By combining the above results, we can conclude the following result.
\begin{Lemma} The functional
\begin{align*}
F[f,\lambda]=\lambda f -\overline{\Psi}_f+\frac{1}{2\pi}\int_{\T} \overline{\Psi}_f(\bar{x},\cdot)d\bar{x}
\end{align*}
acting from $B_{\ep/4}$ to $\S$ is Frechet differentiable and its derivative is continuous.
\end{Lemma}

\section{Analysis of the linear operator}\label{s:analysisLOP}
In order to compute the differential of the functional $F$ at the origin we will take $\eta=0$ in \eqref{dpsibarra}. Since $\Phi_0(x,y)=(x,y)$, we  have that \eqref{ecupsi1} translates into
\begin{equation}\label{ec:Psi0}
\left\{
\begin{aligned}
\Delta \Psi_0&= \omega_0\quad \text{in $\T\times [0,1],$}\\
\Psi_0\big|_{y=0}&= 0,\\
\Psi_0\big|_{y=1}&= - \tfrac{1}{3},
\end{aligned}
\right.
\end{equation}
and \eqref{eq:DfPsieta} translates into
\begin{equation}\label{laplacedif}
\left\{
\begin{aligned}
\Delta D_f\Psi_0 h = &d\omega_{0},\quad \text{ in $\T\times [0,1]$},\\
D_f\Psi_0 h\big|_{y=0}=&0,\\
D_f\Psi_0 h \big|_{y=1}=&0,
\end{aligned}
\right.
\end{equation}
where
\begin{align*}
d\omega_0=\left\{\begin{array}{cc}-\varpi'(y)h(x,y)&\qquad (x,y)\in \T\times [0,a),\\
0 &\qquad (x,y)\in \T\times [a,b],\\
-\varpi'(y)h(x,y) &\qquad (x,y)\in \T\times (b,1]. \end{array}\right.
\end{align*}

For $\Psi_0$, solving \eqref{ec:Psi0}, we obtain that
\begin{align*}
\Psi_0(y)=\int_{0}^y\int_{0}^{\bar{y}}\varpi(z)dzd\bar{y}+v_{\varpi}y,
\end{align*}
with
\begin{align}\label{vvarpi}
v_{\varpi}\equiv-\int_{0}^1\int_{0}^{\bar{y}}\varpi(z)dzd\bar{y}-\frac{1}{3}.
\end{align}
This means that
\begin{align*}
\Psi_0'(y)=\int_{0}^y\varpi(\bar{y})d\bar{y}+v_{\varpi}=u(y)+v_{\varpi},
\end{align*}
where $u$ is given by \eqref{velocidad}. Let us emphasis that the velocity of the flow corresponding with $\varpi$ will be given by
\begin{align*}
\bu_{\text{shear}}(y)=(-\Psi'_0(y),\, 0),
\end{align*}
which for $\ep=0$ is the Poiseuille flow. We find, using \eqref{dpsibarra}, that
\begin{align}\label{datzero}
D_f\overline{\Psi}_0 h (x,y)=h(x,y)\Psi_0'(y)+D_f\Psi_0 h (x,y).
\end{align}
Since the mean of $h$ is equal to zero, we obtain
\begin{align}\label{operadorlineal}
D_fF[0,\lambda]h =&\lambda h -\Psi_0' h-D_f\Psi_0 h+\frac{1}{2\pi}\int_{\T} \left(\Psi_0'(y)h(\bar{x},y)+D_f\Psi_0 h(\bar{x},y)\right)d\bar{x}\\
=&\left(\lambda  -\Psi_0' \right)h-D_f\Psi_0 h. \nonumber
\end{align}
Next, we analyze the term $D_f\Psi_0 h$. In order to do it, we need to introduce the expansions
\begin{align*}
h&=\sum_{n=1}^\infty h_n(y)\cos(nx),\\
\varphi&=\sum_{n=1}^\infty \varphi_n(y)\cos(nx),
\end{align*}
where, for simplicity, we used the notation $\varphi \equiv D_f\Psi_0 h$.

Introducing both expansions into \eqref{laplacedif}, we have the relation
\begin{equation*}
\left\{
\begin{aligned}
\pa^2_y\varphi_n-n^2\varphi_n=&-\varpi'h_n,\quad y\in [0,a)\cup (b,1],\\
\varphi_n(1)=&0,\\
\varphi_n(0)=&0.
\end{aligned}
\right.
\end{equation*}
Thus, solving the above system, we get
\begin{align*}
\varphi_n(y)=\int_{0}^1\varpi'(z)G_n(y,z)h_n(z)dz,
\end{align*}
where
\begin{align*}
G_n(y,z)=\frac{1}{n\sinh(n)}\left\{\begin{array}{cc}\sinh(n(1-y))\sinh(nz) &\quad  z<y,\\ \sinh(ny)\sinh(n(1-z)) &\quad  z>y.\end{array}\right.
\end{align*}
Notice that $G_n(\cdot,\cdot)$ satisfies the following properties
\begin{align}\label{greenfunction}
\pa_y^2G_{n}(y,z)-n^2G_n(y,z)=&-\delta(y-z),\\
G_n(0,z)=&0,\nonumber\\
G_n(1,z)=&0.\nonumber
\end{align}

To abbreviate the notation, let us write
\begin{align}\label{defiL}
\cL^\lambda h \equiv D_f F[0,\lambda] h.
\end{align}
We have obtained the following decomposition
\begin{align*}
\cL^\lambda h(x,y) =\sum_{n=1}^\infty \cL^\lambda_{n}h_n (y)\cos(nx),
\end{align*}
with
\begin{align*}
\cL_{n}^\lambda h_n(y)\equiv(\lambda-\Psi'_0(y))h_n(y)-\int_{0}^1\varpi'(z)G_n(y,z)h_n(z)dz.
\end{align*}
Taking into account that $\Psi_0'(y)=u(y)+v_\varpi$ with $u(y)$ given in \eqref{velocidad} and $v_\varpi$ given in \eqref{vvarpi} we have
\begin{align*}
\cL^\lambda_n h_n(y) = & (\lambda-v_{\varpi}+y^2)h_n(y)+2\int_{0}^1\chi(z)G_n(y,z)h_n(z)dz, \quad  &0\leq y<a,\\
\cL^\lambda_n h_n(y) = &\left(\lambda-v_{\varpi}+(y-b)^2+2ay-a^2\right)h_n(y)+2\int_{0}^1\chi(z)  G_n(y,z)h_n(z)dz, \quad  &b<y\leq 1,
\end{align*}
where $\chi(z)$ is the characteristic function of the domain $[0,a)\cup (b,1]$.

\section{Study of the linear operator kernel}\label{seis}
In sections \ref{seis} and \ref{s:dimkernel} we will prove the following Lemma.
\begin{theorem}\label{mainexistencelemma} There exists $\ep_0>0$ sufficiently small  such that for all $0<\ep<\ep_0$ there exists $\lambda^*=\lambda^*(\ep)\in \R$ such that the kernel of the operator
\begin{align*}
\cL^{\lambda^\ast}\, :\, \S \to \S
\end{align*}
is one dimensional. In addition, the element $h^*$ that spans the kernel of $\cL^{\lambda^\ast}$ is $C^\infty(\T\times [0,a]\cup[b,1])$ and, very importantly, the function $\lambda^*-\Psi'_0(y)$ has not zeros in $[0,a]\cup [b,1]$.
\end{theorem}

The proof of this Lemma is divided  in three different parts:

\begin{itemize}
    \item Combining the results of the sections \ref{existence}-\ref{s:closing}, we will prove
    the existence of an element in the kernel of the linear operator for $\ep$ small enough. That is, we prove
    the existence of a nontrivial  couple $(\lambda^*, h^*)\in \R\times \S$ such that $$\cL^{\lambda^\ast}h^*=0.$$ The element $h^*$ will be actually smooth. The fact that $\lambda^*-\Psi'_0$ has not zeros in $[0,a]\cup [b,1]$ will developed a very important role in the proof.

 \item In section \ref{s:regularidad} and  section \ref{s:dimkernel}, we will prove that the kernel of the linear operator $\cL^{\lambda^\ast}$ is one dimensional. That is, we prove that for that particular  value of $\lambda^*\in\R$ if $f\in\S$ solves $\cL^{\lambda^\ast} f=0$ then $f$ is proportional to $h^*$, i.e., both are equal modulo a multiplication by a constant.

\end{itemize}

%In section \eqref{existence} we will prove the existence of a nontrivial  couple $(\lambda^*, h^*)\in \R\times \S$ such that $\cL^{\lambda^\ast}h^*=0$. In section \eqref{uniqueness} will be proven that for that value $\lambda^*$ if $f\in\S$ solves $\cL^{\lambda^\ast} f=0$ then $f$ is proportional to $h^*$. Finally, the section \eqref{regularity} contains the proof of that $h^*\in C^\infty$.

\subsection{Integral equations for an element in the kernel }\label{existence}
We have to consider the following equations for $(\bar{\lambda}\equiv v_\varpi-\lambda, \{h_n\}_{n=0}^\infty)$:
\begin{align*}
\left(y^2-\bar{\lambda}\right)h_n(y) +2\int_{0}^1\chi(z)G_n(y,z)h_n(z)dz = &0, \quad 0\leq y<a,\\
\left((y-b)^2+2ay-a^2-\bar{\lambda}\right)h_n(y) +2\int_{0}^1\chi(z)G_n(y,z)h_n(z)dz = &0, \quad b<y\leq 1,
\end{align*}
together with the boundary conditions
\begin{align*}
h_n(0)=&0,\\
h_n(1)=&0,
\end{align*}
for all $n=1,2,\ldots$ Notice that if we found such a value $\bar{\lambda}$ then $\lambda^*=-\bar{\lambda}+v_\varpi$.
We will take
\begin{align*}
h_n= 0 , \quad  n\neq m
\end{align*}
for some fixed integer $m>0$. Then, we just have to solve
\begin{align}
\left(y^2-\bar{\lambda}\right)h_m(y) +2\int_{0}^1\chi(z)G_m(y,z)h_m(z)dz = &0, \quad 0\leq y<a,\label{kernel10}\\
\left(y^2-2\ep y +\ep -\bar{\lambda}\right)h_m(y)+2\int_{0}^1\chi(z)G_m(y,z)h_m(z)dz = &0, \label{kernel20}\quad b<y\leq 1,
\end{align}
where we did some algebra and introduced $a=1/2-\ep/2$ and $b=1/2+\ep/2$.

We could work with any integer $m>0$ but, for sake of simplicity, we will take $m=1$. However at this point is still useful to keep the letter $m$.

We now define the auxiliary parameters $\mu$, $\nu$ and $\rho$ in terms of $\ep$ and $\bar{\lambda}$ as follows
\begin{align*}
\mu &\equiv  \sqrt{\bar{\lambda}},\\
\nu &\equiv  \ep +\sqrt{\ep^2-\ep+\bar{\lambda}},\\
\rho &\equiv  \ep -\sqrt{\ep^2-\ep+\bar{\lambda}},
\end{align*}
and we make the very important assumption of
\begin{align}\label{rangolambda}
a^2<\bar{\lambda}< b^2-\ep^2.
\end{align}
Notice that \eqref{rangolambda} is equivalent to
\begin{align*}
-u(a)<\bar{\lambda} <-u(b),
\end{align*}
or
\begin{align*}
-u(a)<-\lambda+v_\varpi<-u(b),
\end{align*}
or equivalently
\begin{align*}
u_{\text{shear}\, 1}(a)=-(u(a)+v_{\varpi})<-\lambda<-(u(b)+v_\varpi)=u_{\text{shear}\, 1}(b).
\end{align*}
This means that $-\lambda=u_{\text{shear 1}}(y^*)$, for some $y^*\in (a,b)$. In fact, \eqref{rangolambda} also implies that there is no zeros for the functions
\begin{align*}
y^2-\bar{\lambda}=&(y-\mu)(y+\mu), \qquad 0\leq y <a,\\
y^2-2\ep y +\ep-\bar{\lambda}=&(y-\nu)(y-\rho),\qquad b<y\leq 1,
\end{align*}
and that $\mu, \nu, \rho \in \R$. Indeed, we have that
\begin{equation}\label{imposeAB}
\mu>a \quad \text{and}\quad \rho<\nu< b.
\end{equation}

Next we introduce the function
\[
H(y)\equiv\begin{cases}
h_m(y)(y-\mu)(y+\mu),\qquad 0\leq y<a,\\
h_m(y)(y-\nu)(y-\rho),\qquad b<y\leq 1.
\end{cases}
\]
In terms of $H$, the equations \eqref{kernel10} and \eqref{kernel20} become
\begin{align}\label{eq1d}
H(y)&+C_m \sinh(m(1-y))\int_{0}^{y}\sinh(mz)\frac{H(z)}{z^2-\mu^2}dz\\&+
C_m\sinh(my)\int_{y}^{a}\sinh(m(1-z))\frac{H(z)}{z^2-\mu^2}dz\nonumber\\
&+C_m \sinh(my)\int_{b}^1 \sinh(m(1-z))\frac{H(z)}{(z-\nu)(z-\rho)}dz=0, \qquad 0\leq y<a, \nonumber
\end{align}
and
\begin{align}\label{eq2d}
H(y)&+C_m\sinh(m(1-y))\int_{0}^a\sinh(mz)\frac{H(z)}{z^2-\mu^2}dz\\&+C_m \sinh(m(1-y))\int_{b}^y \sinh(mz)\frac{H(z)}{(z-\nu)(z-\rho)}dz\nonumber\\
&+C_m\sinh(ny)\int_{y}^1\sinh(m(1-z))\frac{H(z)}{(z-\nu)(z-\rho)}dz=0,\qquad b<y\leq 1\nonumber.
\end{align}
with $$C_m\equiv\frac{2}{m\sinh(m)}.$$

Equations \eqref{eq1d} and \eqref{eq2d} are equivalent to
\begin{align}\label{compacta}
H(y)+2\int_{0}^1G_m(y,z)\left(\frac{\chi_{[0,a)}(z)}{z^2-\mu^2}+\frac{\chi_{(b,1]}(z)}{(z-\nu)(z-\rho)}\right)H(z)dz=0, \qquad y\in[0,a)\cup(b,1].
\end{align}

\subsubsection{Rescaling}
Let us rescale the equations \eqref{eq1d} and \eqref{eq2d}. To do so, we consider the following rescaled variable $x$ introduced below, which allows us to fix the domain,
\[
x=\begin{cases}
\frac{y}{1-\ep},\qquad \text{for $0\leq y<a$}, \\
\frac{y-\ep}{1-\ep}, \qquad \text{for $b<y\leq 1$},
\end{cases}
\]
%\begin{align*}
%x=&\frac{y}{1-\ep},\quad \text{for $0<y<a$},\\
%x=&\frac{y-\ep}{1-\ep}, \quad \text{for $b<y<1$},
%\end{align*}
and the rescale function
\[
F(x)=\begin{cases}
H((1-\ep)x),\qquad &\text{for $0\leq x<\frac{1}{2}$},\\
H((1-\ep)x+\ep), \qquad &\text{for $\frac{1}{2}<x\leq 1$}.
\end{cases}
\]
%\begin{align*}
%F(x)=&H((1-\ep)x),\quad \text{for $0<x<\frac{1}{2}$},\\
%F(x)=&H((1-\ep)x+\ep), \quad \text{for $\frac{1}{2}<x<1$}.
%\end{align*}

We recall and emphasize that  $F(x)$ does not have to be continuous at $x=1/2$.

In addition, we introduce the auxiliary parameters
\begin{align*}
\tilde{\mu}=\frac{\mu}{1-\ep},&& \tilde{\nu}=\frac{\nu-\ep}{1-\ep},&&\tilde{C}_n=\frac{C_n}{1-\ep}, &&\tilde{\rho}=\frac{ \rho-\ep}{1-\ep}.
\end{align*}
Here, we notice that $\tilde{\rho}=-\tilde{\nu}.$

Thus, for $0\leq x<1/2$, we have
\begin{align}\label{F1}
F(x)&+\tilde{C}_m\sinh(m(1-(1-\ep)x))\int_{0}^{x}\frac{\sinh(m(1-\ep)w)F(w)}{w^2-\tilde{\mu}^2}dw\\
&+\tilde{C}_m\sinh(m(1-\ep)x)\int_{x}^\frac{1}{2} \frac{\sinh(m(1-(1-\ep)w))F(w)}{w^2-\tilde{\mu}^2}dw\nonumber\\
&+\tilde{C}_m\sinh(m(1-\ep)x)\int_{\frac{1}{2}}^1\frac{\sinh(m(1-\ep)(1-w))F(w)}{w^2-\tilde{\nu}^2}dw\nonumber=0.
\end{align}
And for $1/2<x\leq 1$, we have
\begin{align}\label{F2}
F(x)&+\tilde{C}_m \sinh(m(1-\ep)(1-x))\int_{0}^\frac{1}{2} \frac{\sinh(m(1-\ep)w)F(w)}{w^2-\tilde{\mu}^2}dw\\
&+\tilde{C}_m \sinh(m(1-\ep)(1-x))\int_{\frac{1}{2}}^x \frac{\sinh(m((1-\ep)w+\ep))F(w)}{w^2-\tilde{\nu}^2}dw\nonumber\\
&+\tilde{C}_m \sinh(m((1-\ep)x+\ep))\int_{x}^1\frac{\sinh(m(1-\ep)(1-w))F(w)}{w^2-\tilde{\nu}^2}dw\nonumber.
\end{align}

\subsection{From integral equations to an ODE system}
If $F(x)$ is smooth and solves \eqref{F1} then $F(x)$ solves the following   ordinary differential equation
\begin{align*}
F''(x)-\left(\frac{2}{x^2-\tilde{\mu}^2}+(1-\ep)^2m^2\right)F(x)&=0, \qquad \text{for }0\leq x< 1/2,\\
F(0)&=0.
\end{align*}
In order to check it, we just take two derivatives in \eqref{F1} and then use again \eqref{F1} (you can also take two derivatives in \eqref{compacta}, use \eqref{greenfunction} and finally rescale). Similarly, if $F(x)$ is smooth and solves \eqref{F2} then $F(x)$ solves
\begin{align*}
F''(x)-\left(\frac{2}{x^2-\tilde{\nu}^2}+(1-\e)^2 m^2 \right)F(x)&=0, \qquad \text{for }1/2<x\leq 1,\\
F(1)&=0.
\end{align*}
%\begin{remark}
%Here, it is important to emphasize that we have gone from a coupled system of integral equations \eqref{F1}-\eqref{F2} to two ODEs defined in disjoint domains, which can be solved separately. Here, it is crucial to note that the coupling between the two ODEs is hidden in the parameters $\tilde{\mu},\tilde{\nu}$ and $\tilde{\rho}$, which at the end of the day depend only on $\tilde{\mu}$.
%\end{remark}

We will call $f_m$ the solution of
\begin{align}
f_m''(x)-\left(\frac{2}{x^2-\tilde{\mu}^2}+(1-\ep)^2m^2\right)f_m(x)&=0,\qquad \text{for } 0\leq x<1/2, \label{f_left} \\
f_m''(x)-\left(\frac{2}{x^2-\tilde{\nu}^2}+(1-\e)^2 m^2 \right)f_m(x)&=0, \qquad \text{for }1/2<x\leq 1, \label{f_right}
\end{align}
with
\begin{equation}\label{tapas_fm}
f_m(0)=0=f_m(1),
\end{equation}
and with ``extra'' boundary conditions
\begin{equation}\label{extrabc_fm}
\lim_{x\to 1/2^{-}}f_m(x)=1=\lim_{x\to 1/2+}f_m(x).
\end{equation}
Notice that $f_m$ depends on $\ep$ and $\tilde{\mu}$. We will not make this dependency explicit unless necessary.
When $m=1$, we denote the solution of \eqref{f_left}-\eqref{f_right} with boundary conditions \eqref{tapas_fm}-\eqref{extrabc_fm} just by $f$, that solves
\begin{align}
f''(x)-\left(\frac{2}{x^2-\tilde{\mu}^2}+(1-\ep)^2\right)f(x)&=0,\qquad \text{for } 0\leq x<1/2, \label{f_leftconm=1} \\
f''(x)-\left(\frac{2}{x^2-\tilde{\nu}^2}+(1-\e)^2  \right)f(x)&=0, \qquad \text{for }1/2<x\leq 1, \label{f_rightconm=1}
\end{align}
with
\begin{equation}\label{bc:fconm=1}
f(0)=0=f(1), \qquad \text{and} \lim_{x\to 1/2^{-}}f(x)=1=\lim_{x\to 1/2+}f(x).
\end{equation}
In addition, we denote by $f_0$ the solution of the limiting problem of \eqref{f_leftconm=1} and \eqref{f_rightconm=1} resulting from making $\ep\to 0$. Here, it is important to emphasize that $\tilde{\mu},\tilde{\nu}\to 1/2$ as $\ep$ goes to zero.
That is, $f_0$ solves
\begin{align}\label{g}
f_0''(x)-\left(\frac{2}{x^2-1/4}+1\right)f_0(x)&=0, \quad \text{for }0\leq x<1/2 \text{ and } 1/2 <x\leq 1,
\end{align}
with
\begin{equation}\label{bc:g}
f_0(0)=0=f_0(1), \quad \text{and} \quad \lim_{x\to 1/2^{-}}f_0(x)=1=\lim_{x\to 1/2+}f_0(x).
\end{equation}

 In the rest of this section, for the sake of brevity,  we will focus our attention on the domain $[0,1/2)$ since the calculations follows by similar arguments for the case $(1/2,1].$

Before we start solving \eqref{g}, we provide some information about the ansatz of $\tilde{\mu}$ and consequently of $\tilde{\nu}$. Here, we recall (see \eqref{imposeAB}) that, on one hand $\mu$ has to be bigger that $a=(1-\ep)/2$ and on the other hand we have that $\nu$ has to be smaller that $b=(1+\ep)/2$. Thus,
\begin{align}
    &\mu>a \quad \Longleftrightarrow \quad \tilde{\mu}>\frac{1}{2}, \label{tildemumayorque}\\
    &\nu<b \quad \Longleftrightarrow \quad \frac{\sqrt{\ep^2-\ep + \mu^2}}{1-\ep}<\frac{1}{2} \quad \Longleftrightarrow \quad \sqrt{\tilde{\mu}^2-\frac{\ep}{1-\ep}}<\frac{1}{2}. \label{tildemumenorque}
\end{align}

We will take
\begin{align*}
\tilde{\mu}=\frac{1}{2}+\tilde{\mu}_1(\ep)\ep=\frac{1}{2}+\left(\frac{1}{2}+\mu_1\right)\ep +\ep^2\log^2(\ep)\tilde{\mu}_2(\ep),
\end{align*}
where $-\tfrac{1}{2}<\mu_1<+\tfrac{1}{2}$ and $\tilde{\mu}_2(\ep)=O(1)$ in terms of parameter $\ep.$

Thus both \eqref{tildemumayorque} and \eqref{tildemumenorque}  are satisfied for sufficiently small $\ep$.

\subsection{Regular-Singular points and the method of Frobenius }\label{frobenius}

\subsubsection{Solving the equation \eqref{g}-\eqref{bc:g}}\label{gg1}

Let us consider the ordinary differential equation
\begin{align}\label{principal}
f_0''(x)-\left(\frac{2}{x^2-1/4}+1\right)f_0(x)=0, \quad \text{for } x\in[0,1/2)\cup(1/2,1],
\end{align}
with
\begin{equation}\label{bc:principal}
f_0(0)=0=f_0(1), \qquad \text{and} \qquad f_0(1/2^{-})=1=f_0(1/2^+).
\end{equation}

\begin{remark}
In the rest of the manuscript, for the sake of brevity we will use the notation $$g(x_0^{\pm})\equiv \lim_{x\to x_0^{\pm}}g(x),$$
for an arbitrary function $g$.
\end{remark}

In the domain $[0,1]$, the only regular-singular point is $x=1/2$.  In order to study the solutions of \eqref{principal} around that point we will use the Frobenius method (see, for example, chapter X in \cite{whittaker}). Following the notation in \cite{whittaker}, we can write \eqref{principal} as
\begin{align*}
f_0''(x)+p(x)f_0'(x)+q(x)f_0(x)=0,
\end{align*}
with
\begin{align*}
p(x)=\frac{P(x-x_0)}{x-x_0}, && q(x)=\frac{Q(x-x_0)}{(x-x_0)^2}.
\end{align*}
In our case
\begin{align}\label{pandq}
P(x-x_0)=0, && Q(x-x_0)=-(x-x_0)^2\left(\frac{2}{(x-x_0)(x+1/2)}+1\right),&& x_0=1/2,
\end{align}
and we have that the indicial polynomial
$$I(r)=r(r-1)+P(0)r+Q(0),$$
is given by
$$I(r)=r(r-1),$$ with the roots $r=1$ and $r=0$.\\

Therefore we have two linear independent solutions:
\begin{itemize}
    \item Associated to the larger root $r=1$ we have an analytic solution  on $(-1/2,3/2)$, which we denote by $g_1(x)$,  (notice that the  only singular point apart from $1/2$ is $-1/2$ and then the radius of convergence of the series around $x=1/2$ is 1) which can be obtained from the expansion
\begin{align}\label{g1}
g_1(x)=\sum_{k=1}^\infty g_{1}^{(k)}(x-1/2)^k,
\end{align}
where $g_{1}^{(k)}\in \R$. We choose, without loss of generality, $g_1^{(1)}=1$.

    \item If the roots differs by a integer number, as in our case, then the second solution can be found through the expression
    \begin{align*}
g_2(x)=g_{2}^{(-1)} g_1(x)\log(|x-1/2|)+\sum_{k=0}^\infty g_{2}^{(k)}(x-1/2)^k,
\end{align*}
where $g_{2}^{(k)}\in \R$ for $k=-1,0,1,...$, with $g_2^{(0)}\neq 0$ (see  \cite{whittaker}, pag. 200-201).
The series converges for $x\in (-1/2,3/2).$
For our equation the coefficient $g_{2}^{(-1)}$ is given by $g_2^{(-1)}=2g_2^{(0)}/g_1^{(1)}\neq 0$, then $g_2(x)$ has a singularity at $x=1/2$ of the type $(x-1/2)\log(|x-1/2|)$.
\end{itemize}

\begin{comment}
Using \textsc{Mathematica} it is not difficult to check that
\begin{align*}
    g_1^{(1)}&=1, \quad g_1^{(2)}=1, \quad g_1^{(3)}=1/6, \quad g_1^{(4)}=1/9,\quad  \ldots\\
    g_2^{(-1)}&=2, \quad g_2^{(0)}=1, \quad g_2^{(1)}=-1,\quad g_2^{(2)}=-81/18, \quad g_2^{(3)}=-23/18, \quad \ldots
\end{align*}
\end{comment}

\subsubsection{Frobenius and the boundary conditions}

From $g_1$ and $g_2$ we can construct $f_0$,  satisfying also the boundary condition given by $\eqref{bc:principal}$, as a linear combination of them. That is,
\[
f_0(x)=\begin{cases}
    A g_1(x)+B g_2(x),\qquad 0\leq x<1/2,\\
    C g_1(x)+D g_2(x),\qquad 1/2<x\leq 1,
\end{cases}
\]
with $(A,B)$ and $(C,D)$ solving respectively
\[
\begin{cases}
    A g_1(1/2)+B g_2(1/2)=1,\\
    A g_1(0)+B g_2(0)=0,
\end{cases}
\]
and
\[
\begin{cases}
    C g_1(1/2)+D g_2(1/2)=1,\\
    Cg_1(1)+Dg_2(1)=0.
\end{cases}
\]
In fact, we have
\[
A=-\frac{g_2(0)}{g_1(0)g_2(1/2)}, \qquad B=\frac{1}{g_2(1/2)},
\]
and
\[
C=-\frac{g_2(1)}{g_1(1)g_2(1/2)}, \qquad D=\frac{1}{g_2(1/2)}.
\]
Notice that, in order to have $A,B,C$ and $D$ perfectly well defined, we just need to check that $g_1(0)$ and $g_1(1)$ are both non-zero, since we have seen before that $ g_2(1/2)\neq 0$ due to the fact that $g_2^{(0)}\neq 0$.

In order to prove it, we will use the following  argument by contradiction. For simplicity, we focus our attention only on the case $g_1(0)\neq 0$ since the case $g_1(1)\neq 0$ follows by straightforward modifications. Let us assume that $g_1(0)=0$. Then, we have that $g_1'(0)$ may be positive, negative or equal to zero. The  last possibility is ruled out by uniqueness since $x=0$ is a ordinary point. Now, in the other two possibilities, we combine the fact that $g_1(1/2)=0$ with the next lemma to obtain a contradiction, which allow us to conclude.

%\myr{METER LA PRUEBA GENERAL AQUI}

\begin{Lemma} \label{positivog} Let $f_0$ be a solution of equation \eqref{principal} satisfying $f_0(0)=0$ and $f_0'(0)>0$ (or $f_0'(0)<0$). Then $f_0(x)>0$ (or $f_0(x)<0$) for $0<x\leq 1/2$.
\end{Lemma}

%\begin{proof}
%Before we begin, we will assume without loss of generality that $f_0'(0)=1$. Note that in another case, since it is a linear equation we can simply normalize the solution by $f_0(x)/f_0'(0)$. Then, we write \eqref{g} in a more convenient way as
%\begin{align}\label{um}
%f_0''(x)-f_0(x)= \frac{2}{x^2-1/4}f_0(x).
%\end{align}
%We find that their solution is given by
%\begin{align*}
%f_0(x)=\sinh(x)+\int_{0}^x \sinh(x-z)\frac{2}{z^2-1/4}f_0(z) dz.
%\end{align*}
%Iterating once this formula, we have
%\begin{align*}
%f_0(x)=&\sinh(x)+\int_{0}^x\sinh(x-z)\sinh(z)\frac{2}{z^2-1/4}dz\\
%&+\int_0^x\sinh(x-z)\frac{2}{z^2-1/4}\int_{0}^z \sinh(z-y)\frac{2}{y^2-1/4}f_0(y)dydz.
%\end{align*}
%We have that, as far as $f_0(x)>0$ we obtain
%
%
%
%\begin{align*}
%f_0(x)\geq & \sinh(x)+\int_{0}^x\sinh(x-z)\sinh(z)\frac{2}{z^2-1/4}dz\\
%\geq & \sinh(x)+2^2\sinh^2\left(\frac{1}{2}\right)\int_0^x\frac{2(x-z)z}{z^2-1/4}dz,
%\end{align*}
%where we have used the inequality $\sinh(x)\leq 2\sinh(1/2)x$ for $x\in[0,1/2]$. In addition, computing the last integral, we have
%\[
%f_0(x)\geq \sinh(x)+2^2\sinh^2\left(\frac{1}{2}\right)\left[\arctanh(2x)+x(\log(1-4x^2)-2)\right],
%\]
%which can be easily verified that is strictly positive for $0<x\leq 1/2$.
%
%\end{proof}

\begin{proof}
A more general proof of this fact has been postponed to Section \ref{s:dimkernel}, see Lemma \ref{l:positivityfn}.
\end{proof}

Then, we have proved that there exists  a solution for \eqref{principal} in $(0,1/2)$ satisfying \eqref{bc:principal}  and that it can be written as
\begin{align}\label{gdefi}
f_0(x)=\sum_{k=1}^\infty f_0^{(k)}(x-1/2)^k\log(|x-1/2|)+\sum_{k=1}^\infty \overline{f}_0^{(k)}(x-1/2)^k+1,
\end{align}
for some coefficients $f_0^{(k)},$ $\overline{f}_0^{(k)}\in \R$.

In addition, we show that the solution $f_0$ is unique. If there are two different solutions $f^1_0$ and $f^2_0$, the difference $d=f^1_0-f^2_0$ solve the sam edifferential equation with boundary condition $d(0)=d(1/2)=0$. But then, by Lemma \ref{positivog}, $d'(0)=0$. And since $x=0$ is an ordinary point, we obtain that $d=0$.

Moreover, $f_0(x)$ and $g_1(x)$ are linearly independent with Wronskiano
$$W[f_0,g_1](x)=f_0(x)g'_1(x)-f_0'(x)g_1(x).$$
In fact, as a consequence of Abel-Liouville's formula we  have that
\begin{equation}
W[f_0,g_1](x)=W[f_0,g_1](1/2)=1, \quad \forall x\in[0,1/2].
\end{equation}

To deal with the right part we use a Lemma analogous to \ref{positivog} in $(1/2,1]$ (see Lemma \ref{l:positivityfnright}). In $(0,1/2)\cup (1/2,1)$, we get a solution $f_0$ with a singularity of the type $(x-1/2)\log(|x-1/2|)$ from both sides of $x=1/2$ and satisfying

\begin{equation}\label{Wrosk_cte}
W[f_0,g_1](x)=W[f_0,g_1](1/2)=1, \quad \forall x\in[0,1].
\end{equation}

\subsubsection{Solving the equation \eqref{f_leftconm=1}-\eqref{f_rightconm=1}-\eqref{bc:fconm=1}.}

Now we consider the equation
\begin{align}
f''(x)-\left(\frac{2}{x^2-\tilde{\mu}^2}+(1-\ep)^2\right)f(x)&=0,\qquad \text{for } 0\leq x<1/2, \label{principalep} \\
f''(x)-\left(\frac{2}{x^2-\tilde{\nu}^2}+(1-\e)^2  \right)f(x)&=0, \qquad \text{for }1/2<x\leq 1, \label{rightprincipalep}
\end{align}
with
\[
f(0)=0=f(1), \quad \text{and} \quad f(1/2^-)=1=f(1/2^+).
\]
Recall that $$\tilde{\mu}=\frac{1}{2}+\tilde{\mu}_1(\ep)\ep=\frac{1}{2}+\left(\frac{1}{2}+\mu_1\right)\ep +\ep^2\log^2(\ep)\tilde{\mu}_2(\ep), \qquad \text{where} \quad -\frac{1}{2}<\mu_1<\frac{1}{2}$$
in such a way that $\tilde{\mu}>\frac{1}{2}$ and $\tilde{\nu}<\tfrac{1}{2}.$

Obtaining either, a solution for \eqref{principalep} with $f^\sharp(0)=0$ and $f^\sharp(\tilde{\mu})=1$ or a solution for \eqref{rightprincipalep} with $f^\sharp(1)=0$ and $f^\sharp(\tilde{\nu})=1$,
follows the same steps than the construction of $f_0$ in Section \ref{gg1}. In the first case, we have that the analogous of Lemma \ref{positivog} is the following result.

\begin{Lemma}\label{positivofm=1} Let $f$ be a solution of equation \eqref{principalep} with $f(0)=0$ and $f'(0)>0$ (or $f'(0)<0$). Then $f(x)>0$ (or $f(x)<0$) for $0<x\leq \tilde{\mu}$.
\end{Lemma}
\begin{proof}
A more general proof of this fact has been postponed to Section \ref{s:dimkernel}, see Lemma \ref{l:positivityfn}.
\end{proof}

In particular, this Lemma also shows that $f^\sharp(1/2)\neq 0$ and then,  to get a solution of \eqref{principalep} for $x\in[0,1/2)$ with the boundary conditions $f(0)=0$ and $f(1/2^-)=1$,
we just take $f(x)=f^\sharp(x)/f^\sharp(1/2).$ Uniqueness is obtained from Lemma \ref{positivofm=1}.

Similarly, we can solve \eqref{rightprincipalep} for $x\in(1/2,1]$ with boundary conditions $f(1)=0$ and $f(1/2^+)=1$ (see Lemma \ref{l:positivityfnright}).

In fact, we find a solution $f$ that can be written as
%\begin{align*}
%f(x)=\sum_{k=1} f^{(k)}\log(|\tilde{\mu}-x|)(x-\tilde{\mu})^k+\sum_{k=0}^\infty \overline{f}^{(k)}(x-\tilde{\mu})^k,
%\end{align*}
\[
f(x)=\begin{cases}
\sum_{k=1} f_1^{(k)}\log(|\tilde{\mu}-x|)(x-\tilde{\mu})^k+\sum_{k=0}^\infty \overline{f}_1^{(k)}(x-\tilde{\mu})^k, \qquad 0\leq x < 1/2,\\
\sum_{k=1} f_2^{(k)}\log(|\tilde{\nu}-x|)(x-\tilde{\nu})^k+\sum_{k=0}^\infty \overline{f}_2^{(k)}(x-\tilde{\nu})^k, \qquad 1/2< x \leq 1,
\end{cases}
\]
where $f_i^{(k)}$, $\overline{f}_i^{(k)}\in \R$.

\begin{comment}
This series can be also written as
\begin{align*}
f(x)=\begin{cases}
1+\sum_{k=1} \mathfrak{f}_1^{(k)}\log(|x-1/2|)(x-1/2)^k+\sum_{k=1}^\infty \overline{\mathfrak{f}}_1^{(k)}(x-1/2)^k,\qquad 0\leq x < 1/2,\\
1+\sum_{k=1} \mathfrak{f}_2^{(k)}\log(|x-1/2|)(x-1/2)^k+\sum_{k=1}^\infty \overline{\mathfrak{f}}_2^{(k)}(x-1/2)^k,\qquad 1/2< x \leq 1,
\end{cases}
\end{align*}
where $\mathfrak{f}_i^{(k)}$, $\overline{\mathfrak{f}}_i^{(k)}\in \R$ and all the series converge for $x\in (-1/2,3/2)$.
\end{comment}

\subsection{Finding a solution in the kernel}

In the following section we will study in detail the difference between the solution of the original system \eqref{f_leftconm=1}-\eqref{f_rightconm=1}-\eqref{bc:fconm=1} and the solution of the limit system \eqref{g}-\eqref{bc:g} obtained by taking $\ep \to 0$. Since the calculations are completely the same for one interval as for the other, we will focus for simplicity on the interval $(0,1/2).$

\subsubsection{The difference on $(0,1/2)$}

Taking $d\equiv f-f_{0}$, with $f$ solving \eqref{principalep} and $f_0$ solving \eqref{principal}, satisfies
\begin{align*}
d''(x)-\left(\frac{2}{x^2-\left(1/2\right)^2}+1\right)d(x)&
=\left(\frac{2}{x^2-\tilde{\mu}^2}-\frac{2}{x^2-(1/2)^2}
+\left((1-\ep)^2-1\right)\right)f(x)\\
&= 2\tilde{\mu}_1(\ep)\ep \frac{\tilde{\mu}+1/2}{(x^2-\tilde{\mu}^2)(x^2-(1/2)^2)}f(x)
+\left((1-\ep)^2-1\right)f(x)\\
&\equiv \ep E(x),
\end{align*}
for $0< x < 1/2$ with boundary conditions
\[
d(0)=0,\qquad d(1/2)=0.
\]
First we note that, from Lemma \ref{positivog}, if $f^\sharp $ solves the homogeneous equation
\begin{align*}
(f^\sharp)''(x)-\left(\frac{2}{x^2-(1/2)^2}+1\right)f^\sharp(x)&=0,\quad 0< x<1/2\\
f^\sharp(0)&=0,\\
f^\sharp(1/2)&=0,
\end{align*}
then $f^\sharp$ is the trivial solution, i.e., $f^\sharp(x)=0$  in all $0< x<1/2$. Thus the method of the variation of constants yields
\begin{align}\label{diferencia}
d(x)&=\ep f_0(x)\int_{x}^\frac{1}{2}\frac{g_1(w)}{W[f_0,g_1](w)}E(w)dw +\ep g_1(x)\int_{0}^x\frac{f_0(w)}{W[f_0,g_1](w)}E(w)dw.
\end{align}
with $f_0(x)$ and $g_1(x)$ as in  section \ref{gg1}. Recall (see \eqref{Wrosk_cte})   that $W[f_0,g_1](w)=1$ for all $w\in[0,1]$. Then, the above expression reduces to
\begin{align*}
d(x)&=\ep f_0(x)\int_{x}^\frac{1}{2} g_1(w) E(w)dw +\ep g_1(x)\int_{0}^x f_0(w) E(w)dw\\
&\equiv  d_1(x)+d_2(x).
\end{align*}

Recall that we have assumed that
\begin{equation}\label{tildemuansatz}
\tilde{\mu}=\frac{1}{2}+\tilde{\mu}_1(\ep)\ep=\frac{1}{2}+\left(\frac{1}{2}+\mu_1\right)\ep +\tilde{\mu}_2(\ep) \ep^2 \log^2(\ep),
\end{equation}
where $ -\frac{1}{2}<\mu_1<\frac{1}{2}$ and $\tilde{\mu}_2(\ep)=O(1)$ in terms of parameter $\ep$. We will assume that $|\tilde{\mu}_2|\leq \tilde{\mu}_{2,max}$ for some $\tilde{\mu}_{2,max}>0$ independent of $\ep$.

\begin{Lemma} \label{Linfty} Let $\tilde{\mu}$ be as in \eqref{tildemuansatz} with $\ep$ small enough such that
\begin{align*}
\left(\frac{1}{2}+\mu_1\right)+ \tilde{\mu}_2(\ep)\ep\log^2(\ep)=\tilde{\mu}_1(\ep)\geq \tilde{\mu}_{1,min}>0, \qquad \text{where} \quad -1/2<\mu_1<1/2.
\end{align*}
Then, we have
\begin{align*}
||d||_{L^\infty} \leq C\ep\log(1/\ep),
\end{align*}
where $C$ depends on  $||f||_{L^\infty}$, $||f_0||_{L^\infty}$, $||g_1||_{L^\infty}$, $||g'_1||_{L^\infty}$, $\mu_1$, $\tilde{\mu}_{2,max}$ and $\tilde{\mu}_{1,min}$, but it does not depend on $\ep$.
\end{Lemma}
\begin{proof}
We bound $||d_1||_{L^\infty}$ and $||d_2||_{L^\infty}$ separately. For $d_1$, we have that
\[
d_1(x)= \ep f_0(x)\int_{x}^\frac{1}{2} g_1(w) \left(E_1(w)+ E_2(w)\right)dw,
\]
with
\[
E_1(x)\equiv 2\tilde{\mu}_1(\ep) \frac{\tilde{\mu}+1/2}{(x^2-\tilde{\mu}^2)(x^2-(1/2)^2)}f(x),
\qquad E_2(x)\equiv -(2-\ep)f(x).
\]
We can directly estimate
\begin{align}\label{aux:d12}
\left|\ep f_0(x)\int_{x}^\frac{1}{2} g_1(w) E_2(w)dw\right|
 \leq \ep \, C  ||f_0||_{L^\infty}||f||_{L^\infty}||g_1||_{L^\infty}.
\end{align}
In addition, we have
\begin{align}\label{aux:d11}
\left|\ep f_0(x)\int_{x}^\frac{1}{2} g_1(w) E_1(w)dw\right|&\leq \ep\, C  ||f_0||_{L^\infty}||f||_{L^\infty}\int_{x}^\frac{1}{2}
\frac{|g_1(w)|}{(\tilde{\mu}-w)(1/2-w)}dw\\ &\leq \ep\,|\log(\tilde{\mu}_1(\ep) \ep)| \, C ||f_0||_{L^\infty}||f||_{L^\infty}||g_1'||_{L^\infty}, \nonumber
\end{align}
where in the last step we have used the fact that
$$|g_1(w)|=|g_1(w)-g_1(1/2)|\leq ||g_1'||_{L^\infty}(1/2-w).$$
Therefore, combining \eqref{aux:d11} and \eqref{aux:d12} we get
\begin{align*}
||d_1||_{L^\infty}& \leq C  ||g_1'||_{L^\infty}||f||_{L^\infty}||f_0||_{L^\infty} \ep \log(1/\ep)\\
&\leq C \ep \log(1/\ep).
\end{align*}

For $d_2$ we have that
\begin{align*}
d_2=\ep g_1(x)\int_0^x f_0(w) \left(E_1(w)+ E_2(w)\right)dw,
\end{align*}
and again we can directly estimate
\begin{align}\label{aux:d22}
\left|\ep g_1(x)\int_{0}^x f_0(w) E_2(w)dw\right|\leq \ep\, C  ||f_0||_{L^\infty}||g_1||_{L^\infty}||f||_{L^\infty}.
\end{align}
For the last term, using that $\tilde{\mu}>1/2$, we get
\begin{align}\label{aux:d21}
&\left|\ep g_1(x)\int_{0}^x f_0(w) E_1(w)dw\right|\leq \ep C ||f_0||_{L^\infty}||f||_{L^\infty} |g_1(x)|\int_{0}^x \frac{1}{(\tilde{\mu}-w)(1/2-w)}dw\\ \nonumber
&\leq \ep C ||f_0||_{L^\infty}||f||_{L^\infty} |g_1(x)|\int_{0}^x \frac{1}{(1/2-w)^2}dw\\ \nonumber
&\leq \ep C ||f_0||_{L^\infty}||f||_{L^\infty} \frac{|g_1(x)|}{1/2-x}\\ \nonumber
&\leq \ep C ||f_0||_{L^\infty}||f||_{L^\infty} ||g'_1||_{L^\infty}. \nonumber
\end{align}
Therefore, combining \eqref{aux:d21} and \eqref{aux:d22} we get
\begin{align*}
||d_2||_{L^\infty}&\leq C \ep.
\end{align*}

\end{proof}

\begin{Lemma} \label{firstLemma}Let $\tilde{\mu}$ be as Lemma \ref{Linfty}.
Then, we have
\begin{align*}
\left|\int_{0}^\frac{1}{2}\frac{\sinh((1-\ep)w)d(w)}{(w+\tilde{\mu})(w-\tilde{\mu})}dw\right|\leq C\ep \log^2(\ep),
\end{align*}
where $C$ depends on  $||f||_{L^\infty}$, $||f_0||_{L^\infty}$, $||g_1||_{L^\infty}$, $||g'_1||_{L^\infty}$, $\mu_1$, $\tilde{\mu}_{2,max}$ and $\tilde{\mu}_{1,min}$, but it does not depend on $\ep$.
\end{Lemma}
\begin{proof}
Recall that $d(x)=d_1(x)+d_2(x)$ with
\[
d_1(x)=\ep f_0(x)\int_{x}^\frac{1}{2} g_1(w) \left(E_1(w)+ E_2(w)\right)dw=:d_{11}(x)+d_{12}(x),
\]
and
\[
d_2(x)= \ep g_1(x)\int_0^x f_0(w)\left(E_1(w)+ E_2(w)\right)dw=:d_{21}(x)+d_{22}(x).
\]
We will now study each of the terms separately.

\underline{Integral term with $d_{12}$:} In first place, we recall that
\begin{align*}
\left|\int_{w}^\frac{1}{2} g_1(z)E_2(z)dz\right|\leq C ||f||_{L^\infty}||g_1||_{L^\infty}(1/2-w),
\end{align*}
for $0\leq w \leq 1/2$ and where the $L^\infty$ norm is taken in $[0,\frac{1}{2}]$.
Therefore
\begin{align*}
||d_{12}||_{L^\infty}&\leq \ep\left|\int_{0}^\frac{1}{2}\frac{\sinh((1-\ep)w)}{w^2-\tilde{\mu}^2}f_0(w)
\int_{w}^\frac{1}{2}g_1(z) E_2(z)dzdw\right|\\
&\leq \ep C ||f||_{L^\infty}||g_1||_{L^\infty}||f_0||_{L^\infty}\int_{0}^\frac{1}{2}\frac{1/2-w}{\tilde{\mu}-w}dw\\
&\leq \ep \frac{C}{2} ||f||_{L^\infty}||g_1||_{L^\infty}||f_0||_{L^\infty}.
\end{align*}
Here, we have used $\tilde{\mu}>1/2$ and the fact that
\begin{align*}
\frac{1/2-w}{\tilde{\mu}-w}\leq 1, \quad  \forall  w\in[0, 1/2],
\end{align*}
for
\begin{align*}
\left(\frac{1}{2}+\mu_1\right)+\ep\log^2(\ep)\tilde{\mu}_2(\ep)=\tilde{\mu}_1(\ep)>0.
\end{align*}

\underline{Integral term with $d_{11}$:} Secondly, we use that
\begin{align*}
\left|\int_{w}^\frac{1}{2} g_1(z) E_1(z)dz\right|
\leq C  \tilde{\mu}_1(\ep)||f||_{L^\infty}\int_{w}^\frac{1}{2} \frac{|g_1(z)|}{(\tilde{\mu}-z)(1/2-z)} dz.
\end{align*}
Since $0<\tilde{\mu}_1(\ep)<2$ and
$$|g_1(z)|\leq ||g_1'||_{L^\infty}|z-1/2|,$$ we find that
\begin{align*}
\left|\int_{w}^\frac{1}{2} g_1(z) E_1(z)dz\right|\leq C  ||g_1'||_{L^\infty}||f||_{L^\infty}\int_{w}^\frac{1}{2}\frac{dz}{\tilde{\mu}-z}\leq C  ||g_1'||_{L^\infty}
||f||_{L^\infty}\log\left(\frac{\tilde{\mu}-w}{\tilde{\mu}_1(\ep)\ep}\right).
\end{align*}
Therefore
\begin{align*}
||d_{11}||_{L^\infty}&\leq \ep\left|\int_{0}^\frac{1}{2}\frac{\sinh((1-\ep)w)}{w^2-\tilde{\mu}^2}f_0(w)
\int_{w}^\frac{1}{2} g_1(z) E_1(z)dzdw\right|\\
&\leq \ep C  ||g_1'||_{L^\infty} ||f||_{L^\infty}||f_0||_{L^\infty}\frac{1}{2}\log^2\left(\frac{\tilde{\mu_1}(\ep)\ep}{\tilde{\mu}}\right),
\end{align*}
where in the last step we used
\begin{align}\label{intlog^2ep}
\int_{0}^\frac{1}{2}\frac{\log(\frac{\tilde{\mu}-w}{\tilde{\mu}_1(\ep)\ep})}{\tilde{\mu}-w}dw=
\frac{1}{2}\log^2\left(\frac{\tilde{\mu_1}(\ep)\ep}{\tilde{\mu}}\right).
\end{align}

\underline{Integral terms $d_{21}$ and $d_{22}$:} Next, we have to control
\begin{align*}
\left|\int_{0}^\frac{1}{2}\frac{\sinh((1-\ep)w)d_2(w)}{w^2-\tilde{\mu^2}}dw\right|
=\ep\left|\int_{0}^\frac{1}{2}\frac{\sinh((1-\ep)w)}{w^2-\tilde{\mu}^2}g_1(w)
\int_{0}^w f_0(z) E(z)dzdw\right|.
\end{align*}
%In addition, as $\tilde{\mu}>1/2$ we notice that
%\begin{align}\label{notice}
%\left|\frac{g_1(w)}{\tilde{\mu}-w}\right|=\left|\frac{g_1(w)}{1/2-w}\right| \left|\frac{1/2-w}{\tilde{\mu}-w}\right|\leq ||g_1'||_{L^\infty}, \qquad \forall w\in[0,1/2].\end{align}
Then, for $E_2$ we bound
\begin{align*}
\left|\int_{0}^w f_0(z) E_2(z)dz\right|\leq C   ||f_0||_{L^\infty}||f||_{L^\infty},
\end{align*}
and for the term coming from $E_1$ we have that
\begin{align*}
\left|\int_{0}^w f_0(z) E_1(z)dz\right|&\leq C  ||f_0||_{L^\infty}||f||_{L^\infty} \int_0^w \frac{1}{(\tilde{\mu}-z)(1/2-z)} dz\\
&\leq C  ||f_0||_{L^\infty}||f||_{L^\infty} \int_0^w \frac{1}{(1/2-z)^2} dz\leq  C  ||f_0||_{L^\infty}||f||_{L^\infty}(1/2-w)^{-1}.
\end{align*}
Therefore
\begin{align*}
||d_{2}||_{L^\infty}&\leq \ep\left|\int_{0}^\frac{1}{2}\frac{\sinh((1-\ep)w)g_1(w)}{w^2-\tilde{\mu}^2}\int_{0}^w f_0(z) E(z)dz dw\right|
\\&\leq\ep  C  ||f_0||_{L^\infty}||f||_{L^\infty} \left|\int_{0}^\frac{1}{2}\frac{1}{\tilde{\mu}-w}\frac{g_1(w)}{1/2-w}dw \right|  \\
&\leq C  ||g_1'||_{L^\infty} ||f||_{L^\infty}||f_0||_{L^\infty}\ep\,\left|\log\left(\frac{\tilde{\mu_1}(\ep)\ep}{\tilde{\mu}}\right)\right|.
\end{align*}
\end{proof}

Below is a more refined version of the above result. Which will be very useful for determining $\tilde{\mu}_2(\ep)$ later through a fixed point argument.

\begin{Lemma}\label{mainLemma} Let $\tilde{\mu}$ be as in \eqref{tildemuansatz}. That is, $\tilde{\mu}=\tfrac{1}{2}+\tilde{\mu}_1(\ep)\ep$. Let assume that
\begin{align*}
\left(\frac{1}{2}+\mu_1\right)+ \tilde{\mu}_2(\ep)\ep\log^2(\ep)=\tilde{\mu}_1(\ep)>0,\qquad \text{and} \quad -1/2<\mu_1<1/2. \end{align*}
Then the integral
\begin{align*}
\int_{0}^\frac{1}{2}\frac{\sinh((1-\ep)w)d(w)}{w^2-\tilde{\mu}^2}dw = M[f_0,\mu_1]+R[f_0,\mu_1;f,\tilde{\mu}_2],
\end{align*}
 where $M[f_0,\mu_1]$, $R[f_0,\mu_1;f,\tilde{\mu}_2]\in \R$ and
 \begin{align*}
M[f_0,\mu_1]&=O(\ep\log^2(\ep))\\
R[f_0,\mu_1;f,\tilde{\mu}_2]&=o(\ep\log^2(\ep)).
 \end{align*}
\begin{remark}
We emphasize that $M[f_0,\mu_1]$ does not depend on either $f$ or $\tilde{\mu}_2(\ep)$.
\end{remark}
\end{Lemma}

\begin{proof}
Since the only term that yields a factor $\ep\log^2(\ep)$ is that of $d_1$ coming from $E_1$, we only have to consider the term coming from $d_{11}$. That is, the integral
\begin{align*}
&\ep\int_{0}^\frac{1}{2}\frac{\sinh((1-\ep)w)}{w^2-\tilde{\mu}^2}f_0(w)
\int_{w}^\frac{1}{2}g_1(z) E_1(z) dzdw,
\end{align*}
where
\begin{align*}
E_1(z)=2\tilde{\mu}_1(\ep)\frac{\tilde{\mu}+1/2}{(\tilde{\mu}^2-z^2)(1/4-z^2)}f(z).
\end{align*}

We will split $E_1(z)$, using $d=f-f_0$, as follow
\begin{align*}
E_1(z)&=2\tilde{\mu}_1(\ep)\frac{\tilde{\mu}+1/2}{(\tilde{\mu}^2-z^2)(1/4-z^2)}d(z)+2\tilde{\mu}_1(\ep)\frac{\tilde{\mu}+1/2}{(\tilde{\mu}^2-z^2)(1/4-z^2)}f_0(z)\\
&\equiv E_{11}(z)+E_{12}(z).
\end{align*}

To bound
\begin{align*}
&\ep\int_{0}^\frac{1}{2}\frac{\sinh((1-\ep)w)}{w^2-\tilde{\mu}^2}f_0(w)
\int_{w}^\frac{1}{2}g_1(z) E_{11}(z)dzdw,
\end{align*}
we proceed as in Lemma \ref{firstLemma} to get
\begin{align*}
\leq  C  ||g_1'||_{L^\infty}||d||_{L^\infty}||f_0||_{L^\infty}\ep\log^2\left(\frac{\tilde{\mu_1}(\ep)\ep}{\tilde{\mu}}\right).
\end{align*}
From Lemma \ref{Linfty} we know that $||d||_{L^\infty}\leq C \ep \log(1/\ep)$. Thus, the term coming from $E_{11}$ behaves like $\ep^2\log^3(\ep)=o(\ep\log^2(\ep))$. As a consequence, we just have to take care about
\begin{align*}
&\ep\int_{0}^\frac{1}{2}\frac{\sinh((1-\ep)w)}{w^2-\tilde{\mu}^2}f_0(w)
\int_{w}^\frac{1}{2} g_1(z) E_{12}(z)dzdw\\
=&2\tilde{\mu}_1(\ep)\ep\int_{0}^\frac{1}{2}\frac{\sinh((1-\ep)w)}{w^2-\tilde{\mu}^2}f_0(w)
\int_{w}^\frac{1}{2} g_1(z) \frac{\tilde{\mu}+1/2}{(\tilde{\mu}^2-z^2)(1/4-z^2)}f_0(z)dzdw,\\
=&2\left(\tfrac{1}{2}+\mu_1\right)\ep\int_{0}^\frac{1}{2}\frac{\sinh((1-\ep)w)}{w^2-\tilde{\mu}^2}f_0(w)
\int_{w}^\frac{1}{2} g_1(z) \frac{\tilde{\mu}+1/2}{(\tilde{\mu}^2-z^2)(1/4-z^2)}f_0(z)dzdw\\
+&o(\ep\log^2(\ep)).
\end{align*}
In addition, adding and subtracting  appropriate terms we obtain
\begin{multline}\label{simplifactionLemma}
= 2\left(\tfrac{1}{2}+\mu_1\right)\ep\int_{0}^\frac{1}{2}\frac{\sinh((1-\ep)w)}{w^2-(1/2+(1/2+\mu_1)\ep)^2}f_0(w)\\
 \times \int_{w}^\frac{1}{2} g_1(z) \frac{f_0(z)}{(1/2+(1/2+\mu_1)\ep+z)(1/2+(1/2+\mu_1)\ep-z)(1/4-z^2)}dzdw+o(\ep\log^2(\ep)).
\end{multline}
More specifically, we expand below the calculations to get \eqref{simplifactionLemma}. First we note that
\begin{align*}
&2\left(\tfrac{1}{2}+\mu_1\right)\ep\int_{0}^\frac{1}{2}\frac{\sinh((1-\ep)w)}{w^2-\tilde{\mu}^2}f_0(w)
\int_{w}^\frac{1}{2} g_1(z) \frac{\tilde{\mu}+1/2}{(\tilde{\mu}^2-z^2)(1/4-z^2)}f_0(z)dzdw\\
&=2\left(\tfrac{1}{2}+\mu_1\right)\ep\int_{0}^\frac{1}{2}\sinh((1-\ep)w)\frac{1}{w+\tilde{\mu}}
\left(\frac{1}{w-\tilde{\mu}}-\frac{1}{w-1/2-(1/2+\mu_1)\ep}\right)f_0(w)\\ &
\hspace{2cm}\times \int_{w}^\frac{1}{2} g_1(z) \frac{\tilde{\mu}+1/2}{(\tilde{\mu}^2-z^2)(1/4-z^2)}f_0(z)dzdw\\
&+2\left(\tfrac{1}{2}+\mu_1\right)\ep\int_{0}^\frac{1}{2}\sinh((1-\ep)w)\frac{1}{w+\tilde{\mu}}
\frac{1}{w-1/2-(1/2+\mu_1)\ep}f_0(w)\\ &
\hspace{2cm}\times \int_{w}^\frac{1}{2} g_1(z) \frac{\tilde{\mu}+1/2}{(\tilde{\mu}^2-z^2)(1/4-z^2)}f_0(z)dzdw\\
&\equiv K_1+K_2.
\end{align*}

Now, we turn our attention to the first summand above. For the difference appearing in $K_1$, it is straightforward to check that
$$\left(\frac{1}{w-\tilde{\mu}}-\frac{1}{w-1/2-(1/2+\mu_1)\ep}\right)=
\frac{\tilde{\mu}_2(\ep) \ep^2\log^2(\ep)}{(w-\tilde{\mu})(w-1/2-(1/2+\mu_1)\ep)}.$$
Thus,
\begin{align*}
K_1&=2\left(\tfrac{1}{2}+\mu_1\right)\tilde{\mu}_2(\ep) \ep^3 \log^2(\ep)\int_{0}^\frac{1}{2}\sinh((1-\ep)w)\frac{1}{w+\tilde{\mu}}
\frac{1}{(w-\tilde{\mu})(w-1/2-(1/2+\mu_1)\ep)}f_0(w)\\ &
\, \hspace{2cm}\times \int_{w}^\frac{1}{2} g_1(z) \frac{\tilde{\mu}+1/2}{(\tilde{\mu}^2-z^2)(1/4-z^2)}f_0(z)dzdw,
\end{align*}
and proceeding as we did before in Lemma \ref{firstLemma}, we arrive to
\begin{align*}
K_1 &\leq C \ep^3 \log^2(\ep)\int_0^{\frac{1}{2}} \frac{1}{(w-\tilde{\mu})(w-1/2-(1/2+\mu_1)\ep)} \log\left(\frac{\tilde{\mu}-w}{\tilde{\mu}_1(\ep)\ep}\right)dw\\
&\leq  C\ep^2\log^3(\ep)=o(\ep\log^2(\ep)).
\end{align*}
where in the last step we have used a slight modification of \eqref{intlog^2ep}.

To handle the remaining term $K_2$,  for brevity, we just focus our attention in the most delicate one. This is the term that appears when we replace $(w+\tilde{\mu})^{-1}$ and $(z+\tilde{\mu})^{-1}$ by $(w+1/2+(\mu_1+1/2)\ep)^{-1}$ and  $(z+1/2+(\mu_1+1/2)\ep)^{-1}$ respectively. That is,
\begin{align*}
&2\left(\tfrac{1}{2}+\mu_1\right)\ep\int_{0}^\frac{1}{2}
\frac{\sinh((1-\ep)w)}{w^2-(1/2+(1/2+\mu_1)\ep)^2}f_0(w)\\
&\hspace{2cm}\times \int_{w}^\frac{1}{2} g_1(z) \frac{1}{(z+1/2+(1/2+\mu_1)\ep)(\tilde{\mu}-z)(1/4-z^2)}f_0(z)dzdw.
\end{align*}
One more time, adding and subtracting appropriate terms we obtain
\begin{align*}
&2\left(\tfrac{1}{2}+\mu_1\right)\ep\int_{0}^\frac{1}{2}
\frac{\sinh((1-\ep)w)}{w^2-(1/2+(1/2+\mu_1)\ep)^2}f_0(w)\\
&\hspace{2 cm}\times \int_{w}^\frac{1}{2} g_1(z) \frac{1}{z+1/2+(1/2+\mu_1)\ep} \left(\frac{1}{\tilde{\mu}-z}-\frac{1}{(1/2+(\mu_1+1/2)\ep)-z} \right) \frac{1}{1/4-z^2}f_0(z)dzdw\\
&+ 2\left(\tfrac{1}{2}+\mu_1\right)\ep\int_{0}^\frac{1}{2}
\frac{\sinh((1-\ep)w)}{w^2-(1/2+(1/2+\mu_1)\ep)^2}f_0(w)\\
&\hspace{2 cm}\times \int_{w}^\frac{1}{2} g_1(z) \frac{1}{(z+1/2+(1/2+\mu_1)\ep)(1/2+(\mu_1+1/2)\ep)-z) (1/4-z^2)} f_0(z)dzdw\\
&\equiv K_{21}+K_{22}.
\end{align*}
Then, the second summand $K_{22}$ is exactly the principal term of \eqref{simplifactionLemma}. To conclude we only need to check that $K_{21}=o(\ep\log^2(\ep))$. To do so, proceeding in the same way as before, for the difference appearing in $K_{21}$, we have
\[
\frac{1}{\tilde{\mu}-z}-\frac{1}{1/2+(\mu_1+1/2)\ep-z}=-\frac{\tilde{\mu}_2(\ep)\ep^2 \log^2(\ep)}{(\tilde{\mu}-z)(1/2+(\mu_1+1/2)\ep-z)},
\]
and consequently, we get
\begin{align*}
K_{21} &= 2\left(\tfrac{1}{2}+\mu_1\right)\tilde{\mu}_2(\ep)\ep^3\log^2(\ep)\int_{0}^\frac{1}{2}
\frac{\sinh((1-\ep)w)}{w^2-(1/2+(1/2+\mu_1)\ep)^2}f_0(w)\\
&\qquad \times \int_{w}^\frac{1}{2} g_1(z) \frac{1}{(z^2-(1/2+(1/2+\mu_1)\ep)^2)}\frac{1}{(1/4-z^2)(\tilde{\mu}-z)}f_0(z)dzdw\\
&\leq C\ep^2\log^5(\ep)=o(\ep\log^2(\ep)).
\end{align*}
%And we have that
%\begin{align*}
%M[f_0,\mu_1]&\equiv
%2\left(\tfrac{1}{2}+\mu_1\right)\ep\int_{0}^\frac{1}{2}\frac{\sinh((1-\ep)w)}{w-1/2-(1/2+\mu_1)\ep}f_0(w)\\
%&\int_{w}^\frac{1}{2} g_1(z) \frac{f_0(z)}{(1/2+(1/2+\mu_1)\ep+z)(1/2+(1/2+\mu_1)\ep-z)(1/4-z^2)}dzdw.
%\end{align*}
\end{proof}

Proceeding in a completely similar way as we did before with the interval $[0,1/2)$ we can obtain the same type of result for the interval $(1/2,1]$ by obtaining the analogues of the Lemmas \ref{Linfty}, \ref{firstLemma} which allow us to conclude the following result. The only difference in this case is that we have to assure that $\tilde{\nu}>1/2$, i.e.,

\begin{align*}
\-\frac{1}{2}+\mu_1+ \tilde{\mu}_2(\ep)\ep\log^2(\ep)+\tilde{\mu}_1(\ep)^2\ep^2-\frac{\ep^2}{1-\ep}=\tilde{\mu}_1(\ep)\geq \tilde{\mu}_{1,min}>0, \qquad \text{where} \quad -1/2<\mu_1<1/2.
\end{align*}

\begin{Lemma}\label{mainLemmaright}Let $\tilde{\mu}$ be as in \eqref{tildemuansatz}. That is, $\tilde{\mu}=\tfrac{1}{2}+\tilde{\mu}_1(\ep)\ep$. Let assume that
\begin{align*}
\mu_1-1/2+ \tilde{\mu}_2(\ep)\ep\log^2(\ep)+\tilde{\mu}_1(\ep)^2\ep^2-\frac{\ep^2}{1-\ep}<0,\qquad \text{and} \quad -1/2<\mu_1<1/2. \end{align*}
Then the integral
\begin{align*}
\int_{\frac{1}{2}}^1 \frac{\sinh((1-\ep)(1-w))d(w)}{w^2-\tilde{\nu}^2}dw= M[f_0,\mu_1]+R[f_0,\mu_1;f,\tilde{\mu}_2],
\end{align*}
where $M[f_0,\mu_1]$, $R[f_0,\mu_1;f,\tilde{\mu}_2]\in \R$ and
 \begin{align*}
M[f_0,\mu_1]&=O(\ep\log^2(\ep))\\
R[f_0,\mu_1;f,\tilde{\mu}_2]&=o(\ep\log^2(\ep)).
 \end{align*}
We insist once again that $M[f_0,\mu_1]$ does not depend on either $f$ or $\tilde{\mu}_2(\ep)$
\end{Lemma}

\subsection{Solving the original system of integral equations}
Up to this point, we have learned how to handle the system of second order differential equations associated with the original system of integral equations. In the following we are going to see how to use this to actually solve our system of integral equations.

In this section we show the existence of a solution $(\tilde{\mu},\, F)$ to the equations \eqref{F1} and \eqref{F2} with $m=1$. We must solve, for $0\leq x<1/2$,

\begin{align}\label{F11}
&F(x)+\tilde{C}_1\sinh(1-(1-\ep)x)\int_{0}^{x}\frac{\sinh((1-\ep)w)F(w)}{w^2-\tilde{\mu}^2}dw\\
&+\tilde{C}_1\sinh((1-\ep)x)\int_{x}^\frac{1}{2} \frac{\sinh(1-(1-\ep)w)F(w)}{w^2-\tilde{\mu}^2}dw\nonumber\\
&+\tilde{C}_1\sinh((1-\ep)x)\int_{\frac{1}{2}}^1\frac{\sinh((1-\ep)(1-w))F(w)}{w^2-\tilde{\nu}^2}dw\nonumber=0,
\end{align}
and for $1/2<x\leq 1$,
\begin{align}\label{F21}
&F(x)+\tilde{C}_1 \sinh((1-\ep)(1-x))\int_{0}^\frac{1}{2} \frac{\sinh((1-\ep)w)F(w)}{w^2-\tilde{\mu}^2}dw\\
&+\tilde{C}_1 \sinh((1-\ep)(1-x))\int_{\frac{1}{2}}^x \frac{\sinh((1-\ep)w+\ep)F(w)}{w^2-\tilde{\nu}^2}dw\nonumber\\
&+\tilde{C}_1 \sinh((1-\ep)x+\ep)\int_{x}^1\frac{ \sinh((1-\ep)(1-w)) F(w)}{w^2-\tilde{\nu}^2}dw=0\nonumber.
\end{align}

Now, we take
\begin{align}
F(x)&=Af(x),\quad \text{for $0\leq x<1/2$,}\label{defF_left}\\
F(x)&=Bf(x),\quad \text{for $1/2<x\leq 1$,} \label{defF_right}
\end{align}
where $f$ solves
\begin{align*}
f''(x)-\left(\frac{2}{x^2-\tilde{\mu}^2}+(1-\ep)^2\right)f(x)&=0\quad \text{for $0\leq x<1/2$},\\
f(0)&=0,    \\
f(1/2^-)&=1,
\end{align*}
and
\begin{align*}
f''(x)-\left(\frac{2}{x^2-\tilde{\nu}^2}+(1-\ep)^2\right)f(x)&=0\quad \text{for $1/2<x\leq 1$},\\
f(1)&=0,\\
f(1/2^+)&=1,
\end{align*}
and where $A, B\in \mathbb{R}$. For this specific choice of $F$, given by \eqref{defF_left}-\eqref{defF_right}, we see that it is sufficient to solve \eqref{F11}  in $x=1/2^-$ and \eqref{F21} in  $x=1/2^+$. Indeed, the function
\begin{align*}
N(x)\equiv  A f(x)&+\tilde{C}_1\sinh(1-(1-\ep)x)\int_{0}^{x}\frac{\sinh((1-\ep)w)Af(w)}{w^2-\tilde{\mu}^2}dw\nonumber\\
&+\tilde{C}_1\sinh((1-\ep)x)\int_{x}^\frac{1}{2} \frac{\sinh(1-(1-\ep)w)Af(w)}{w^2-\tilde{\mu}^2}dw\nonumber\\
&+\tilde{C}_1\sinh((1-\ep)x)\int_{\frac{1}{2}}^1\frac{\sinh((1-\ep)(1-w))Bf(w)}{w^2-\tilde{\nu}^2}dw\nonumber,
\end{align*}
satisfies
\begin{align*}
N''(x)-(1-\ep)^2N(x)&=0,\quad 0< x<1/2\\
N(0)&=0.
\end{align*}
Therefore, if we impose $N(1/2^-)=0$ then $N(x)=0$ for all $[0,1/2]$ and we have solved \eqref{F11}.   The same argument applies in the interval $[1/2,1]$ simply by imposing $N(1/2^+)=0$.

Now, we impose the previous conditions $N(1/2^{-})=0=N(1/2^+)$. That is, taking limits we find
\begin{align}\label{matrix}
\left(\begin{array}{cc}1+\sinh(b)I_1 & \sinh(a)I_2\\ \sinh(a)I_1  &  1+\sinh(b)I_2 \end{array}\right)\left(\begin{array}{cc} A \\ B\end{array}\right)=0,
\end{align}
with
\begin{align*}
I_1&\equiv \tilde{C}_1\int_{0}^\frac{1}{2}\sinh((1-\ep)w)\frac{f(w)}{w^2-\tilde{\mu}^2}dw,\\
I_2&\equiv \tilde{C}_1\int_{\frac{1}{2}}^1\sinh((1-\ep)(1-w))\frac{f(w)}{w^2-\tilde{\nu}^2}dw.
\end{align*}

Thus, if $(A,B)\neq (0,0)$ we obtain that
\begin{align*}
\det\left(\begin{array}{cc}1+\sinh(b)I_1 & \sinh(a)I_2\\ \sinh(a)I_1  &  1+\sinh(b)I_2 \end{array}\right)=0.
\end{align*}
The above is equivalent to
\begin{align}\label{detsol}
1
+\sinh(b)\left(I_1+I_2\right)
+\left(\sinh^2(b)-\sinh^2(a)\right)I_1I_2 =0.
\end{align}
We understand this last expression as an equation for $\tilde{\mu}$ and introduce our ansatz \eqref{tildemuansatz}. That is,
\begin{align*}
\tilde{\mu}= \frac{1}{2}+\tilde{\mu}_1(\ep)\ep =\frac{1}{2}+\left(\frac{1}{2}+\mu_1\right)\ep +\tilde{\mu}_2(\ep)\ep^2\log^2(\ep).
\end{align*}
Until now the parameter $\mu_1\in(-1/2,1/2)$ has remained free. We will fix it now. Indeed, we take $\mu_1$ as the only solution of
\begin{align}\label{eculimit}
&1+C_1\sinh(1/2)\int_{0}^\frac{1}{2}\frac{\sinh(x)f_0(x)-\sinh(1/2)}{x^2-1/4}dx\\
&+C_1\sinh(1/2)
\int_{1/2}^1\frac{\sinh(1-x)f_0(x)-\sinh(1/2)}{x^2-1/4}dx
\nonumber\\ &+C_1\sinh^2(1/2)\log\left(\frac{1+2\mu_1}{1-2\mu_1}\right)-C_1\sinh^2(1/2)\log(3)=0\nonumber.
\end{align}

The reason for defining  $\mu_1$ in this precise way becomes clearer below.

Here, we recall that  $f_0$ solves
\begin{align*}
f_0''(x)-\left(\frac{2}{x^2-1/4}+1\right)f_0(x)=0, \quad \text{for }0\leq x<1/2 \text{ and } 1/2 <x\leq 1,
\end{align*}
with
\[
f_0(0)=0=f_0(1), \quad \text{and} \quad f_0(1/2^-)=1=f_0(1/2^+).
\]
Notice that the integrals in \eqref{eculimit} are finite ($f_0$ behaves as $1+ (x-1/2)\log(|x-1/2|)$ around $x=1/2$). Therefore, there is only one $\mu_1\in(-1/2,1/2)$ that solves \eqref{eculimit}. Then,   $\tilde{\mu}_2(\ep)$ is our unknown quantity.

We will look for an equation of the type
\begin{align*}
\tilde{\mu}_2(\ep)=M[f_0,\mu_1,\ep]+ R[f_0,\mu_1, f,\tilde{\mu}_2,\ep],
\end{align*}
where $M$ is $O(1)$ and $R$ is $o(1)$ in $\ep$.

Notice that $f$ depends on $\tilde{\mu}_2(\ep)$. We will make this dependence explicit later when we needed it.

\subsubsection{The main term $I_1+I_2$}

We know study
\begin{align}
\label{I1masI2entreC}\frac{I_1+I_2}{\tilde{C}_1}&=\int_{0}^\frac{1}{2}\frac{\sinh((1-\ep)w)f(w)}{w^2-\tilde{\mu}^2}dw
+\int_{1/2}^1\frac{\sinh((1-\ep)(1-w))f(w)}{w^2-\tilde{\nu}^2}dw\\
&=\int_{0}^\frac{1}{2}\frac{\sinh((1-\ep)w)f(w)-\sinh((1-\ep)/2)}{w^2-\tilde{\mu}^2}dw\nonumber \\
&\quad +\int_{1/2}^1\frac{\sinh((1-\ep)(1-w))f(w)-\sinh((1-\ep)/2)}{w^2-\tilde{\nu}^2}dw\nonumber\\
&\quad +\sinh((1-\ep)/2)\left(\int_{0}^\frac{1}{2}\frac{1}{w^2-\tilde{\mu}^2}dw
+\int_{\frac{1}{2}}^1\frac{1}{w^2-\tilde{\nu}^2}dw\right)\nonumber\\
&\equiv J_1+J_2 \nonumber\\
&\quad +\sinh((1-\ep)/2)\left(\frac{1}{2\tilde{\mu}}\log\left(\frac{\tilde{\mu}-1/2}{\tilde{\mu}+1/2}\right)+\frac{1}{2\tilde{\nu}}
\log\left(\frac{(1-\tilde{\nu})(1/2+\tilde{\nu})}{(1+\tilde{\nu})(1/2-\tilde{\nu})}\right)\right).\nonumber
\end{align}
Here we note that the last factor of \eqref{I1masI2entreC} can be more conveniently decomposed as
\begin{align*}
&\frac{1}{2\tilde{\mu}}\log\left(\frac{\tilde{\mu}-1/2}{\tilde{\mu}+1/2}\right)+\frac{1}{2\tilde{\nu}}
\log\left(\frac{(1-\tilde{\nu})(1/2+\tilde{\nu})}{(1+\tilde{\nu})(1/2-\tilde{\nu})}\right)\\
&=\frac{1}{2\tilde{\mu}}\log(\tilde{\mu}-1/2)+\frac{1}{2\tilde{\nu}}\log\left(\frac{1}{1/2-\tilde{\nu}}\right)\\
&\quad +\frac{1}{2\tilde{\mu}}\log\left(\frac{1}{\tilde{\mu}+1/2}\right)
+\frac{1}{2\tilde{\nu}}\log\left(\frac{(1-\tilde{\nu})(1/2+\tilde{\nu})}{1+\tilde{\nu}}\right)\\
&=\frac{1}{2\tilde{\mu}}\log\left(\frac{\tilde{\mu}-1/2}{1/2-\tilde{\nu}}\right) +\left(\frac{1}{2\tilde{\nu}}-\frac{1}{2\tilde{\mu}}\right)\log\left(\frac{1}{1/2-\tilde{\nu}}\right)\\
&\quad +\frac{1}{2\tilde{\mu}}\log\left(\frac{1}{\tilde{\mu}+1/2}\right) +\frac{1}{2\tilde{\nu}}\log\left(\frac{(1-\tilde{\nu})(1/2+\tilde{\nu})}{1+\tilde{\nu}}\right).
\end{align*}

These terms will be analyzed in the following Lemma.

\begin{Lemma} \label{tedious}Let $|\tilde{\mu}_2(\ep)|<\tilde{\mu}_{2,\text{max}}<\infty$ and $-1/2<\mu_1<1/2$. Then  the following estimates hold:
\begin{enumerate}

\item \begin{align*}&\frac{1}{2\tilde{\mu}}\log\left(\frac{\tilde{\mu}-1/2}{1/2-\tilde{\nu}}\right)\\
&=\log\left(\frac{1/2+\mu_1}{1/2-\mu_1}\right)+\frac{\tilde{\mu}_2}{1/4-\mu_1^2}\ep\log^2(\ep)+o(\ep\log^2(\ep)),\end{align*}

\item \begin{align*}
\left(\frac{1}{2\tilde{\nu}}-\frac{1}{2\tilde{\mu}}\right)\log\left(\frac{1}{1/2-\tilde{\nu}}\right)=O(\ep\log(\ep))=o(\ep\log^2(\ep)),
\end{align*}

\item
\begin{align*}
\frac{1}{2\tilde{\nu}}\log\left(\frac{(1-\tilde{\nu})(1/2+\tilde{\nu})}{1+\tilde{\nu}}\right)=-\log(3)+O(\ep)=-\log(3)+o(\ep\log^2(\ep)),\end{align*}

\item
\[
\frac{1}{2\tilde{\mu}}\log\left(\frac{1}{\tilde{\mu}+1/2}\right)=o(\ep\log^2(\ep)).
\]

\end{enumerate}
\end{Lemma}

\begin{proof}[Proof of (1)]
We will start by estimating
\begin{align*}
\log\left(\frac{\tilde{\mu}-1/2}{1/2-\tilde{\nu}}\right)=\log\left(\frac{(\tilde{\mu}-1/2)(1/2+\tilde{\nu})}{1/4-\tilde{\nu}^2}\right).
\end{align*}
Since $\tilde{\nu}^2=\tilde{\mu}^2-\ep/(1-\ep)$ and $\tilde{\mu}=1/2+\tilde{\mu}_1(\ep)\ep$ we obtain that the above denominator can be written as
\begin{align*}
1/4-\tilde{\nu}^2&=1/4-(\tilde{\mu}^2-\frac{\ep}{1-\ep})=(1/2-\tilde{\mu})(1/2+\tilde{\mu})+\frac{\ep}{1-\ep}=-\tilde{\mu}_1(\ep)\ep-\tilde{\mu}_1(\ep)^2\ep^2+\frac{\ep}{1-\ep}\\
&=(1-\tilde{\mu}_1(\ep))\ep +\ep^2\left(-\tilde{\mu}_1(\ep)^2+\frac{1}{1-\ep}\right)\\
&=(1/2-\mu_1)\ep-\tilde{\mu}_2(\ep)\ep^2\log^2(\ep) +\ep^2\left(-\tilde{\mu}_1(\ep)^2+\frac{1}{1-\ep}\right),
\end{align*}
where in the last step we have used that $\tilde{\mu}_1(\ep)=\left(\frac{1}{2}+\mu_1\right) +\tilde{\mu}_2(\ep)\ep\log^2(\ep)$.

Therefore
\begin{align*}
\frac{\tilde{\mu}-1/2}{1/2-\tilde{\nu}}&=\frac{(\tilde{\mu}-1/2)(1/2+\tilde{\nu})}{1/4-\tilde{\nu}^2}\\
&=\frac{\left(1/2+\mu_1\right)\ep +\tilde{\mu}_2(\ep)\ep^2\log^2(\ep)}{(1/2-\mu_1)\ep-\tilde{\mu}_2(\ep)\ep^2\log^2(\ep) +\ep^2\left(-\tilde{\mu}_1(\ep)^2+\frac{1}{1-\ep}\right)}(1/2+\tilde{\nu})\\
&=\frac{1/2+\mu_1}{1/2-\mu_1}
\frac{1+\frac{\tilde{\mu}_2(\ep)}{1/2+\mu_1}\ep\log^2(\ep)}{1-\frac{\tilde{\mu}_2(\ep)}{1/2-\mu_1}\ep\log^2(\ep)+O(\ep)}
(1/2+\tilde{\nu}),
\end{align*}
and we obtain that
\begin{align*}
\log\left(\frac{\tilde{\mu}-1/2}{1/2-\tilde{\nu}}\right)=&\log\left(\frac{1/2+\mu_1}{1/2-\mu_1}\right)\\
&+ \log\left(\frac{1+\frac{\tilde{\mu}_2(\ep)}{1/2+\mu_1}\ep\log^2(\ep)}{1-\frac{\tilde{\mu}_2(\ep)}{1/2-\mu_1}\ep\log^2(\ep)+O(\e)}\right)+\log(1/2+\tilde{\nu}).
\end{align*}
We can easily calculate that
\begin{align*}
1/2+\tilde{\nu}=1+\tilde{\nu}-1/2=1+\frac{\tilde{\nu}^2-1/4}{\tilde{\nu}+1/2}
=1+\frac{(\tilde{\mu}_1(\ep)-1)\ep+\ep^2\left(\tilde{\mu}_1(\ep)^2-\frac{1}{1-\ep}\right)}{\tilde{\nu}+1/2},
\end{align*}
which means
\begin{align*}
\log\left(1/2+\tilde{\nu}\right)=O(\ep)=o(\ep\log^2(\ep)).
\end{align*}
In addition, since
\[
\frac{1+\frac{\tilde{\mu}_2(\ep)}{1/2+\mu_1}\ep\log^2(\ep)}{1-\frac{\tilde{\mu}_2(\ep)}{1/2-\mu_1}\ep\log^2(\ep)+O(\ep)}=1+\frac{\frac{\tilde{\mu}_2(\ep)}{1/4-\mu_1^2}\ep\log^2(\ep)+O(\e)}{1-\frac{\tilde{\mu}_2(\ep)}{1/2-\mu_1}\ep\log^2(\ep)+O(\ep)},
\]
we can conclude that
\begin{align*}
\log\left(\frac{1+\frac{\tilde{\mu}_2(\ep)}{1/2+\mu_1}\ep\log^2(\ep)}{1-\frac{\tilde{\mu}_2(\ep)}{1/2-\mu_1}\ep\log^2(\ep)+O(\ep)}\right)=\frac{\tilde{\mu}_2(\ep)}{1/4-\mu_1^2}\ep\log^2(\ep)+ o(\ep\log^2(\ep)).
\end{align*}
Finally, since
\begin{align*}
\frac{1}{2\tilde{\mu}}-1=\frac{1-2\tilde{\mu}}{2\tilde{\mu}}=O(\ep),
\end{align*}
we find that
\begin{align*}&\frac{1}{2\tilde{\mu}}\log\left(\frac{\tilde{\mu}-1/2}{1/2-\tilde{\nu}}\right)\\
&=\log\left(\frac{1/2+\mu_1}{1/2-\mu_1}\right)+\frac{\tilde{\mu}_2}{1/4-\mu_1^2}\ep\log^2(\ep)+o(\ep\log^2(\ep)).
\end{align*}
\end{proof}

\begin{proof}[Proof of (2)] We begin by noting that
\begin{align*}
\left(\frac{1}{2\tilde{\nu}}-\frac{1}{2\tilde{\mu}}\right)\log\left(\frac{1}{1/2-\tilde{\nu}}\right)&=\frac{\tilde{\mu}-\tilde{\nu}}{2\tilde{\mu}\tilde{\nu}}\log\left(\frac{1/2+\tilde{\nu}}{1/4-\tilde{\nu}^2}\right)\\
&=\frac{\tilde{\mu}-\tilde{\nu}}{2\tilde{\mu}\tilde{\nu}}\left[\log\left(1/2+\tilde{\nu}\right)-\log\left(1/4-\tilde{\nu}^2\right)\right].
\end{align*}
Since we have seen before that $\log\left(1/2+\tilde{\nu}\right)=O(\ep)$ and $\log\left(1/4-\tilde{\nu}^2\right)=O(\log(\ep))$ together with the fact that
\[
\tilde{\mu}-\tilde{\nu}=\tilde{\mu}-\sqrt{\tilde{\mu}^2-\frac{\ep}{1-\ep}}=\frac{\frac{\ep}{1-\ep}}{\left(\tilde{\mu}+\sqrt{\tilde{\mu}^2-\frac{\ep}{1-\ep}}\right)}=O(\ep),
\]
we can conclude that this term behaves like $O(\ep\log(\ep))$.
\end{proof}

\begin{proof}[Proof of (3)] We have
\begin{align*}
\frac{(1-\tilde{\nu})(1/2+\tilde{\nu})}{1+\tilde{\nu}}&=\frac{1}{3}+\frac{5}{9}(1-\tilde{\mu}_1(\ep))\ep + O(\ep^2)\\
&=\frac{1}{3}+\frac{5}{9}(1/2-\mu_1)\ep -\tilde{\mu}_2(\ep)\ep^2\log^2(\ep)+ O(\ep^2).
\end{align*}
Since $2\tilde{\nu}=O(1),$ we get
\begin{align*}
\frac{1}{2\tilde{\nu}}\log\left(\frac{(1-\tilde{\nu})(1/2+\tilde{\nu})}{1+\tilde{\nu}}\right)=-\log(3)+O(\ep)=-\log(3)+o(\ep\log^2(\ep)).
\end{align*}
\end{proof}
\begin{proof}[Proof of (4)] This case follows trivially since
\begin{align*}
\frac{1}{2\tilde{\mu}}\log\left(\frac{1}{1/2+\tilde{\mu}}\right)&=-\log(1+\tilde{\mu}_1(\ep)\ep)+ O(\ep\log(\ep))\\
&=O(\ep)=o(\ep\log^2(\ep)).
\end{align*}
\end{proof}

Therefore, combining all the computations of Lemma \ref{tedious} we finally arrive to
\begin{align*}
&\frac{1}{2\tilde{\mu}}\log\left(\frac{\tilde{\mu}-1/2}{\tilde{\mu}+1/2}\right)+\frac{1}{2\nu}
\log\left(\frac{(1-\tilde{\nu})(1/2+\tilde{\nu})}{(1+\tilde{\nu})(1/2-\tilde{\nu})}\right)\\
& = \log\left(\frac{1/2+\mu_1}{1/2-\mu_1}\right)-\log(3)+ \frac{\tilde{\mu}_2(\ep)}{1/4-\mu_1^2}\ep\log^2(\ep)+o(\ep\log^2(\ep)).
\end{align*}

In addition, as an immediate consequence we obtain the next result.

\begin{Lemma}\label{tedious2}The following estimate holds:
\begin{align*}
&\sinh((1-\ep)/2)\left(\frac{1}{2\tilde{\mu}}\log\left(\frac{\tilde{\mu}-1/2}{\tilde{\mu}+1/2}\right)+\frac{1}{2\nu}
\log\left(\frac{(1-\tilde{\nu})(1/2+\tilde{\nu})}{(1+\tilde{\nu})(1/2-\tilde{\nu})}\right)\right)\\
& = \sinh(1/2)\log\left(\frac{1/2+\mu_1}{1/2-\mu_1}\right)-\sinh(1/2)\log(3)+ \sinh(1/2)\frac{\tilde{\mu}_2(\ep)}{1/4-\mu_1^2}\ep\log^2(\ep)+o(\ep\log^2(\ep)).
\end{align*}
\end{Lemma}

Next we bound $J_1$ and $J_2$ in \eqref{I1masI2entreC}. To do so, we begin by pointing out that
\begin{align*}
J_1=&\int_{0}^\frac{1}{2}\frac{\sinh((1-\ep)w)f(w)-\sinh((1-\ep)/2)}{w^2-\tilde{\mu}^2}dw\\
=&\int_{0}^\frac{1}{2}\frac{\sinh((1-\ep)w)f_0(w)-\sinh((1-\ep)/2)}{w^2-\tilde{\mu}^2}dw +\int_{0}^\frac{1}{2}\frac{\sinh((1-\ep)w)(f(w)-f_0(w))}{w^2-\tilde{\mu}^2}dw,
\end{align*}
and
\begin{align*}
J_2&=\int_{1/2}^1\frac{\sinh((1-\ep)(1-w))f(w)-\sinh((1-\ep)/2)}{w^2-\tilde{\nu}^2}dw\\
&=\int_{1/2}^1\frac{\sinh((1-\ep)(1-w))f_0(w)-\sinh((1-\ep)/2)}{w^2-\tilde{\nu}^2}dw\\
&\quad + \int_{1/2}^1\frac{\sinh((1-\ep)(1-w))(f(w)-f_0(w))}{w^2-\tilde{\nu}^2}dw.
\end{align*}
Applying Lemmas \ref{mainLemma} and \ref{mainLemmaright} into the last integral terms of $J_1, J_2$ we obtain the following result.

\begin{Lemma} \label{J1}It holds that
\begin{align*}
J_1=&\int_{0}^\frac{1}{2}\frac{\sinh((1-\ep)w)f_0(w)-\sinh((1-\ep)/2)}{w^2-\tilde{\mu}^2}dw+M[f_0,\mu_1]+ R[f_0,\mu_1;f,\tilde{\mu}_2],\\
J_2=&\int_{\frac{1}{2}}^1\frac{\sinh((1-\ep)(1-w))f_0(w)-\sinh((1-\ep)/2)}{w^2-\tilde{\nu}^2}dw+M[f_0,\mu_1]+ R[f_0,\mu_1;f,\tilde{\mu}_2],
\end{align*}
with
\begin{align*}
M[f_0,\mu_1]&=O(\ep\log^2(\ep)),\\
R[f_0,\mu_1;f,\tilde{\mu}_2]&=o(\ep\log^2(\ep)).
\end{align*}
\end{Lemma}

Putting together Lemmas \ref{tedious2} and \ref{J1} we get that \eqref{I1masI2entreC} can be written as:

\begin{align}\label{casi}
\frac{I_1+I_2}{\tilde{C}_1}=&  \int_{0}^\frac{1}{2}\frac{\sinh((1-\ep)w)f_0(w)-\sinh((1-\ep)/2)}{w^2-\tilde{\mu}^2}dw\\
&+\int_{1/2}^1\frac{\sinh((1-\ep)(1-w))f_0(w)-\sinh((1-\ep)/2)}{w^2-\tilde{\nu}^2}dw \nonumber\\
&+\sinh(1/2)\log\left(\frac{1/2+\mu_1}{1/2-\mu_1}\right)-\sinh(1/2)\log(3) \nonumber\\
&+M[f_0,\mu_1] +\sinh(1/2)\frac{\tilde{\mu}_2(\ep)}{1/4-\mu_1^2}\ep\log^2(\ep)\nonumber\\
&+ R[f_0,\mu_1;f,\tilde{\mu}_2], \nonumber
\end{align}
with
\begin{align*} M[f_0,\mu_1]&=O(\ep\log^2(\ep)),\\
R[f_0,\mu_1;f,\tilde{\mu}_2]&=o(\ep\log(\ep)).
\end{align*}

Since
\begin{align*}
\sinh((1+\ep)/2)\tilde{C}_1&= \sinh(1/2)C_1+ O(\ep)\\
&=\sinh(1/2)C_1+ o(\ep\log^2(\ep)),
\end{align*} the expression \eqref{casi} shows
\begin{align}\label{postcasi1}
\sinh((1+\ep)/2)(I_1+I_2)&=C_1\sinh(1/2) \int_{0}^\frac{1}{2}\frac{\sinh((1-\ep)w)f_0(w)-\sinh((1-\ep)/2)}{w^2-\tilde{\mu}^2}dw\\
&+C_1\sinh(1/2)\int_{1/2}^1\frac{\sinh((1-\ep)(1-w))f_0(w)-\sinh((1-\ep)/2)}{w^2-\tilde{\nu}^2}dw\nonumber\\
&+C_1\sinh^2(1/2)\log\left(\frac{1/2+\mu_1}{1/2-\mu_1}\right)-C_1\sinh^2(1/2)\log(3)\nonumber\\
&+M[f_0,\mu_1]+\sinh^2(1/2)C_1\frac{\tilde{\mu}_2(\ep)}{1/4-\mu_1^2}\ep\log^2(\ep)\nonumber \nonumber \\
&+ R[f_0,\mu_1;f,\tilde{\mu}_2]. \nonumber
\end{align}

In addition, we have
\begin{Lemma}The following estimates hold:
\begin{align}\label{postcasi2}
&\int_{0}^\frac{1}{2}\frac{\sinh((1-\ep)w)f_0(w)-\sinh((1-\ep)/2)}{w^2-\tilde{\mu}^2}dw\\
&=\int_{0}^\frac{1}{2}\frac{\sinh(w)f_0(w)-\sinh(1/2)}{w^2-1/4}dw+M[f_0,\mu_1]+R[f_0,\mu_1;f,\tilde{\mu}_2],\nonumber
\end{align}
and
\begin{align}\label{postcasi3}
&\int_\frac{1}{2}^1\frac{\sinh((1+\ep)w)f_0(w)-\sinh((1+\ep)/2)}{w^2-\tilde{\nu}^2}dw\\
&=\int_{\frac{1}{2}}^1\frac{\sinh(w)f_0(w)-\sinh(1/2)}{w^2-1/4}dw+M[f_0,\mu_1]+R[f_0,\mu_1;f,\tilde{\mu}_2],\nonumber
\end{align}
where
\begin{align*}
M[f_0,\mu_1]&=O(\ep\log^2(\ep),\\
R[f_0,\mu_1;f,\tilde{\mu}_2]&= o(\ep\log^2(\ep)).
\end{align*}
\end{Lemma}
\begin{proof}
Since the proofs of \eqref{postcasi2} and \eqref{postcasi3} follow similar ideas, we will just  give all the details of the former.

Let us start with \eqref{postcasi2}. Since
\begin{align*}
\frac{1}{w^2-\tilde{\mu}^2}=\frac{1}{2\tilde{\mu}}\left(\frac{1}{w-\tilde{\mu}}-\frac{1}{w+\tilde{\mu}}\right),
\end{align*}
we have that
\begin{align*}
\int_{0}^\frac{1}{2}\frac{\sinh((1-\ep)w)f_0(w)-\sinh((1-\ep)/2)}{w^2-\tilde{\mu}^2}dw=
&\frac{1}{2\tilde{\mu}}\int_0^{\frac{1}{2}}\frac{\sinh((1-\ep)w)f_0(w)-\sinh((1-\ep)/2)}{w-\tilde{\mu}}dw\\
&-\frac{1}{2\tilde{\mu}}\int_0^\frac{1}{2}\frac{\sinh((1-\ep)w)f_0(w)-\sinh((1-\ep)/2)}{w+\tilde{\mu}}dw\\
\equiv & \frac{1}{2\tilde{\mu}}\left(A_1-A_2\right).
\end{align*}
Here, we observe that the prefactor is only
\begin{align*}
\frac{1}{2\tilde{\mu}}=1-\frac{2\tilde{\mu}_1(\ep)}{2\tilde{\mu}}=1-
O(\ep).
\end{align*}
Now we analyze $A_1$ and $A_2$ separately. For $A_1$ we have that
\begin{align*}
A_1=&\int_{0}^\frac{1}{2}\frac{\sinh((1-\ep)w)f_0(w)-\sinh((1-\ep)/2)}{w-\tilde{\mu}}dw\\
=& \int_{0}^\frac{1}{2}\sinh((1-\ep)w)f_0(w)-\sinh((1-\ep)/2)\left(\frac{1}{w-\tilde{\mu}}-\frac{1}{w-1/2}\right)dw\\
&+ \int_{0}^\frac{1}{2}\frac{\left(\sinh((1-\ep)w)f_0(w)-\sinh((1-\ep)/2)\right)-\left(\sinh(w)f_0(w)-\sinh(1/2)\right)}{w-1/2}dw\\
&+ \int_{0}^\frac{1}{2}\frac{\sinh(w)f_0(w)-\sinh(1/2)}{w-1/2}dw\\
\equiv & A_{11}+A_{12}+A_{13}.
\end{align*}
For $A_2$ doing the same we have
\begin{align*}
A_2=&\int_{0}^\frac{1}{2}\frac{\sinh((1-\ep)w)f_0(w)-\sinh((1-\ep)/2)}{w+\tilde{\mu}}dw\\
=& \int_{0}^\frac{1}{2}\sinh((1-\ep)w)f_0(w)-\sinh((1-\ep)/2)\left(\frac{1}{w+\tilde{\mu}}-\frac{1}{w+1/2}\right)dw\\
&+ \int_{0}^\frac{1}{2}\frac{\left(\sinh((1-\ep)w)f_0(w)-\sinh((1-\ep)/2)\right)-\left(\sinh(w)f_0(w)-\sinh(1/2)\right)}{w+1/2}dw\\
&+ \int_{0}^\frac{1}{2}\frac{\sinh(w)f_0(w)-\sinh(1/2)}{w+1/2}dw\\
\equiv & A_{21}+A_{22}+A_{23}.
\end{align*}
In particular, we have that
\begin{align*}
\frac{1}{2\tilde{\mu}}\left(A_1-A_2\right)=\frac{1}{2\tilde{\mu}}\left(A_{11}+A_{12}-A_{21}-A_{22}\right)+(1+O(\ep))(A_{13}-A_{23}),
\end{align*}
with
\[
A_{13}-A_{23}=\int_{0}^\frac{1}{2}\frac{\sinh(w)f_0(w)-\sinh(1/2)}{w^2-1/4}dw.
\]
Moreover, since the above integral is finite, we also have that $O(\ep)(A_{13}-A_{23})=O(\ep)=o(\ep\log^2(\ep))$.
To conclude, we must study each of the remaining terms in detail.

With $A_{11}$ we can proceed as follows
\begin{align*}
A_{11}=\int_{0}^\frac{1}{2}\sinh((1-\ep)w)f_0(w)-\sinh((1-\ep)/2)\frac{\ep\tilde{\mu}_1(\ep)}{(w-\tilde{\mu})(w-1/2)}dw.
\end{align*}
Now we use that
$$f_0(w)=1+f_0^{(1)}(w-1/2)\log(1/2-w)+f_0^R(w),$$
where $f_0^R(w)=O((w-1/2)^2\log(1/2-w))$. We get that
 \begin{align*}
A_{11}=&\ep\tilde{\mu}_1(\ep)\int_{0}^\frac{1}{2}\frac{\sinh((1-\ep)w)f_0(w)-\sinh((1-\ep)/2)}{(w-\tilde{\mu})(w-1/2)}dw\\
=&\ep \tilde{\mu}_1(\ep)\int_{0}^\frac{1}{2}\frac{\sinh((1-\ep)w)-\sinh((1-\ep)/2)}{(w-\tilde{\mu})(w-1/2)}dw\\
&+\ep\tilde{\mu}_1(\ep)\int_{0}^\frac{1}{2}\sinh((1-\ep)w) \frac{\log(1/2-w)f_0^{(1)}+\frac{f_0^R(w)}{w-1/2}}{w-\tilde{\mu}}dw\\
\equiv & \ep\tilde{\mu}_1(\ep)\left(A_{111}+A_{112}\right).
\end{align*}
For $A_{111}$ we have that
\begin{align*}
A_{111}=&\int_{0}^\frac{1}{2}\frac{A_{111}(w)-A_{111}(1/2)}{w-\tilde{\mu}}dw
+A_{111}(1/2)\int_{0}^\frac{1}{2}\frac{1}{w-\tilde{\mu}}dw\\
=&\int_{0}^\frac{1}{2}\frac{A_{111}(w)-A_{111}(1/2)}{w-1/2}\frac{w-1/2}{w-\tilde{\mu}}dw
+A_{111}(1/2)\left(\log(\ep\tilde{\mu}_1(\ep))-\log(\tilde{\mu})\right)\\
=&O(1)+A_{111}(1/2)\log(\ep\tilde{\mu}_1(\ep)),
\end{align*}
where, abusing the notation a little, we have used
\begin{align*}
A_{111}(w)\equiv \frac{\sinh((1-\ep)w)-\sinh((1-\ep)/2)}{w-1/2}.
\end{align*}
Since
\begin{align*}
\log(\ep\tilde{\mu}_1(\ep))=&\log\left(\left(\frac{1}{2}+\mu_1\right)\ep+\tilde{\mu}_2(\ep)\ep^2\log^2(\ep)\right)\\
=&\log\left(\left(\frac{1}{2}+\mu_1\right)\ep\right)
+\log\left(1+\frac{\tilde{\mu}_2(\ep)}{\frac{1}{2}+\mu_1}\ep\log^2(\ep)\right).
\end{align*}
Thus
\begin{equation}\label{aux:A111}
\ep \tilde{\mu}_1(\ep)A_{111}=M[f_0,\mu_1]+R[f_0,\mu_1;\tilde{\mu}_2],
\end{equation}
where
\begin{align*}
M[f_0,\mu_1]&=O(\ep\log^2(\ep)),\\
R[f_0,\mu_1;\tilde{\mu}_2]&=o(\ep\log^2(\ep)).
\end{align*}
For $A_{112}$ we have that
\begin{align*}
A_{112}=&\int_{0}^\frac{1}{2}\sinh((1-\ep)w)\frac{\log(1/2-w)f_0^{(1)}}{w-\tilde{\mu}}dw+
\int_{0}^\frac{1}{2}\sinh((1-\ep)w)\frac{f_0^R(w)}{(w-1/2)^2}\frac{w-\frac{1}{2}}{w-\tilde{\mu}}dw\\
=&f_0^{(1)}\int_{0}^\frac{1}{2}\sinh((1-\ep)w)\frac{\log(1/2-w)}{w-\tilde{\mu}}dw+O(1),
\end{align*}
since $\left|\sinh((1-\ep)w) \frac{f_0^R(w)}{(w-1/2)^2}\right|$ is integrable in $[0,1/2]$.

The remaining integral can be handled as
\begin{multline*}
\int_{0}^\frac{1}{2}\sinh((1-\ep)w)\frac{\log(1/2-w)}{w-\tilde{\mu}}dw=\int_{0}^\frac{1}{2}\frac{\sinh((1-\ep)w)-\sinh((1-\ep)/2)}{w-1/2}\frac{w-1/2}{w-\tilde{\mu}} \log(1/2-w)dw\\
+\sinh((1-\ep)/2)\int_{0}^\frac{1}{2}\frac{\log(1/2-w)}{w-\tilde{\mu}}dw.
\end{multline*}
Here we pay special attention to the most delicate term. This is
\begin{align}\label{aux:remainingINT}
\int_{0}^\frac{1}{2}\frac{\log(1/2-w)}{w-\tilde{\mu}}dw=O(\log^2(\ep\tilde{\mu}_1(\ep)).
\end{align}
%\begin{Lemma}
%\begin{align*}
%\int_{0}^\frac{1}{2}\frac{\log(1/2-w)}{1/2+\tilde{\mu}_1-w}dw = O(\log^2(\tilde{\mu}_1).
%\end{align*}
%\end{Lemma}

\begin{proof}[Proof of \eqref{aux:remainingINT}]By making the appropriate changes of variable we obtain
\begin{align*}
\int_{0}^\frac{1}{2}\frac{\log(1/2-w)}{1/2+\ep\tilde{\mu}_1(\ep)-w}dw=&\int_0^\frac{1}{2}\frac{\log(w)}{w+\ep \tilde{\mu}_1(\ep)}dw=
\int_{0}^\frac{1}{2\ep \tilde{\mu}_1(\ep)}\frac{\log(\ep\tilde{\mu}_1(\ep)w)}{w+1}dw\\
&= \log(\ep\tilde{ \mu}_1(\ep))\log\left(1+\frac{1}{2\ep\tilde{\mu}_1(\ep)}\right)+\int_{0}^\frac{1}{2\ep\tilde{\mu}_1(\ep)}\frac{\log(w)}{1+w}dw.
\end{align*}
Notice that
\[
\log(\ep\tilde{ \mu}_1(\ep))\log\left(1+\frac{1}{2\ep\tilde{\mu}_1(\ep)}\right)
=\frac{1}{\ep}\left(M[f_0,\mu_1]+R[f_0,\mu_1;f,\tilde{\mu}_2]\right).
\]
In addition
\begin{align*}
\int_{0}^\frac{1}{2\ep \tilde{\mu}_1(\ep)}\frac{\log(w)}{w+1}dw= \int_{0}^1\frac{\log(w)}{w+1}dw
+\int_{1}^\frac{1}{2\ep\tilde{\mu}_1(\ep)}\frac{\log(w)}{w+1}dw,
\end{align*}
where
\begin{align*}
\int_{1}^\frac{1}{2\ep\tilde{\mu}_1(\ep)}\frac{\log(w)}{w+1}dw = \int_{1}^\frac{1}{2\ep\tilde{\mu}_1(\ep)}\left(\frac{1}{w+1}-\frac{1}{w}\right)\log(w)dw+\int_{1}^\frac{1}{2\ep\tilde{\mu}_1(\ep)}\frac{\log(w)}{w}dw.
\end{align*}
Applying integration by parts we obtain that
\begin{align*}
\int_{1}^\frac{1}{2\ep\tilde{\mu}_1(\ep)}\frac{\log(w)}{w}dw=\log^2(2\ep\tilde{\mu}_1(\ep))-\int_{1}^\frac{1}{2\ep\tilde{\mu}_1(\ep)}\frac{\log(w)}{w}dw,
\end{align*}
which yields
\begin{align*}
\int_{1}^\frac{1}{2\ep\tilde{\mu}_1(\ep)}\frac{\log(w)}{w}dw=\frac{1}{2}\log^2(2\ep\tilde{\mu}_1(\ep)).
\end{align*}
\end{proof}

Thus
\begin{equation}\label{aux:A112}
\ep\tilde{\mu}_1(\ep)A_{112}=M[f_0,\mu_1]+R[f_0,\mu_1;f,\tilde{\mu}_2],
\end{equation}
where
\begin{align*}
M[f_0,\mu_1]&=O(\ep\log^2(\ep)),\\
R[f_0,\mu_1;f,\tilde{\mu}_2]&=o(\ep\log^2(\ep)).
\end{align*}

Consequently, combining \eqref{aux:A111} and \eqref{aux:A112} we have proved the desired computation for the term $A_{11}$.
The estimation of the other integral term $A_{12}$ follows similar steps. Finally, it is easy to check that $A_{21}, A_{22}=O(1)$.
\end{proof}

Then, combining \eqref{postcasi1}-\eqref{postcasi2}-\eqref{postcasi3} we obtain
\begin{align*}
\sinh(b)(I_1+I_2)&=C_1\sinh(1/2)\int_{0}^\frac{1}{2}\frac{\sinh(w)f_0(w)-\sinh(1/2)}{w^2-1/4}dw\\
&+C_1\sinh(1/2)\int_{\frac{1}{2}}^1\frac{\sinh(1-w)f_0(w)-\sinh(1/2)}{w^2-1/4}dw\\
&+C_1\sinh^2(1/2)\log\left(\frac{1/2+\mu_1}{1/2-\mu_1}\right)-C_1\sinh^2(1/2)\log(3)\nonumber\\
&+M[f_0,\mu_1]+\sinh^2(1/2)C_1\frac{\tilde{\mu}_2(\ep)}{1/4-\mu_1^2}\ep\log^2(\ep)\nonumber \nonumber \\
&+ R[f_0,\mu_1;f,\tilde{\mu}_2].
\end{align*}
Now, using expression \eqref{eculimit}, we achieve the desired result which we collect in the following Lemma.

\begin{Lemma}\label{auxl1} The following relation holds:
\begin{align*}
1+\sinh(b)(I_1+I_2)=\sinh^2(1/2)C_1\frac{\tilde{\mu}_2(\ep)}{1/4-\mu_1^2}\ep\log^2(\ep)+M[f_0,\mu_1]+ R[f_0,\mu_1;f,\tilde{\mu}_2],
\end{align*}
with
\begin{align*}
 M[f_0,\mu_1]=&O(\ep\log^2(\ep)),\\
 R[f_0,\mu_1;f,\tilde{\mu}_2]=&o(\ep\log^2(\ep)).
 \end{align*}
 \end{Lemma}

\subsubsection{The term $I_1I_2$}

We have that
%\begin{align*}
%&\left|\frac{I_1}{\tilde{C}_1}\right|=\left|\int_{0}^\frac{1}{2} \frac{\sinh((1-\ep)w)f(w)}{w^2-\tilde{\mu}^2}\right|
%\\&\leq C||f||_{L^\infty}\int_{0}^\frac{1}{2}\frac{1}{\mu-w}dw \\&\leq C||f||_{L^\infty}\log\left(\frac{\tilde{\mu}-1/2}{\tilde{\mu}}\right)\\
%&\leq C||f||_{L^\infty}\log\left((1/2+\mu_1)\ep+\tilde{\mu}_2(\ep)\ep^2\log^2(\ep)\right)\\
%&\leq C ||f||_{L^\infty}\log(1/\ep)+C||f||_{L^\infty}\log\left(1+\frac{\tilde{\mu}_2(\ep)}{1/2+\mu_1}\ep\log^2(\ep)\right) \\
%&\leq C||f||_{L^\infty}\log(1/\ep)+C||f||_{L^\infty}\ep\log^2(\ep)\\
%&\leq C ||f_0||_{L^\infty}\log(1/\ep)+ C(||d||_{L^\infty}+||f||_{L^\infty})\ep\log^2(\ep)
%\end{align*}

\begin{align*}
\frac{I_1}{\tilde{C}_1}&=\int_{0}^\frac{1}{2} \frac{\sinh((1-\ep)w)f(w)}{w^2-\tilde{\mu}^2}dw\\
&=\int_{0}^\frac{1}{2} \frac{\sinh((1-\ep)w)(f(w)-f_0(w))}{w^2-\tilde{\mu}^2}dw+\int_{0}^\frac{1}{2} \frac{\sinh((1-\ep)w)f_0(w)}{w^2-\tilde{\mu}^2}dw.
\end{align*}
For the first integral, we get
\begin{align*}
\left|\int_{0}^\frac{1}{2} \frac{\sinh((1-\ep)w)(f(w)-f_0(w))}{w^2-\tilde{\mu}^2}dw\right|\leq C||f-f_0||_{L^\infty}\log(1/\ep)\leq C \ep\log^2(\ep),
\end{align*}
where, in the last step, we have used Lemma \ref{Linfty}.
In addition
\begin{align*}
\left|\int_{0}^\frac{1}{2} \frac{\sinh((1-\ep)w)f_0(w)}{w^2-\tilde{\mu}^2}\right|
&\leq C||f_0||_{L^\infty}\int_{0}^\frac{1}{2}\frac{1}{\tilde{\mu}-w}dw \\&\leq C||f_0||_{L^\infty}\log\left(\frac{\tilde{\mu}-1/2}{\tilde{\mu}}\right)\\
&\leq C||f_0||_{L^\infty}\log\left((1/2+\mu_1)\ep+\tilde{\mu}_2(\ep)\ep^2\log^2(\ep)\right)\\
&\leq C ||f_0||_{L^\infty}\log(1/\ep)+C||f_0||_{L^\infty}\log\left(1+\frac{\tilde{\mu}_2(\ep)}{1/2+\mu_1}\ep\log^2(\ep)\right) \\
&\leq C||f_0||_{L^\infty}\log(1/\ep)+C||f_0||_{L^\infty}\ep\log^2(\ep).
\end{align*}
That is,
\begin{align*}
\int_{0}^\frac{1}{2} \frac{\sinh((1-\ep)w)f_0(w)}{w^2-\tilde{\mu}^2}dw=M_1[f_0,\mu_1]+R_1[f_0,\mu_1;f,\tilde{\mu}_2]\end{align*}
where
\begin{align*}
M_1[f_0,\mu_1]=&O(\log(1/\ep)),\\
R_1[f_0,\mu_1;f,\tilde{\mu}_2]=& O(\ep\log^2(\ep)).
\end{align*}

Thus,
\begin{align*}
\frac{I_1}{C_1}=M_1[f_0,\mu_1]+R_1[f_0,\mu_1;f,\tilde{\mu}_2]=& O(\log(1/\ep)),
\end{align*}
and proceeding similarly
\begin{align*}
\frac{I_2}{C_1}=M_1[f_0,\mu_1]+R_1[f_0,\mu_1;f,\tilde{\mu}_2]=& O(\log(1/\ep)).
\end{align*}
Notice that we have eliminated the tilde in $C_1$. In addition, multiplying both, we arrive to
\begin{align*}
\frac{I_1}{C_1}\frac{I_2}{C_1}=M_2[f_0,\mu_1]+R_2[f_0,\mu_1;f,\tilde{\mu}_2]=& O(\log^2(\ep)),
\end{align*}
where
\begin{align*}
M_2[f_0,\mu_1]&=O(\log^2(\ep))\\
R_2[f_0,\mu_1;f,\tilde{\mu}_2]&=O(\ep\log^3(\ep)).
\end{align*}
Since $\sinh^2(b)-\sinh^2(a)=O(\ep)$, we get the desired result, which is shown in the following Lemma.
\begin{Lemma}\label{auxl2}
The following relation holds:
\begin{align*}
(\sinh^2(b)-\sinh^2(a))I_1I_2=M[f_0,\mu_1]+R[f_0,\mu_1;f,\tilde{\mu}_2].
\end{align*}
with
\begin{align*}
 M[f_0,\mu_1]&=O(\ep\log^2(\ep)),\\
 R[f_0,\mu_1;f,\tilde{\mu}_2]&=o(\ep\log^2(\ep)).
\end{align*}
\end{Lemma}

\subsection{Closing the proof by a fixed point argument}\label{s:closing}
By Lemmas \ref{auxl1} and \ref{auxl2} we have that
\begin{align*}
\tilde{\mu}_2(\ep)=M[f_0,\mu_1,\ep]+R[f_0,\mu_1,\tilde{\mu}_2(\ep),f[\tilde{\mu}_2(\ep)],\ep]
\end{align*}
where $M[f_0,\mu_1,\ep]$ is $O(1)$ in $\ep$. Let us consider the unit ball $B_1(M)\equiv |\mu-M[f_0,\mu_1,\ep]|< 1$. This gives us an a priori bound for $\mu$.

Then, for any $\mu\in B_1(M)$, we have $R[f_0,\mu_1,\mu,f[\mu],\ep]=o(1)$ in $\ep$. In addition, the application
\begin{align*}
\mathcal{B} (\mu)=M[f_0,\mu_1,\ep]+R[f_0,\mu_1,\mu,f[\mu],\ep]
\end{align*}
is continuous in $\mu$ for $\mu\in B_1(M)$ and $\ep$ small enough. It also  maps $B_1(M)$ to $B_{1/2}(M)$, for $\ep$ small enough, since
\begin{align*}
|\mathcal{B}(\mu)-M[f_0,\mu_1,\ep]|\leq o(1)C(M).
\end{align*}

Then, by Brouwer's fixed point theorem, we obtain a solution $\tilde{\mu}_2(\ep)$ of equation \eqref{detsol} satisfying
\begin{align*}
\tilde{\mu}&>1/2, \\ \tilde{\nu}&< 1/2.
\end{align*}

\subsection{Undoing the changes and recovering the solution}

To sum up, we have proved that the solution of the original integral system \eqref{F11}-\eqref{F21} is given by $$F(x)=\begin{cases}
A f(x), \qquad \text{for } 0\leq x < 1/2,\\
B f(x), \qquad \text{for } 1/2 < x \leq 1,
\end{cases}$$
with $f$ solving \eqref{principalep}-\eqref{rightprincipalep} and $(A,B)$ the only (up to multiplication by constant) non-trivial solution of \eqref{matrix}
for
\begin{equation}\label{eq:finalansatz}
\tilde{\mu}= \frac{1}{2}+\tilde{\mu}_1(\ep)\ep =\frac{1}{2}+\left(\frac{1}{2}+\mu_1\right)\ep +\tilde{\mu}_2(\ep)\ep^2\log^2(\ep),
\end{equation}
where $\mu_1$ is the unique solution of \eqref{eculimit} and $\tilde{\mu}_2(\ep)$ given by the fixed point argument.

After undoing the rescaling, we obtain that solution of \eqref{kernel10}-\eqref{kernel20} is given by
\begin{equation}\label{solafterrescalin}
    h(y)=\begin{cases}
        \frac{1}{(y-\mu)(y+\mu)}F\left(\frac{y}{1-\ep}\right), \qquad \text{for } 0\leq y < a,\\
        \frac{1}{(y-\nu)(y-\rho)}F\left(\frac{y-\ep}{1-\ep}\right), \qquad \text{for } b< y \leq 1.
    \end{cases}
\end{equation}

Then, we have $h^\ast(x,y)=h(y)\cos(x)$ and
 $\lambda^\ast=v_\varpi-\bar{\lambda}=v_\varpi-\mu^2=v_\varpi-(1-\ep)^2\tilde{\mu}^2.$

\subsection{Uniqueness of the solution}\label{s:regularidad}

Let us fix some $(\ep, \tilde{\mu}(\ep))$, from the previous section for which we know that there exists a solution $F$ that solves \eqref{F1} and \eqref{F2}. Let us assume that we have a solution $G=G_{in}\chi_{[0,1/2]} +G_{out}\chi_{[1/2,1]}$ with $G_{in}\in C^{1+\delta}([0,1/2])$ and $G_{out}\in C^{1+\delta}([1/2,1])$ of \eqref{F1} and \eqref{F2}.  By bootstrapping we can prove that, in fact,  $G_{in}\in C^{\infty}([0,1/2])$ and $G_{out}\in C^{\infty}([1/2,1])$. Thus, taking derivatives, we can show that $G$ satisfies \eqref{f_leftconm=1}, \eqref{f_rightconm=1} and $G(0)=G(1)=0$. By Lemmas  \ref{l:positivityfn} and \ref{l:positivityfnright} with $n=1$ we get that $G_{in}=A F$ in $[0,1/2]$ and $G_{out}=B F$ in $[1/2,1]$, for some real parameters $A$ and $B$. Then, $A$ and $B$ have to satisfies \eqref{matrix}. Since all the coefficients in the matrix in \eqref{matrix} cannot be zero at the same time, we get that $G$ must be proportional to $F$ in $[0,1]$.

\section{One dimensionality of the kernel of the linear operator}\label{s:dimkernel}
In this section we will show that once we have fixed $\tilde{\mu}$ the kernel of the linear operator is one-dimensional. For this we will use an argument by contradiction. That is, we will start by assuming that there is another element in the kernel and we will obtain a contradiction.

Let $\tilde{\mu}$ be given by \eqref{eq:finalansatz}. Because of section \ref{s:regularidad} we just need to focus on frequencies different from one.  For $n\neq 1$ let $F=F_{in}\chi_{[0,1/2]}+F_{out}\chi_{[1/2,1]}$, with $F_{in}\in C^{1+\alpha}([0,1/2])$ and $F_{out}\in C^{1+\alpha}([1/2,1])$, be solution of

\begin{align}\label{F1conN}
F(x)&+\tilde{C}_n\sinh(n(1-(1-\ep)x))\int_{0}^{x}\frac{\sinh(n(1-\ep)w)F(w)}{w^2-\tilde{\mu}^2}dw\\
&+\tilde{C}_n\sinh(n(1-\ep)x)\int_{x}^\frac{1}{2} \frac{\sinh(n(1-(1-\ep)w))F(w)}{w^2-\tilde{\mu}^2}dw\nonumber\\
&+\tilde{C}_n\sinh(n(1-\ep)x)\int_{\frac{1}{2}}^1\frac{\sinh(n(1-\ep)(1-w))F(w)}{w^2-\tilde{\nu}^2}dw=0, \qquad \text{for } 0\leq x<1/2, \nonumber
\end{align}
and
\begin{align}\label{F2conN}
F(x)&+\tilde{C}_n \sinh(n(1-\ep)(1-x))\int_{0}^\frac{1}{2} \frac{\sinh(n(1-\ep)w)F(w)}{w^2-\tilde{\mu}^2}dw\\
&+\tilde{C}_n \sinh(n(1-\ep)(1-x))\int_{\frac{1}{2}}^x \frac{\sinh(n((1-\ep)w+\ep))F(w)}{w^2-\tilde{\nu}^2}dw\nonumber\\
&+\tilde{C}_n \sinh(n((1-\ep)x+\ep))\int_{x}^1\frac{\sinh(n(1-\ep)(1-w))F(w)}{w^2-\tilde{\nu}^2}dw, \qquad \text{for } 1/2<x\leq 1.\nonumber
\end{align}

Notice that, since we have fixed $\tilde{\mu}$, the only unknown is $F$.

Then, since actually $F_{in}$ and $F_{out}$ have to be smooth in $(0,1/2)$ and $(1/2,1)$ respectively,  we know that $F$ satisfies
\begin{align*}
F''(x)-\left(\frac{2}{x^2-\tilde{\mu}^2}+(1-\ep)^2 n^2\right)F(x)&=0, \qquad \text{for }0\leq x< 1/2,\\
F''(x)-\left(\frac{2}{x^2-\tilde{\nu}^2}+(1-\e)^2 n^2 \right)F(x)&=0, \qquad \text{for }1/2<x\leq 1,\\
F(0)&=0,\\
F(1)&=0.
\end{align*}

Next we prove Lemmas \ref{l:positivityfn} and \ref{l:positivityfnright} which play a very important role in our analysis. They are a generalization of Lemmas \ref{positivog} and \ref{positivofm=1} and give us positivity and uniqueness of solutions of equations \eqref{fn_L}-\eqref{fn_R:bc12}.
\begin{Lemma}\label{l:positivityfn}
Let $f_n$ be a solution of
\begin{equation}\label{eq:fn}
f_n''(x)-\left(\frac{2}{x^2-\tilde{\mu}^2}+(1-\ep)^2 n^2\right)f_n(x)=0,
\end{equation}
satisfying $f_n(0)=0$ and $f_n'(0)>0$ (or $f_n'(0)<0$). Then $f_n(x)>0$ (or $f_n(x)<0$) for $0<x\leq  \tilde{\mu}$.
\end{Lemma}

\begin{proof}
We will assume, without loss of generality, that $f_n'(0)=1$. We start writing \eqref{eq:fn} in a more convenient way as
\begin{align*}
f_n''(x)-(1-\ep)^2 n^2 f_n(x)= \frac{2}{x^2-\tilde{\mu}^2}f_n(x).
\end{align*}
We find that
\begin{align*}
f_n(x)=\frac{\sinh((1-\ep)nx)}{(1-\ep)n}+\frac{1}{(1-\ep)n}\int_{0}^x \sinh[(1-\ep)n(x-z)]\frac{2f_n(z)}{z^2-\tilde{\mu}^2}dz.
\end{align*}
Iterating once this formula, we have
\begin{align*}
f_n(x)=&\frac{\sinh((1-\ep)nx)}{(1-\ep)n}+\frac{1}{(1-\ep)^2 n^2}\int_{0}^x\sinh[(1-\ep)n(x-z)]\sinh((1-\ep)nz)\frac{2}{z^2-\tilde{\mu}^2}dz\\
&+\frac{1}{(1-\ep)^2 n^2}\int_0^x\sinh[(1-\ep)n(x-z)]\frac{2}{z^2-\tilde{\mu}^2}\int_{0}^z \sinh[(1-\ep)n(z-y)]\frac{2}{y^2-\tilde{\mu}^2}f_n(y)dydz.
\end{align*}
We have that, as far as $f_n(x)>0$ and $x\leq \tilde{\mu}$ we obtain
\begin{align*}
f_n(x)\geq & \frac{\sinh((1-\ep)nx)}{(1-\ep)n}+\frac{1}{(1-\ep)^2 n^2}\int_{0}^x\sinh[(1-\ep)n(x-z)]\sinh((1-\ep)nz)\frac{2}{z^2-\tilde{\mu}^2}dz.
\end{align*}
Defining the auxiliary parameter $p\equiv (1-\ep)n\in \R^{+}$, the proof can be reduced to checking that
\begin{equation}\label{eq:auxpositiviy}
\frac{\sinh(px)}{p}-\frac{1}{p^2}\int_{0}^x\sinh[p(x-z)]\sinh(p z)\frac{2}{\tilde{\mu}^2-z^2}dz>0, \qquad\text{for } 0<x\leq \tilde{\mu}.
\end{equation}
In first place, we note that the above follows from
\begin{equation*}
\frac{\sinh(px)}{p}-\frac{1}{p^2}\int_{0}^x\sinh[p(x-z)]\sinh(p z)\frac{2}{x^2-z^2}dz>0, \qquad\text{for } 0<x\leq \tilde{\mu}.
\end{equation*}
In addition, doing the change of variables $z=\bar{z}/p$ and introducing the auxiliary variable $\xi=px$ on the above expression, we obtain that \eqref{eq:auxpositiviy} follows from
\begin{equation}\label{aux:corolario}
\sinh(\xi)-\int_{0}^{\xi}\sinh(\xi-\bar{z})\sinh(\bar{z})\frac{2}{\xi^2-\bar{z}^2}d\bar{z}>0, \qquad \text{for }\xi>0.
\end{equation}
To conclude, we just need to note that
\[
0\leq \sinh(\xi-\bar{z})\sinh(\bar{z}) \leq \begin{cases}
\sinh(\xi)\bar{z} \qquad &0\leq \bar{z} \leq \frac{\xi}{2},\\
\sinh(\xi)(\xi-\bar{z}) \qquad &\frac{\xi}{2}\leq \bar{z} \leq \xi.
\end{cases}
\]
This allows us to say that
\begin{multline*}
\sinh(\xi)-\int_{0}^{\xi}\sinh(\xi-\bar{z})\sinh(\bar{z})\frac{2}{\xi^2-\bar{z}^2}d\bar{z} \\
\geq \sinh(\xi)\left(1-\int_0^{\xi/2} \frac{2\bar{z}}{\xi^2-\bar{z}^2}d\bar{z}-\int_{\xi/2}^\xi \frac{2(\xi-\bar{z})}{\xi^2-\bar{z}^2}d\bar{z}\right) \\
=\sinh(\xi)\left(1-\log\left(\frac{4}{3}\right)-\log\left(\frac{16}{9}\right)\right)=\sinh(\xi)(1-0.87)>0.
\end{multline*}
Going back in the argument, we have proved \eqref{eq:auxpositiviy} and then the positivity of the solution. In fact

\begin{align}\label{pex}
f_n(x)\geq c \frac{\sinh((1-\ep)nx)}{(1-\ep)n}, \qquad \text{for} \quad 0<c < 0.13.
\end{align}

\end{proof}

A minor modification of the previous argument also allows us to ensure the same result by setting the problem at any arbitrary point of the interval.

\begin{Lemma}\label{cor:positivoxstar}
Fixed a point $x^\ast\in(0,\tilde{\mu})$. Let $f_n$ be a solution of
\begin{equation*}
f_n''(x)-\left(\frac{2}{x^2-\tilde{\mu}^2}+(1-\ep)^2 n^2\right)f_n(x)=0,
\end{equation*}
satisfying $f_n(x^\ast)=0$ and $f_n'(x^\ast)>0$. Then $f_n(x)>0$  for $x^\ast<x\leq  \tilde{\mu}$.
In a similar manner to the previous case, when $f_n'(x^\ast)<0$, it follows that $f_n(x)<0$ for all $x^\ast< x\leq \mu$.
\end{Lemma}
\begin{proof}
We find that solution of the above ODE with boundary conditions $f_n(x^\ast)=0$ and $f_n'(x^\ast)=1$ is given by
\begin{align*}
f_n(x)=\frac{\sinh((1-\ep)n(x-x^\ast))}{(1-\ep)n}+\frac{1}{(1-\ep)n}\int_{x^\ast}^x \sinh[(1-\ep)n(x-z)]\frac{2f_n(z)}{z^2-\tilde{\mu}^2}dz.
\end{align*}
Iterating the formula, we have that as fas as $f_n(x)\geq 0$  and $ x^\ast <x \leq \tilde{\mu}$ we obtain
\begin{multline*}
f_n(x)\geq  \frac{\sinh((1-\ep)n(x-x^\ast))}{(1-\ep)n}\\
-\frac{1}{(1-\ep)^2 n^2}\int_{x^\ast}^x\sinh[(1-\ep)n(x-z)]\sinh((1-\ep)n(z-x^\ast))\frac{2}{x^2-z^2}dz.
\end{multline*}
Introducing the auxiliary variables $\xi=(1-\ep)n x$ and $\bar{\xi}=(1-\ep)n x^\ast$ and doing the change of variable $(1-\ep)n z = \bar{z}$ we get that we have reduced the problem to check that
\[
\sinh(\xi-\bar{\xi})-\int_{\bar{\xi}}^{\xi}\sinh(\xi-\bar{z})\sinh(\bar{z}-\bar{\xi})\frac{2}{\xi^2-\bar{z}^2}d\bar{z}>0, \qquad \text{for }\xi>\bar{\xi}.
\]
We make the change $w=\bar{z}-\bar{\xi}$ to obtain
\[
\sinh(\xi-\bar{\xi})-\int_{0}^{\xi-\bar{\xi}}\sinh(\xi-\bar{\xi}-w)\sinh(w)\frac{2}{\xi^2-{(w+\bar{\xi})^2}}dw>0.
\]
Since $(\xi^2-(z+\bar{\xi})^2)^{-1}\leq ((\xi-\bar{\xi})^2-z^2)^{-1}$ for $z\leq \xi -\bar{\xi}$, the above reduces to show that
\[
\sinh(\xi-\bar{\xi})-\int_{0}^{\xi-\bar{\xi}}\sinh(\xi-\bar{\xi}-w)\sinh(w)\frac{2}{(\xi-\bar{\xi})^2-w^2}dw>0,  \qquad \text{for }\xi-\bar{\xi}>0.
\]
Then we have translated our problem to \eqref{aux:corolario}, which allow us to apply the last part of the previous Lemma and conclude the proof.
\end{proof}

\begin{Lemma}\label{l:positivityfnright}
Let $f_n$ be a solution of
\begin{equation*}
    f_n''(x)-\left(\frac{2}{x^2-\tilde{\nu}^2}+(1-\ep)^2 n^2\right)f_n(x)=0.
\end{equation*}
satisfying $f_n(1)=0$ and $f_n'(1)<0$ (or $f_n'(1)>0$). Then $f_n(x)>0$ (or $f_n(x)<0$) for $ \tilde{\nu}\leq x< 1$.
\end{Lemma}
\begin{proof}

We will write the previous ODE in a more convenient way as
\begin{align*}
f_n''(x)-(1-\ep)^2 n^2 f_n(x)= \frac{2}{x^2-\tilde{\nu}^2}f_n(x).
\end{align*}
Without loss of generality, we can assume that $f_n'(1)=-1$. We find that their solution is given by
\begin{align*}
f_n(x)=\frac{\sinh((1-\ep)n(1-x))}{(1-\ep)n}+\frac{1}{(1-\ep)n}\int_{x}^1 \sinh[(1-\ep)n(z-x)]\frac{2f_n(z)}{z^2-\tilde{\nu}^2}dz.
\end{align*}
By hypothesis $f_n(1)=0$ and $f_n'(1)<0$, then we know that there exist $x^\ast<1$ such that $f_n(x)>0$ for all $x\in(x^\ast,1).$  Since $\tilde{\nu}<1/2<\tilde{\mu}$,  we have that, as far as $f_n(x) > 0$ we obtain
\[
f_n(x)\geq \frac{\sinh((1-\ep)n(1-x))}{(1-\ep)n}.
\]
Then, we can iterate this argument to conclude that $f_n(x)>0$ for all $\tilde{\nu}\leq x < 1.$
\end{proof}

Finally, we include without proof the analogous result of the Lemma \ref{cor:positivoxstar} for the case of setting the problem at any arbitrary point $x^\ast \in(\tilde{\nu},1)$.
\begin{Lemma}\label{cor:positivoxstarright}
Fixed a point $x^\ast\in(\tilde{\nu},1)$. Let $f_n$ be a solution of
\begin{equation*}
f_n''(x)-\left(\frac{2}{x^2-\tilde{\nu}^2}+(1-\ep)^2 n^2\right)f_n(x)=0,
\end{equation*}
satisfying $f_n(x^\ast)=0$ and $f_n'(x^\ast)>0$. Then $f_n(x)<0$  for   $\tilde{\nu}\leq x < x^\ast$.
In a similar manner to the previous case, when $f_n'(x^\ast)<0$, it follows that $f_n(x)>0$ for all $\tilde{\nu}\leq x < x^\ast$.
\end{Lemma}
\begin{comment}
\begin{proof}
Solo hay que usar que la solucion  con condicones de borde $f_n(x^*)=0$ y $f_n'(x^+)=-1$ viene dada por
\[
f_n(x)=\frac{\sinh((1-\ep)n(x^*-x))}{(1-\ep)n}+\frac{1}{(1-\ep)n}\int_{x}^{x^*} \sinh[(1-\ep)n(z-x)]\frac{2f_n(z)}{z^2-\tilde{\nu}^2}dz.
\]
\end{proof}
\end{comment}

Next we show that $F(1/2^-)\neq 0$. Indeed, if $F(1/2^-)=0$ then   $F'(0)=0$, by Lemma \ref{l:positivityfn}. Since $x=0$ is an ordinary  point we would obtain that $F=0$ in $(0,1/2)$. By making a similar argument in the interval $(1/2,1)$, using Lemma \ref{l:positivityfnright} instead of \ref{l:positivityfn}, we can write that

\begin{align*}
F(x)&=Af_n(x) \quad\text{for $0\leq x<1/2,$}\\
F(x)&=Bf_n(x) \quad \text{for $1/2< x \leq 1$},
\end{align*}
where $A$ and $B$ are real numbers and
\begin{align}
f_n''(x)-\left(\frac{2}{x^2-\tilde{\mu}^2}+(1-\ep)^2 n^2\right)f_n(x)&=0, \qquad \text{for }0\leq x< 1/2, \label{fn_L}\\
f_n''(x)-\left(\frac{2}{x^2-\tilde{\nu}^2}+(1-\e)^2 n^2 \right)f_n(x)&=0, \qquad \text{for }1/2<x\leq 1,\label{fn_R}\\
f_n(0)&=0,\label{fn_L:bc0}\\
f_n(1)&=0,\label{fn_R:bc1}\\
f_n(1/2^+)&=1, \label{fn_L:bc12}\\
f_n(1/2^-)&=1.\label{fn_R:bc12}
\end{align}

In addition, by taking the limits $x\to  1/2^{\pm}$ in the equations \eqref{F1conN} and \eqref{F2conN} we find that
\begin{align*}
\left(\begin{array}{cc} 1+\sinh(nb)I_1 & \sinh(na)I_2\\ \sinh(na)I_1  &  1+\sinh(nb)I_2 \end{array}\right)\left(\begin{array}{cc} A \\ B\end{array}\right)=0,
\end{align*}
with
\begin{align*}
I_1&\equiv \tilde{C}_n\int_{0}^\frac{1}{2}\sinh(n(1-\ep)w)\frac{f_n(w)}{w^2-\tilde{\mu}^2}dw,\\
I_2&\equiv \tilde{C}_n\int_{\frac{1}{2}}^1\sinh(n(1-\ep)(1-w))\frac{f_n(w)}{w^2-\tilde{\nu}^2}dw.
\end{align*}
Thus if $(A,B)\neq 0$ we obtain that
\begin{align*}
\det\left(\begin{array}{cc} 1+\sinh(nb)I_1 & \sinh(na)I_2\\ \sinh(na)I_1  &  1+\sinh(nb)I_2 \end{array}\right)=0,
\end{align*}
which implies that
\begin{equation}\label{eq:contradiccion}
1+\sinh(n b)(I_1+I_2)+(\sinh^2(n b))-\sinh^2(n a))I_1 I_2=0.
\end{equation}
In the following we are going to study in detail each of the terms that appear in the above expression. This will allow us to conclude that it is impossible to be satisfied and from there we will be able to obtain the contradiction we are looking for.

\subsubsection{Analyzing equation \eqref{eq:contradiccion}}

In the following, we will assume that $||f_n||_{L^\infty([0,1])}\leq K$  where $K$ does not depend on either $n$ or $\ep$. We can be sure of this as an immediate consequence of Lemmas \ref{l:positivityfn} and \ref{l:positivityfnright}, where we showed that $f_n$ is non-negative and \eqref{fndecrece_n}, where we will prove that
\begin{align*}
&f_1(x)\geq f_2(x)\geq f_3(x)\geq \ldots \geq 0, \qquad \text{for } 0\leq x \leq 1.
\end{align*}

Then
\begin{align*}
    |I_1|&\leq K \tilde{C}_n\int_{0}^\frac{1}{2}\frac{\sinh(n(1-\ep)w)}{(\tilde{\mu}-w)(\tilde{\mu}+w)}dw\\
    &\leq 2 K\tilde{C}_n \sinh(na) \int_0^\frac{1}{2}\frac{1}{1/2+\tilde{\mu}_1(\ep)\ep-w}dw\\
    &= 2K \tilde{C_n}\sinh(na)\log\left(\frac{\frac{1}{2}+\tilde{\mu}_1(\ep)\ep}{\tilde{\mu}_1(\ep)\ep}\right),
\end{align*}
and
\begin{align*}
|I_2|&\leq  K \tilde{C}_n\int_{1/2}^1\frac{\sinh(n(1-\ep)(1-w))}{(w-\tilde{\nu})(w+\tilde{\nu})}dw\\
&\leq 2K \tilde{C}_n \sinh(na) \int_{1/2}^1 \frac{1}{w-\tilde{\nu}}dw\\
&=2K \tilde{C}_n \sinh(na) \log\left(\frac{1-\tilde{\nu}}{1/2-\tilde{\nu}} \right)
\end{align*}
where we have used that $\sinh(\cdot)$ is strictly increasing and the upper bound $(w+\tilde{\mu})^{-1},(w+\tilde{\nu})^{-1}\leq 2$.

In addition, using definition of $\tilde{\mu}$ it is clear that
\[
\log\left(\frac{\tilde{\mu}}{\tilde{\mu}-1/2}\right)=\log\left(\frac{\frac{1}{2}+\tilde{\mu}_1(\ep)\ep}{\tilde{\mu}_1(\ep)\ep}\right)=O(\log(1/\ep)).
\]
Similarly, by definition of $\tilde{\nu}$ and the computations we have seen in Lemma \ref{tedious}, we have that
\[
\log\left(\frac{1-\tilde{\nu}}{1/2-\tilde{\nu}} \right)= \log\left(\frac{(1-\tilde{\nu})(1/2+\tilde{\nu})}{1/4-\tilde{\nu}^2} \right)=\log\left(\frac{(1-\tilde{\nu})(1/2+\tilde{\nu})}{(1-\tilde{\mu}_1(\ep))\ep + O(\e^2)} \right)=O(\log(1/\ep)).
\]

Therefore, by all the above, we have seen that
%\[
%|I_1+I_2|\leq  \tilde{C}_n \sinh(n a)O(\log(1/\ep)),
%\]
%and
\[
|I_1 I_2|\leq \tilde{C}_n^2 \sinh^2(n a) O(\log^2(1/\ep)).
\]

%In order to conclude with our result we will have to make a much more detailed study of the terms involved. On the one hand, using that
%\begin{align}\label{aux:boundsinh}
%\sinh(nb)\sinh(na)=\frac{1}{2}\cosh(n)-\frac{1}{2}\cosh(n\ep)
%\end{align}
%and since $\tilde{C}_n=\frac{2}{(1-\ep)n\sinh(n)}$ we have that
%%\begin{align*}
%%|\sinh(nb)(I_1+I_2)|&\leq  \tilde{C}_n \sinh(n a)\sinh(nb)O(\log(1/\ep))\\
%%&\leq  \frac{1}{n\sinh(n)}(\cosh(n)-\cosh(n\ep))O(\log(1/\ep)\\
%%&\leq \frac{1}{n\tanh(n)}O(\log(1/\ep)\leq \frac{1}{n}O(\log(1/\ep)),
%%\end{align*}
%%where $O(\log(1/\ep))$ does not depend on $n$.
%
%On the other hand,

Using that
\begin{align*}
\sinh^2(nb)-\sinh^2(na)&=(\sinh(nb)+\sinh(na))(\sinh(nb)-\sinh(na))\\
&\leq 2 \sinh(nb)(\sinh(nb)-\sinh(na)),
\end{align*}
and the fact that
\begin{align*}
\sinh(nb)-\sinh(na)&=\int_{-1}^1 \frac{d}{ds} \sinh\left(n\left(\frac{1+s\ep}{2}\right)\right)ds
=\frac{n\ep}{2}\int_{-1}^1\cosh\left(n\left(\frac{1+s\ep}{2}\right)\right)ds\\
&\leq n\ep \cosh(nb),
\end{align*}
we get
\begin{align*}
\sinh^2(nb)-\sinh^2(na)\leq 2n\ep \sinh(nb)\cosh(nb).
\end{align*}
Therefore, recalling that $\tilde{C}_n=\frac{2}{(1-\ep)n\sinh(n)}$ we have that,
\begin{align*}
\left(\sinh^2(nb)-\sinh^2(na)\right)|I_1I_2|\leq \frac{\ep}{n}\frac{\sinh(nb)\sinh(na)\cosh(nb)\sinh(na))}{\sinh(n)\sinh(n)}O(\log^2(1/\ep)),
\end{align*}
where $O(\log^2(\ep))$ does not depend on $n$.

Since
\begin{align}\label{aux:boundsinh}
\sinh(nb)\sinh(na)=\frac{1}{2}\cosh(n)-\frac{1}{2}\cosh(n\ep),
\end{align}
we have
\begin{align*}
\frac{\sinh(nb)\sinh(na)}{\sinh(n)}\leq C
\end{align*}
uniformly in $n$ and $\ep$.  In addition, since
\begin{align*}
\cosh(nb)\sinh(na)=\frac{1}{2}\sinh(n(b+a))-\frac{1}{2}\sinh(n(b-a))=\frac{1}{2}\sinh(n)-\frac{1}{2}\sinh(n\ep),
\end{align*}
we also have that
\begin{align*}
\frac{\cosh(nb)\sinh(na)}{\sinh(n)}\leq C
\end{align*}
uniformly in $n$ and $\ep$.

Then, we finally have
\begin{align}\label{goodbound}
\left(\sinh^2(nb)-\sinh^2(na)\right)|I_1I_2|\leq \frac{1}{n}O(\ep\log^2(\ep)),
\end{align}
where $O(\ep\log^2(1/\ep))$ does not depend on $n$.

Thus, applying \eqref{goodbound} into \eqref{eq:contradiccion} we have learned that
\begin{align}\label{detr1}
1+\sinh(nb)(I_1+I_2)= \frac{1}{n}O(\ep\log^2(\ep)),
\end{align}
where $O(\ep\log^2(\ep))$ does not depend on $n$.

The hardest term to estimate is $I_1+I_2$. To handle it, we have to be much more careful with it. We first consider the integral
\begin{align*}
\frac{I_1}{\tilde{C}_n}&=\int_{0}^\frac{1}{2}\sinh(n(1-\ep)w)\frac{f_n(w)}{w^2-\tilde{\mu}^2}dw,
\end{align*}
and recall that $f_n$ solves \eqref{fn_L:bc0} with boundary conditions \eqref{fn_L:bc0}, \eqref{fn_L:bc12}. In particular, we have
\begin{align*}
\frac{I_1}{\tilde{C}_n}&=  \frac{1}{2}\int_{0}^\frac{1}{2}\sinh(n(1-\ep)w)\left(f_n''(w)-(1-\ep)^2 n^2 f_n(w)\right)dw.
\end{align*}
Applying integration by parts and using the above boundary conditions we get
\begin{equation}\label{I1:computation}
\frac{I_1}{\tilde{C}_n}=\frac{1}{2}\sinh(na)f_n'(1/2^{-})-\frac{n(1-\ep)}{2}\cosh(na).
\end{equation}
Proceeding similarly with the other integral
\[
\frac{I_2}{\tilde{C}_n}=\int_{\frac{1}{2}}^1 \sinh(n(1-\ep)(1-w))\frac{f_n(w)}{w^2-\tilde{\nu}^2}dw,
\]
and recalling that now $f_n$ solves \eqref{fn_R} with boundary conditions \eqref{fn_R:bc1}, \eqref{fn_R:bc12} and applying integration by parts we obtain
\begin{align}\label{I2:computation}
\frac{I_2}{\tilde{C}_n}&=\frac{1}{2}\int_{\frac{1}{2}}^1 \sinh(n(1-\ep)(1-w))\left(f_n''(w)-(1-\ep)^2 n^2 f_n(w)\right)dw\\
&=-\frac{1}{2}\sinh(na)f_n'(1/2^+)-\frac{n(1-\ep)}{2}\cosh(na).\nonumber
\end{align}

Therefore, combining \eqref{I1:computation} and \eqref{I2:computation} we  arrive to
\begin{equation}\label{presalto}
    \frac{I_1+I_2}{\tilde{C}_n}= \frac{1}{2}\sinh(na)\left(f_n'(1/2^-)-f_n'(1/2^+)\right)-n(1-\ep)\cosh(na).
\end{equation}

It is now necessary to study in detail the jump that occurs at the point $x=1/2$. To do this we remember that on the left of this point we have seen that
\[
f_1(x)>f_2(x)>f_3(x)>  \ldots > 0, \qquad \text{for } 0 < x < 1/2,
\]
with the boundary condition  $f_n(1/2^-)=1$ for all $n\geq 1$. Then, we have that
\[
f_1'(1/2^-)\leq f_2'(1/2^-)\leq f_3'(1/2^-)<\ldots
\]
and consequently, by Lemma \ref{desizq},
\begin{equation}\label{fnprimaleft}
f_1'(1/2^-)<f_2'(1/2^-)\leq f_n'(1/2^-), \qquad \text{for all } n>1.
\end{equation}
Similarly, on the right we also have that
\[
f_1(x)>f_2(x)>f_3(x)>  \ldots > 0, \qquad \text{for } 1/2 < x < 1,
\]
with the boundary condition $f_n(1/2^+)=1$ for all $n\geq 1$. Then, we have that
\[
f_1'(1/2^+)\geq f_2'(1/2^+)\geq f_3'(1/2^+)>\ldots
\]
and consequently, by Lemma \ref{desder},
\begin{equation}\label{fnprimaright}
f_1'(1/2^+)>f_2'(1/2^+)\geq f_n'(1/2^+), \qquad \text{for all } n>1.
\end{equation}
Thus, combining \eqref{fnprimaleft} and \eqref{fnprimaright},
the jump at the point $x=1/2$  can be bounded by below as
\[
f_n'(1/2^-)-f_n'(1/2^+) \geq f_2'(1/2^-)-f_2'(1/2^+)
\]
In addition, by lemmas \ref{desizq} and \eqref{desder}
\begin{equation}\label{salto+delta}
f_n'(1/2^-)-f_n'(1/2^+) >f_1'(1/2^-)-f_1'(1/2^+)+\delta/2.
\end{equation}
uniformly in $\ep$, for $\ep$ small enough.

In this way, by combining \eqref{presalto} and \eqref{salto+delta} we have
\begin{equation}\label{presaltof1}
    \frac{I_1+I_2}{\tilde{C}_n}> \frac{1}{2}\sinh(na)\left(f_1'(1/2^-)-f_1'(1/2^+)\right)-n(1-\ep)\cosh(na)+\frac{\delta}{4}\sinh(na).
\end{equation}
Now it is time to remember that we have obtained an element in the kernel of the linear operator for the particular case $n=1$. This has been the main part of the manuscript. In summary, we have solved
\begin{align*}
1
+\sinh(b)\left(I_1^\ast+I_2^\ast\right)
+\left(\sinh^2(b)-\sinh^2(a)\right)I_1^\ast I_2^\ast =0,
\end{align*}
with
\begin{align*}
I_1^\ast&= \tilde{C}_1\int_{0}^\frac{1}{2}\sinh((1-\ep)w)\frac{f_1(w)}{w^2-\tilde{\mu}^2}dw,\\
I_2^\ast&= \tilde{C}_1\int_{\frac{1}{2}}^1\sinh((1-\ep)(1-w))\frac{f_1(w)}{w^2-\tilde{\nu}^2}dw.
\end{align*}
Proceeding exactly as before, we can conclude that
\[
\left(\sinh^2(b)-\sinh^2(a)\right)I_1^\ast I_2^\ast=O(\ep\log^2(\ep)),
\]
and
\[
\frac{I_1^\ast+I_2^\ast}{\tilde{C}_1}=\frac{1}{2}\sinh(a)\left(f_1'(1/2^-)-f_1'(1/2^+)\right)-(1-\ep)\cosh(a).
\]
Then,
\begin{align}\label{saltof1}
\frac{1}{2}\left(f_1'(1/2^-)-f_1'(1/2^+)\right)&=\frac{1}{\sinh(a)}\frac{I_1^\ast+I_2^\ast}{\tilde{C}_1}+(1-\ep)\frac{\cosh(a)}{\sinh(a)}\\
&=-\frac{1}{\sinh(a)\sinh(b)\tilde{C}_1}+(1-\ep)\frac{\cosh(a)}{\sinh(a)}+O(\ep\log^2(\ep)).\nonumber
\end{align}

Moreover, combining \eqref{presaltof1} and \eqref{saltof1} we obtain
\begin{multline*}
\frac{I_1+I_2}{\tilde{C}_n}> -\frac{\sinh(na)}{\sinh(a)\sinh(b)\tilde{C}_1}+(1-\ep)\frac{\sinh(na)\cosh(a)}{\sinh(a)} -n(1-\ep)\cosh(na)\\
+\frac{\delta}{4}\sinh(na)+\sinh(na)O(\ep\log^2(\ep)).
\end{multline*}
At this point we already have all the ingredients to conclude that
\begin{align*}
\text{det}&=1+\sinh(n b)(I_1+I_2)+(\sinh^2(n b))-\sinh^2(n a))I_1 I_2\\
&>1+\sinh(n b)(I_1+I_2)+\frac{1}{n}O(\ep\log^2(\ep))\\
&>1-\frac{\sinh(na)\sinh(nb)}{\sinh(a)\sinh(b)}\frac{\tilde{C_n}}{\tilde{C}_1}+(1-\ep)\tilde{C}_n \frac{\sinh(na)\sinh(nb)\cosh(a)}{\sinh(a)}\\
&\quad -n(1-\ep)\tilde{C}_n \sinh(nb)\cosh(na)+\sinh(nb)\tilde{C}_n\left(\frac{\delta}{4}\sinh(na)+\sinh(na)O(\ep\log^2(\ep))\right)\\
&\quad +\frac{1}{n}O(\ep\log^2(\ep)).
\end{align*}
Now, recalling that $\tilde{C}_n=\frac{2}{(1-\ep)n\sinh(n)}$, we obtain
\begin{multline}\label{detentrozos}
\text{det}>1-2\frac{\sinh(nb)\cosh(na)}{\sinh(n)}\\
+\frac{\sinh(na)\sinh(nb)}{n\sinh(n)\sinh(a)}\left(2\cosh(a)-\frac{\sinh(1)}{\sinh(b)}\right)\\
+\frac{\sinh(nb)\sinh(na)}{n\sinh(n)}\frac{\delta(1-\ep)}{2} \\
+\frac{\sinh(na)\sinh(nb)}{n\sinh(n)}O(\ep\log^2(\ep))+\frac{1}{n}O(\ep\log^2(\ep)).
\end{multline}
The next step is to study in detail each of the terms that appear in order to obtain a lower bound. In first place, since
\begin{align*}
\sinh(nb)\cosh(na)=\frac{1}{2}\sinh(n(a+b))-\frac{1}{2}\sinh(n(a-b))=\frac{1}{2}\sinh(n)+\frac{1}{2}\sinh(n\ep),
\end{align*}
we have that
\[
1-2\frac{\sinh(nb)\cosh(na)}{\sinh(n)}=1-\frac{\sinh(n)+\sinh(n\ep)}{\sinh(n)}=-\frac{\sinh(n\ep)}{\sinh(n)}.
\]
In addition,
\[
-\frac{\sinh(n\ep)}{\sinh(n)}=-\frac{\int_0^1\left(\sinh(n\ep s)\right)'ds}{\sinh(n)}=-(n\ep)\frac{\int_0^1\cosh(n\ep s)ds}{\sinh(n)}=-\frac{n\ep}{e^{n/2}}\frac{\int_0^1(e^{n\ep s}+ e^{-n\ep s}))ds}{e^{n/2}-e^{-3n/2}}.
\]
Consequently, as parameter $\ep>0$ is small enough, we have
\begin{equation}\label{trozodet1}
\left|1-2\frac{\sinh(nb)\cosh(na)}{\sinh(n)}\right|\leq C\frac{\ep}{n}.
\end{equation}
For the following term, we just have to remember that we have already proved (see \eqref{aux:boundsinh}) that
\begin{align*}
0\leq \frac{\sinh(nb)\sinh(na)}{\sinh(n)}\leq C
\end{align*}
uniformly in $n$ and $\ep$. Then, we have
\begin{equation}\label{trozodet2}
0\leq \frac{\sinh(na)\sinh(nb)}{n\sinh(n)\sinh(a)}\left(2\cosh(a)-\frac{\sinh(1)}{\sinh(b)}\right) \leq C\frac{\ep}{n},
\end{equation}
and for free we also have
\[
\frac{\sinh(na)\sinh(nb)}{n\sinh(n)}O(\ep\log^2(\ep)) \leq \frac{1}{n}O(\ep\log^2(\ep)).
\]
Finally, for the last term, we have
\begin{equation}\label{trozodet3}
\frac{\sinh(nb)\sinh(na)}{n\sinh(n)}\frac{\delta(1-\ep)}{2} +\frac{1}{n}O(\ep\log^2(\ep)) > c \frac{\delta}{n}+ \frac{1}{n}O(\ep\log^2(\ep)).
\end{equation}
Putting together \eqref{trozodet1}, \eqref{trozodet2} and \eqref{trozodet3} into \eqref{detentrozos} we finally arrive to
\[
\text{det}> \frac{\delta+o(1)}{n} ,
\]
which is strictly positive and gives us the desired contradiction.

\subsubsection{Monotonicity of $f_n$}
The only thing left  to show is that
\begin{equation}\label{fndecrece_n}
f_1(x)>f_2(x)>f_3(x)>  \ldots > 0, \qquad \text{for } x\in(0,1/2)\cup(1/2,1).
\end{equation}

\begin{remark}
By definition we know that $f_n(0)=0=f_n(1)$ and $f_n(1/2^-)=1=f_n(1/2^+), \,\, \forall n\geq 1.$
\end{remark}

To prove \eqref{fndecrece_n}, we distinguish two different cases:

\noindent
\underline{Case $x\in(0,1/2)$:}
Here, we recall that $f_n$ solves
\begin{align*}
f_n''(x)-\left(\frac{2}{x^2-\tilde{\mu}^2}+(1-\ep)^2 n^2\right)f_n(x)&=0, \qquad \text{for }0< x< 1/2,\\
f_n(0)&=0, \\
f_n(1/2^-)&=1.
\end{align*}
We will show that if $n_1< n_2$ with $n_1,n_2\in \N$ then $f_{n_1}(x)> f_{n_2}(x)$ for all $x\in(0,1/2)$.
To compare $f_{n_1}$ and $f_{n_2}$ we will use the method of variation of constants. We have that
\begin{align*}
(f_{n_1}-f_{n_2})''(x)-\left(\frac{2}{x^2-\tilde{\mu}^2}+(1-\ep)^2 n_1^2\right)(f_{n_1}-f_{n_2})(x)&=(1-\ep)^2(n_1^2-n_2^2)f_{n_2}(x),\\
(f_{n_1}-f_{n_2})(0)&=0,\\
(f_{n_1}-f_{n_2})(1/2^-)&=0.
\end{align*}

Next we consider the solution  $\bar{f}_{n_1}(x)$ of the equation
\begin{align*}
\bar{f}''_{n_1}(x)-\left(\frac{2}{x^2-\tilde{\mu}^2}+(1-\ep)^2n_1^2\right)\bar{f}_{n_1}(x)&=0, \qquad \text{for }0< x< 1/2,\\
\bar{f}_{n_1}(1/2)&=0\\
\bar{f}'_{n_1}(1/2)&=1.
\end{align*}
This solution exists and  is unique and  smooth since there are no singular points in $[0,1/2]$. Also, it is linearly independent of $f_{n_1}$. In addition, $\bar{f}_{n_1}(x)<0$ in $[0,1/2)$. Indeed, since $\bar{f}_{n_1}'(1/2)=1$ we have that $\bar{f}_{n_1}(x)<0$ for some interval $(x^*,1/2)$.  If there exists $x^\sharp\in(0,1/2)$ such that $\bar{f}_{n_1}(x^\sharp)=0$ and $\bar{f}_{n_1}(x)<0$ in $(x^\sharp,1/2)$, we must have that $\bar{f}_{n_1}'(x^\sharp)<0$ because $x=x^\sharp$ is an ordinary point. But then, by Lemma \ref{cor:positivoxstar}, $\bar{f}_{n_1}(1/2)<0$, which is also a contradiction. Then, we have proved that $\bar{f}_{n_1}(x)<0$ over the entire interval $[0,1/2)$.

Because Lemma \ref{l:positivityfn} and the method of variation of constants the solution satisfies
\begin{align}\label{eq:fn1-fn2}
(f_{n_1}-f_{n_2})(x)&=(1-\ep)^2(n_1^2-n_2^2)f_{n_1}(x)\int_{x}^\frac{1}{2}\frac{\bar{f}_{n_1}(w)}{W[f_{n_1},\bar{f}_{n_1}](w)}f_{n_2}
(w)dw \\
&\quad +(1-\ep)^2(n_1^2-n_2^2) \bar{f}_{n_1}(x)\int_{0}^x\frac{f_{n_1}(w)}{W[f_{n_1},\bar{f}_{n_1}](w)}f_{n_2}(w)dw. \nonumber
\end{align}

We recall that Wronskian is given by $$W[f_{n_1},\bar{f}_{n_1}](x)=f_{n_1}(x)\bar{f}_{n_1}'(x)-f_{n_1}'(x)\bar{f}_{n_1}(x).$$
Thus, we get
\begin{equation*}
W[f_{n_1},\bar{f}_{n_1}](1/2)= f_{n_1}(1/2)=1,
\end{equation*}
Therefore, as a consequence of Abel-Liouville's formula we  have that
$$W[f_{n_1},\bar{f}_{n_1}](x)=W[f_{n_1},\bar{f}_{n_1}](1/2) \qquad \text{for } 0 \leq x \leq 1/2.$$
So, in particular, we have seen that
\begin{equation*}
W[f_{n_1},\bar{f}_{n_1}](x)>0, \qquad \text{for } 0 \leq x \leq 1/2.
\end{equation*}

Combining everything, we  have seen that $f_{n_1}(x)$ is positive,  $\bar{f}_{n_1}(x)$ is negative and $W$ is positive, so that
\begin{align*}
f_{n_1}(x)\frac{\bar{f}_{n_1}(w)}{W[f_{n_1},\bar{f}_{n_1}](w)}< 0 \qquad \text{and} \qquad \bar{f}_{n_1}(x)\frac{f_{n_1}(w)}{W[f_{n_1},\bar{f}_{n_1}](w)}< 0,
\end{align*}
for all $w,x\in(0,1/2)$.

Since $f_{n_2}(x)$ is also positive and $n_1^2-n_2^2 <0$ we deduce from \eqref{eq:fn1-fn2} that $(f_{n_1}-f_{n_2})(x) > 0$.
As we have proved this for any $n_1,n_2\in\N$ satisfying $n_1< n_2$ we have finally proved that
\[
f_1(x)> f_2(x)> f_3(x)> \ldots>0, \qquad  \text{for $0< x < 1/2$}.
\]

Another important fact that can be deduced from the same  argument is the next one.
\begin{lemma}\label{desizq}
The following strict inequality holds:
\[
f_1'(1/2^-)-f_2'(1/2^-)<-\delta <0.
\]
uniformly in $\ep$, for $\ep$ small enough.
\end{lemma}
\begin{proof}
First of all, let us remark that we can construct, for $m=2$ in \eqref{f_left}, solutions $f_2$, $f_0^\sharp=f_2|_{\ep=0}$, and $g_1^\sharp$, analogous to $f=f_1$, $f_0=f_1|_{\ep=0}$ and $g_1$ respectively. These functions $f$, $f_0$ and $g_1$ can be found in section \ref{frobenius}. In addition, we have the same type of estimate for $d^\sharp=f_2-f_0^\sharp$, than for $d=f-f_0$ in Lemma $\ref{Linfty}$. Notice that $f_1$ and $f_2$ depend on $\ep$. We just write \eqref{eq:fn1-fn2} for the particular case $n_1=1$ and $n_2=2$. This implies that
\begin{align*}
(f_{1}-f_{2})(x)&=(1-\ep)^2(-3)f_{1}(x)\int_{x}^\frac{1}{2}\frac{\bar{f}_{1}(w)}{W[f_{1},\bar{f}_{1}](w)}f_{2}
(w)dw \\
&\quad +(1-\ep)^2(-3) \bar{f}_{1}(x)\int_{0}^x\frac{f_{1}(w)}{W[f_{1},\bar{f}_{1}](w)}f_{2}(w)dw. \nonumber
\end{align*}
Taking one derivate we get
\begin{align*}
(f_{1}-f_{2})'(x)&=(1-\ep)^2(-3)f_{1}'(x)\int_{x}^\frac{1}{2}\frac{\bar{f}_{1}(w)}{W[f_{1},\bar{f}_{1}](w)}f_{2}
(w)dw \\
&\quad +(1-\ep)^2(-3) \bar{f}_{1}'(x)\int_{0}^x\frac{f_{1}(w)}{W[f_{1},\bar{f}_{1}](w)}f_{2}(w)dw. \nonumber
\end{align*}
Then, evaluating the above expression and recalling that $\bar{f}_{1}'(1/2)=1$  and that $W[f_{1},\bar{f}_{1}](w)=1 $ we obtain

\begin{align*}
(f_{1}-f_{2})'(1/2^-)=&-3(1-\ep)^2\int_{0}^{1/2}f_{1}(w)f_{2}(w)dw\\=&-3(1-\ep)^2\int_{0}^{1/2}f_{0}(w)f^\sharp_{0}(w)dw+O(\ep\log(1/\ep)).
\end{align*}

%Because of \eqref{pex} we have that
%\begin{align*}
%f_1(x)\geq cf'_1(0)\frac{\sinh((1-\ep)x)}{(1-\ep)}, && f_2(x)\geq cf'_2(0)\frac{\sinh(2(1-\ep)x)}{2(1-\ep)}.
%\end{align*}
%Notice that in Lemma \ref{l:positivityfn} we have normalized to have derivative equal to one at zero. Therefore
%
%\[
%(f_{1}-f_{2})'(1/2^-)=-\frac{3}{2}f_1'(0)f_2'(0)\int_{0}^{1/2}\sinh((1-\ep)w)\sinh(2(1-\ep)w)dw.
%\]
%Taking a derivative in \eqref{diferencia} we have that
%
%\begin{align*}
%f_1'(0)-f_1'|_{\ep=0}(0)=\ep f'_1|_{\ep=0}(0)\int_0^\frac{1}{2}g_1(w)E(w)dw\leq C\ep \log(1/\ep)f_1'|_{\ep=0}(0)
%\end{align*}
%Since $f'_1|_{\ep=0}>0$, we find that $f'_1(0)\geq c>0$ uniformly in $\ep$, for $\ep$ small enough. We can reproduce this argument for $f_2$

Thus, we find that
\begin{align*}
(f_{1}-f_{2})'(1/2^-)<-\delta,
\end{align*}
uniformly in $\ep$, for $\ep$ small enough.
\end{proof}

\noindent
\underline{Case $x\in(1/2,1)$:}
Here, we recall that $f_n$ solves
\begin{align*}
f_n''(x)-\left(\frac{2}{x^2-\tilde{\nu}^2}+(1-\ep)^2 n^2\right)f_n(x)&=0, \qquad \text{for } 1/2< x\leq 1,\\
f_n(1)&=0, \\
f_n(1/2^+)&=1.
\end{align*}
We will show that if $n_1< n_2$ with $n_1,n_2\in \N$ then $f_{n_1}(x)> f_{n_2}(x)$ for all $x\in(1/2,1)$.
To compare $f_{n_1}$ and $f_{n_2}$ we will use the method of variation of constants. We have that
\begin{align*}
(f_{n_1}-f_{n_2})''(x)-\left(\frac{2}{x^2-\tilde{\nu}^2}+(1-\ep)^2 n_1^2\right)(f_{n_1}-f_{n_2})(x)&=(1-\ep)^2(n_1^2-n_2^2)f_{n_2}(x),\\
(f_{n_1}-f_{n_2})(1)&=0,\\
(f_{n_1}-f_{n_2})(1/2^+)&=0.
\end{align*}

As we did before, we can construct $\bar{f}_{n_1}(x)$,  satisfying
\begin{align*}
\bar{f}''_{n_1}(x)-\left(\frac{2}{x^2-\tilde{\nu}^2}+(1-\ep)^2n_1^2\right)\bar{f}_{n_1}(x)&=0, \qquad \text{for }1/2< x< 1,\\
\bar{f}_{n_1}(1/2)&=0,\\
\bar{f}'_{n_1}(1/2)&=1.
\end{align*}
%\begin{comment}
This solution exists and  is unique and  smooth since there are no singular points in $[1/2,1]$. Also, it is linearly independent of $f_{n_1}$. In addition, $\bar{f}_{n_1}(x)>0$ in $(1/2,1]$. Indeed, since $\bar{f}_{n_1}'(1/2)=1$ we have that $\bar{f}_{n_1}(x)>0$ for some interval $(1/2,x^*)$.  If there exists $x^\sharp\in(1/2,1)$ such that $\bar{f}_{n_1}(x^\sharp)=0$ and $\bar{f}_{n_1}(x)>0$ in $(1/2, x^\sharp)$, we must have that $\bar{f}_{n_1}'(x^\sharp)<0$ because $x=x^\sharp$ is an ordinary point. But then, by Lemma \ref{cor:positivoxstarright}, $\bar{f}_{n_1}(1/2)>0$, which is also a contradiction. Then, we have proved that $\bar{f}_{n_1}(x)>0$ over the entire interval $(1/2,1]$.
%\end{comment}

Then, using the method of variation of constants, the solution satisfies
\begin{align}\label{eq:fn1-fn2right}
(f_{n_1}-f_{n_2})(x)&=(1-\ep)^2(n_1^2-n_2^2)f_{n_1}(x)\int_\frac{1}{2}^x\frac{(-1)\bar{f}_{n_1}(w)}{W[f_{n_1},\bar{f}_{n_1}](w)}f_{n_2}
(w)dw \\
&\quad +(1-\ep)^2(n_1^2-n_2^2) \bar{f}_{n_1}(x)\int_x^1\frac{(-1)f_{n_1}(w)}{W[f_{n_1},\bar{f}_{n_1}](w)}f_{n_2}(w)dw. \nonumber
\end{align}

\begin{comment}
Since $\bar{f}_{n_1}(1/2)=0$ and $x=1/2$ is an ordinary point, we have that $\bar{f}_{n_1}'(1/2)\neq 0$. If $\bar{f}_{n_1}'(1/2)<0 $, $\bar{f}_{n_1}(x)<0$ in some interval $(1/2,x^*)$. But, because $\bar{f}_{n_1}(1)=1$ there must be a point $x^\sharp$ such that $\bar{f}_{n_1}(x)<0$ in $(1/2,x^\sharp)$ and $\bar{f}_{n_1}(x^\sharp)=0$. Since $x=x^\sharp$ is an ordinary point we get that $\bar{f}'_{n_1}(x^\sharp)>0$ and then we obtain, by Lemma \ref{cor:positivoxstarright}, that $\bar{f}_{n_1}(1/2)<0$ which is a contradiction.
Therefore $\bar{f}_{n_1}'(1/2)>0$ and $\bar{f}_{n_1}(x)>0$ for some interval $(1/2,x^*)$.  If there exists $x^\sharp\in(1/2,1)$ such that $\bar{f}_{n_1}(x^\sharp)=0$ and $\bar{f}_{n_1}(x)>0$ in $(1/2,x^\sharp)$, we must have that $\bar{f}_{n_1}'(x^\sharp)<0$ because $x=x^\sharp$ is an ordinary point. But then, by Lemma \ref{cor:positivoxstarright}, $\bar{f}_{n_1}(1/2)>0$, which is also a contradiction. Then, we have proved that $\bar{f}_{n_1}(x)>0$ over the entire interval $(1/2,1)$.
\end{comment}

%Since $\bar{f}_{n_1}$ satisfies that $\bar{f}_{n_1}(\tilde{\nu})=0$ and $\bar{f}_{n_1}'(\tilde{\nu})=1$ then we conclude that $\bar{f}_{n_1}(x)>0$ for $\tilde{\nu}<x<1.$ The proof of this fact follows by the same argument we made in $(0,\tilde{\mu})$ using  Lemma \ref{cor:positivoxstarright} instead of \ref{cor:positivoxstar}.

We recall that Wronskian is given by $$W[f_{n_1},\bar{f}_{n_1}](x)=f_{n_1}(x)\bar{f}_{n_1}'(x)-f_{n_1}'(x)\bar{f}_{n_1}(x).$$
Thus,  we get
\begin{equation*}
W[f_{n_1},\bar{f}_{n_1}](1)= -f'_{n_1}(1)\bar{f}_{n_1}(1)>0.
\end{equation*}
since we saw earlier that $f'_{n_1}(1)<0$. Therefore, as a consequence of Abel-Liouville's formula we  have that
$$W[f_{n_1},\bar{f}_{n_1}](x)>0,\quad 1/2\leq x\leq 1$$

Combining everything, we  have seen that $f_{n_1}(x)$ is positive, $\bar{f}_{n_1}(x)$ is  positive and $W$ is positive, so that
\begin{align*}
f_{n_1}(x)\frac{(-1)\bar{f}_{n_1}(w)}{W[f_{n_1},\bar{f}_{n_1}](w)}< 0 \qquad \text{and} \qquad \bar{f}_{n_1}(x)\frac{(-1)f_{n_1}(w)}{W[f_{n_1},\bar{f}_{n_1}](w)}< 0,
\end{align*}
for all $w,x\in(0,1/2)$.

Since $f_{n_2}(x)$ is also positive and $n_1^2-n_2^2 > 0$ we deduce from \eqref{eq:fn1-fn2right} that $(f_{n_1}-f_{n_2})(x) > 0$.
As we have proved this for any $n_1,n_2\in\N$ satisfying $n_1< n_2$ we have finally proved that
\[
f_1(x)> f_2(x)> f_3(x)> \ldots>0, \qquad  \text{for $1/2< x < 1$}.
\]

Moreover, proceeding as we did in Lemma \ref{desizq}, it is straightforward to check that

\begin{lemma}\label{desder}
The following strict inequality holds:
\[
f_1'(1/2^+) > f_2'(1/2^+)+\delta,
\]
for some $\delta>0$, uniformly in $\ep$, for $\ep$ small enough.
\end{lemma}

\section{Codimension of the image of the linear operator}

In this section we shall prove the following theorem

\begin{theorem}\label{thmcodimension} There exists $\ep_0>0$ sufficiently small  such that for all $0<\ep<\ep_0$ there exists $\lambda^*=\lambda^*(\ep)\in \R$ such that the kernel and the cokernel of the operator
\begin{align*}
\cL^{\lambda^\ast}\, :\, \S \to \S
\end{align*}
are one dimensional. In addition, the element $h^*$ that spans the kernel of $\cL^{\lambda^\ast}$ is $C^\infty(\T\times [0,a]\cup[b,1])$ and the function $\lambda^*-\Psi'_0(y)$ has not zeros in $[0,a]\cup [b,1]$.
\end{theorem}

Due to the Theorem \ref{mainexistencelemma} we just have to prove that the codimension of $\mathcal{L}^{\lambda^*}$ is one. Recall that $\mathcal{L}^\lambda$ is given by \eqref{defiL} and \eqref{operadorlineal}. On one hand, since $\lambda^*-\Psi'_0(y)$ has not zeros in $[0,a]\cup [b,1]$ and $\Psi_0'(y)$ is $C^\infty([0,a]\cup[b,1])$,  the multiplication operator by   $\lambda^*-\Psi'_0(y)$ is an isomorphism between $\S$ and $\S$. On the other hand, the operator
\begin{align*}
D_f\Psi_0\, :  \S\to C^{1+\beta}_{0,even}
\end{align*}
is well defined  for any $\beta<1$. Taking $\beta>\alpha$ we find that $D_f\Psi_0$ is compact from $\S\to\S$. This implies that $\mathcal{L}^{\lambda^*}$ is a Fredholm operator  of zero index. Thus, since the dimension of the kernel is one, the codimension of the image needs to be also one. Let us mention that we learn this argument from
\cite{GHS}, see also \cite{G2,GHM,GHM2,GHM3}.

Therefore, to apply the bifurcation argument it only remains to check the transversality condition in the Crandall-Rabinowitz Theorem.

\section{The transversality property}\label{s:transversality}

This section is devoted to the transversality assumption concerning the fourth and last hypothesis of the Crandall–Rabinowitz’s theorem.

We just have to check that
$$  D_{\l,f}^2 F[0,\lambda^\ast ]h^\ast \notin \text{Image}(D_f F[0,\l^\ast]),$$
where $h^\ast$ is the element that span the Kernel of $D_f F[0,\lambda^\ast]=\cL^{\lambda^\ast}$ and $\lambda^\ast$ is the value we have found. Since $D^2_{\l,f} F[0,\l^\ast]h^\ast=h^\ast$, we have to show that $h^\ast$ is not in the Image of $D_f F[0,\l^\ast]=\cL^{\lambda^\ast}
$.

This means that the equation
\begin{align}\label{aqui1}
(\lambda^\ast-v_{\varpi}+y^2)h(y)+2\int_{0}^1\chi(z)G_1(y,z)h(z)dz&=h^\ast(y), \quad \text{if $0\leq y<a$},\\
\left(\lambda^\ast-v_{\varpi}+(y-b)^2+2ay-a^2\right) h(y)+2\int_{0}^1\chi(z)G_1(y,z)h(z)dz&=h^\ast(y), \quad \text{if $b<y\leq1$},\label{aqui2}
\end{align}
where
\begin{align*}
(\lambda^\ast-v_{\varpi}+y^2)h^\ast(y)+2\int_{0}^1\chi(z)G_1(y,z)h^\ast(z)dz&=0, \quad \text{if $0\leq y<a$},\\
\left(\lambda^\ast-v_{\varpi}+(y-b)^2+2ay-a^2\right) h^\ast(y)+2\int_{0}^1\chi(z)G_1(y,z)h^\ast(z)dz&=0, \quad \text{if $b<y\leq 1$},
\end{align*}
has no solutions.

Suppose by reduction to the absurd that $h$ does indeed solve the system \eqref{aqui1}-\eqref{aqui2}. Recall that $\cL^{\lambda^\ast}=D_f F[0,\l^\ast]$. Thus, we have that
\begin{align*}
(\cL^{\lambda^\ast} h,h^\ast)_{L^2}=(h^\ast ,h^\ast)_{L^2}=||h^\ast||^2_{L^2},
\end{align*}
just by testing \eqref{aqui1}-\eqref{aqui2} against $h^\ast$. In addition, as the operator $\cL^{\lambda^\ast}$ is symmetric in the space $L^2\left([0,a]\cup[b,1]\right)$ we find that $(\cL^{\lambda^\ast} h,h^\ast)_{L^2}=(h,\cL^{\lambda^\ast} h^\ast)_{L^2}=0$ that give us a contradiction.

This concludes the last required condition of Crandall-Rabinowitz Theorem \ref{th:CR}.

\section{The main theorem}\label{s:mainthm}

After verifying all conditions for the application of the C-R theorem \ref{th:CR} and the discussion in Section \ref{s:formulationproblem} we obtain the following theorem:

\begin{theorem}\label{thm:detallado}  There exist $\ep_0$ small enough such that for every $0<\ep<\ep_0$  there exist a branch of solutions, $f_{\ep}^\sigma\in \S$ parameterize by $\sigma$, of equation  \eqref{ecuf1}, with $|\sigma|<\sigma_0$, for some small number $\sigma_0>0$, $\varpi$ as in \eqref{profile} and $\lambda=\lambda^\sigma_{\ep}.$ These solutions satisfy:
\begin{enumerate}
\item $f^\sigma_{\ep}(x,y)$ is $2\pi-$periodic on $x$.
\item The branch $$\ff^\sigma_{\ep}=\sigma h^\ast +o(\sigma) \quad\text{ in\quad  $\S$},$$
and the speed
$$\lambda^\sigma_{\ep}=\lambda^{\ast}+o(1),$$ where $(h^\ast,\l^\ast)$ are given in Theorem \ref{mainexistencelemma}.
\item  $f^\sigma_{\ep}(x,y)$ depends on $x$ in a nontrivial way.
\end{enumerate}
In addition, the vorticity $\omega^\sigma_{\ep}\in W^{1,\infty}(\T\times [-1,1])$,  given implicitly by
\begin{align*}
&\omega^\sigma_{\ep}(x_1,x_2)=\varpi(y),
\end{align*}
for $(x_1,x_2)=(x,y+f^\sigma_{\ep}(x,y))$ with $x\in\T$ and $y\in [-1,-\tfrac{1+\ep}{2})\cup (-\tfrac{1-\ep}{2},+\tfrac{1-\ep}{2})\cup(+\tfrac{1+\ep}{2},+1]$, where profile  function $\varpi(y)$ is given by \eqref{profile},
and
\[
\omega^\sigma_{\ep,\kappa,m}(x_1,x_2)=
\begin{cases}
-2\frac{1-\ep}{2}, \qquad x_1\in \T, \quad  +\frac{1-\ep}{2}+f^\sigma_\ep\left(x_1, +\frac{1-\ep}{2}\right)\leq x_2\leq +\frac{1+\ep}{2}+f^\sigma_\ep\left(x_1, +\frac{1+\ep}{2}\right),\\
+2\frac{1-\ep}{2}, \qquad  x_1\in \T \quad  -\frac{1+\ep}{2}+f^\sigma_\ep\left(x_1,- \frac{1+\ep}{2}\right)\leq x_2\leq -\frac{1-\ep}{2}+f^\sigma_\ep\left(x_1, -\frac{1-\ep}{2}\right),
\end{cases}
\]
yields a traveling way solution for 2D Euler in the sense that $$\omega^\sigma_{\ep}(x_1+\lambda^\sigma_{\ep}t,x_2)$$
satisfies  the system \eqref{eulerw}-\eqref{stream}. Importantly, $\omega^\sigma_{\ep}(x_1,x_2)$ depends non trivially on $x_1$.
\end{theorem}

Then, in order to prove Theorem \ref{thmbasic} it remains to prove that $H^{\frac{3}{2}-}(\T\times [-1,1])$-norm of $\omega^\sigma_{\ep}+2y$ can be made as small as we want and  we do this in Theorem \ref{t:distance}.

\subsection{Distance of the traveling wave to the Poiseuille flow}
The solution $\omega_{\ep}^\sigma$ obtained in
Theorem \ref{thm:detallado} satisfies the following statement:
\begin{theorem}\label{t:distance}
Fixed $0\leq \gamma < 3/2$, for all $\epsilon>0$, there exist
$\ep>0$ and $\sigma>0$ such that
\[
||\omega_\e^\sigma +2 y ||_{H^{\gamma}(\T\times[-1,1])}<\epsilon.
\]
\end{theorem}

\begin{proof}
Let us emphasize that we can make $||f_\ep^\sigma||_{C^{1+\a}(\T\times\supp(\varpi'))}$ arbitrarily small fixed $\ep$ taking $\sigma$ small. We have all the ingredients to obtain a quantitative estimate of the distance between the Poiseuille flow and the constructed traveling wave.

To alleviate the notation let us skip the subscript $\ep$  and the superscript $\sigma$ on $\omega_\ep^\sigma$  and $f_\ep^\sigma$ in the rest of the section. For simplicity we use the notation $\Omega\equiv \mathbb{T}\times[-1,1].$

In first place, we note that
\[
||\omega +2 y ||_{H^{\gamma}(\Omega)}=||\omega +2 y ||_{L^2(\Omega)}+||\omega +2 y ||_{\dot{H}^{\gamma}(\Omega)}.
\]

\underline{$L^2(\Omega)$-norm:} Here, we just note that
\begin{align*}
||\omega +2 y ||^2_{L^2(\Omega)}& = \int_{\Omega_f}|\omega(x,y)+2y|^2 \dx\dy+\int_{\Omega\setminus\Omega_f}|\omega(x,y)+2y|^2 \dx\dy.
\end{align*}
with
\[
\Omega_f=\{ (\bar{x},\bar{y})\in \T \times [-1,1]\, :\, (\bar{x},\bar{y})=\Phi_f(x,y)  \text{ with } x\in \T \text{ and } y\in \supp(\varpi')\}.
\]
On the one hand, recalling \eqref{profile} we have
\begin{align*}
\int_{\Omega_f}|\omega(x,y)+2y|^2 \dx\dy &=\int_{\T\times\supp(\varpi')} |\omega(\Phi_f(x,y))+2(y+f(x,y)) |^2\det\nabla \Phi_f(x,y)dxdy\\
&=\int_{\T}\int_{[-1,-b_\ep]\cup [-a_\ep,+a_\ep]\cup [+b_\ep,+1]}|\varpi(y)+2y +2f(x,y)|^2\det\nabla \Phi_f(x,y)dxdy\\
&=\int_{\T}\int_{ [-a_\ep,+a_\ep]}|2f(x,y)|^2\det\nabla \Phi_f(x,y)dxdy\\
&+\int_{\T}\int_{[-1,-b_\ep]\cup  [+b_\ep,+1]}|2(b_\ep-a_\ep) +2f(x,y)|^2\det\nabla \Phi_f(x,y)dxdy.
\end{align*}
These two terms can be made arbitrarily small by first making $\ep$ and then $\sigma$ small.

On the other hand, we have that
\begin{align*}
\int_{\Omega\setminus\Omega_f}|\omega(x,y)+2y|^2 =&\int_{\T}\int_{a_\ep+f(x,a_\ep)\leq y\leq b_\ep+f(x,b_\ep)}|-2a_\ep+2y|^2dydx \\&+\int_{\T}\int_{-b_\ep+f(x,-b_\ep)\leq y\leq -a_\ep+f(x,-a_\ep)}|2a_\ep+2y|^2dydx.
\end{align*}
These integrals can easily be made small as before by first making $\ep$ and then $\sigma$ small.

Therefore, we have proven that we can make the $L^2(\Omega)$-norm as small as we want.\\

\underline{$\dot{H}^{3/2^{-}}(\Omega)$-seminorm:}
We deal with the norm
\begin{align*}
||\nabla(\omega(x,y)+2y)||_{H^{1/2-}(\Omega)}.
\end{align*}

We use Sobolev–Slobodeckij seminorm  and will consider that
\begin{align*}
||\nabla(\omega +2 y) ||^2_{\dot{H}^{\gamma}(\Omega)}&= \int_\Omega \int_\Omega |\nabla\omega(\bx)-\nabla\omega(\by)|^2K(x_1-y_1,x_2-y_2)d\by d\bx
\end{align*}
where
$$K(x_1,x_2)= \frac{1}{(\sin^2\left(\frac{x_1}{2}\right)+x_2^2)^{1+\gamma}}.$$

It turns that, in order to estimate this norm, it will be enough to estimate
\begin{align*}
&I=\int_{\Omega^{ext}} \int_{\Omega{ext}} \frac{|\nabla\omega(\bx)-\nabla\omega(\by)|^2}{|\bx-\by|^{2+2\gamma}}d\by d\bx,
\end{align*}
where $\Omega^{ext}=[-2\pi,2\pi]\times [-1,1]\equiv T\times[-1,1]$. In the remainder of the article, we will use the periodic extension of the vorticity $\omega$ to the extended domain $\Omega^{ext}$.

In first place, we have
\begin{align*}
I&=\int_{\Omega^{ext}_f} \int_{\Omega^{ext}_f} \frac{|\nabla\omega(\bx)-\nabla\omega(\by)|^2}{|\bx-\by|^{2+2\gamma}}d\by d\bx + 2\int_{\Omega^{ext}_f} \int_{\Omega^{ext}\setminus\Omega^{ext}_f} \frac{|\nabla\omega(\bx)-\nabla\omega(\by)|^2}{|\bx-\by|^{2+2\gamma}}d\by d\bx\\
&=I_1+2 I_2,
\end{align*}
with
\[
\Omega^{ext}_f = \{(\bar{x},\bar{y})\in T \times [-1,1]\, :\, (\bar{x},\bar{y})=\Phi_f(x,y)  \text{ with } x\in T \text{ and } y\in \supp(\varpi')\}.
\]

Let us start working with the first integral term. Then,  making the appropriate change of variables, we obtain
\begin{equation*}
I_1=\int_{T\times\supp(\varpi')}\int_{T\times\supp(\varpi')}\frac{|\nabla\omega(\Phi_f(\bx))-\nabla\omega(\Phi_f(\by))|^2}{|\Phi_f(\bx)-\Phi_f(\by)|^{2+2\gamma}}(1+\pa_2 f(\bx))(1+\pa_2 f(\by))d\by d\bx.
\end{equation*}
In order to estimate $I_1$, we recall that
\[
\nabla \omega(\Phi_f(\bx))=\frac{\left(-\pa_1 f(\bx), 1\right)}{1+\pa_2 f(\bx)}\varpi'(y), \qquad \text{on } \supp(\nabla\omega).
\]
and the fact that $\varpi'= (-2)\chi_{[-1,-b_\ep]\cup[-a_\ep,+a_\ep]\cup[+b_\ep,+1]}$. Then, we get
\begin{multline*}
\frac{|\nabla\omega(\Phi_f(\bx))-\nabla\omega(\Phi_f(\by))|^2}{|\Phi_f(\bx)-\Phi_f(\by)|^{2+2\gamma}}(1+\pa_2 f(\bx))(1+\pa_2 f(\by))\\
=\frac{(-2)^2}{(1+\pa_2 f(\bx))(1+\pa_2 f(\by))}\frac{\left(\pa_1 f(\by)(1+\pa_2 f (\bx))-\pa_1 f(\bx)(1+\pa_2 f(\by))\right)^2+(\pa_2 f(\by)-\pa_2 f(\bx))^2}{|\Phi_f(\bx)-\Phi_f(\by)|^{2+2\gamma}}.
\end{multline*}

Since $\Phi_f:T\times\supp(\varpi')\rightarrow \Omega^{ext}_f$ is a diffeomorphism  and  $||f||_{C^{1+\alpha}(T\times\supp(\varpi'))}$ is small enough, we have that there exists universal positive constants  $c_1,c_2>0$ such that
\[
c_1\leq \frac{|\bx-\by|^{2+2\gamma}}{|\Phi_f(\bx)-\Phi_f(\by)|^{2+2\gamma}} \leq c_2.
\]
In addition, adding and subtracting appropriate terms in the numerator and using the fact that $||f||_{C^{1+\alpha}(T\times\supp(\varpi'))}$ is small enough we arrive to the upper bound
\begin{multline*}
\frac{|\nabla\omega(\Phi_f(\bx))-\nabla\omega(\Phi_f(\by))|^2}{|\Phi_f(\bx)-\Phi_f(\by)|^{2+2\gamma}}(1+\pa_2 f(\bx))(1+\pa_2 f(\by)) \\
\leq C    \frac{|(\pa_1 f(\by)-\pa_1 f(\bx))(1+\pa_2 f(\bx))+\pa_1 f(\bx)(\pa_2 f(\bx)-\pa_2 f(\by))|^2}{|\bx-\by|^{2+2\gamma}}\\
+C \frac{|\pa_2 f(\by)-\pa_2 f(\bx)|^2}{|\bx-\by|^{2+2\gamma}},
\end{multline*}
for some uniform constant $C$.

Then, we have seen that
\begin{multline*}
I_1=\int_{T\times\supp(\varpi')}\int_{T\times\supp(\varpi')}\frac{|\nabla\omega(\Phi_f(\bx))-\nabla\omega(\Phi_f(\by))|^2}{|\Phi_f(\bx)-\Phi_f(\by)|^{2+2\gamma}}(1+\pa_2 f(\bx))(1+\pa_2 f(\by))d\by d\bx\\
\leq C\int_{T\times\supp(\varpi')}\int_{T\times\supp(\varpi')}\frac{|\nabla f(\bx)-\nabla f(\by)|^2}{|\bx-\by|^{2+2\gamma}}d\by d\bx\\
=C\int_{T\times\supp(\varpi')}\int_{T\times\supp(\varpi')}\frac{|\nabla f(\bx)-\nabla f(\by)|^2}{|\bx-\by|^{2\alpha}}\frac{1}{|\bx-\by|^{2+2(\gamma-\alpha)}}d\by d\bx.
\end{multline*}
To conclude with $I_1$, we just need to note that  $||f||_{C^{1+\alpha}(T\times\supp(\varpi'))}=o(1)$ in terms of $\sigma$ implies that
\[
\frac{|\nabla f(\bx)-\nabla f(\by)|^2}{|\bx-\by|^{2\alpha}}=o(1).
\]
Now, recalling that in fact we are just interested in  $0<\gamma<1/2$ and since $1/2\leq \alpha<1$ we obtain that the last term integrate and we get the desired result. That is,
\begin{multline*}
C\int_{T\times\supp(\varpi')}\int_{T\times\supp(\varpi')}\frac{|\nabla f(\bx)-\nabla f(\by)|^2}{|\bx-\by|^{2\alpha}}\frac{1}{|\bx-\by|^{2+2(\gamma-\alpha)}}d\by d\bx\\
=C o(1)\int_{T\times\supp(\varpi')}\int_{T\times\supp(\varpi')}\frac{1}{|\bx-\by|^{2+2(\gamma-\alpha)}}d\by d\bx = C o(1)=o(1).
\end{multline*}

To finish with the proof we have to estimate $I_2$. We start noting that
\begin{equation*}
I_2=\int_{\Omega^{ext}_f} \int_{\Omega^{ext}\setminus\Omega^{ext}_f} \frac{|\nabla\omega(\bx)-\nabla\omega(\by)|^2}{|\bx-\by|^{2+2\gamma}}d\by d\bx=\int_{\Omega^{ext}_f} \int_{\Omega^{ext}\setminus\Omega^{ext}_f} \frac{|\nabla\omega(\bx)|^2}{|\bx-\by|^{2+2\gamma}}d\by d\bx
\end{equation*}
In addition, by using  that $\omega\in W^{1,\infty}(\Omega)$, we obtain
\[
I_2\leq C \int_{\Omega^{ext}_f} \int_{\Omega^{ext}\setminus\Omega^{ext}_f} \frac{1}{|\bx-\by|^{2+2\gamma}}d\by d\bx.
\]
Now, using the oddness of $f$ with respect to the vertical variable, we have that
\begin{multline*}
\int_{\Omega^{ext}_f} \int_{\Omega^{ext}\setminus\Omega^{ext}_f}\frac{1}{|\bx-\by|^{2+2\gamma}}d\by d\bx\\
=2\int_{T\times T}dx_1 dy_1 \int_{[0,a_\ep+f(x_1,a_\ep)]\cup[b_\ep+f(x_1,b_\ep),1]}dx_2\int_{[a_\ep+f(y_1,a_\ep),b_\ep + f(y_1,b_\ep)]}dy_2 \frac{1}{|\bx-\by|^{2+2\gamma}}\\
=2\int_{T\times T}dx_1 dy_1 \int_{0}^{a_\ep+f(x_1,a_\ep)}dx_2\int_{a_\ep+f(y_1,a_\ep)}^{b_\ep + f(y_1,b_\ep)}dy_2\frac{1}{\left[(x_1-y_1)^2+(x_2-y_2)^2\right]^{1+\gamma}}\\
+2\int_{T\times T}dx_1 dy_1 \int_{b_\ep+f(x_1,b_\ep)}^{1}dx_2\int_{a_\ep+f(y_1,a_\ep)}^{b_\ep + f(y_1,b_\ep)}dy_2\frac{1}{\left[(x_1-y_1)^2+(x_2-y_2)^2\right]^{1+\gamma}}.
\end{multline*}

We focus on just one term of the above summand, since the other follows by the same argument. For example, we provide all the details for the first one. That is,
\[
\int_{T\times T}dx_1 dy_1 \int_{0}^{a_\ep+f(x_1,a_\ep)}dx_2\int_{a_\ep+f(y_1,a_\ep)}^{b_\ep + f(y_1,b_\ep)}dy_2\frac{1}{\left[(x_1-y_1)^2+(x_2-y_2)^2\right]^{1+\gamma}}.
\]
In fact, taking $\sigma$ small enough, we can decompose the double integral in $(x_2,y_2)$-variables  as follows
\begin{multline*}
\int_{0}^{a_\ep + f(x_1,a_\ep)}dx_2 \int_{a_\ep + f(y_1,a_\ep)}^{b_\ep + f(y_1,b_\ep)}dy_2=\int_{0}^{a_\ep + f(x_1,a_\ep)}dx_2 \int_{a_\ep + f(y_1,a_\ep)}^{1/2}dy_2\\
+\int_{0}^{a_\ep + f(x_1,a_\ep)}dx_2 \int_{1/2}^{b_\ep + f(y_1,b_\ep)}dy_2,
\end{multline*}
where only the first double integral gives us a singularity in the denominator. For that reason, in the following we will focus only on estimating the most delicate term. That is,
\[
J=\int_{T}dx_1\int_{T}dy_1\int_{0}^{a_\ep+f(x_1,a_\ep)}dx_2\int_{a_\ep+f(y_1,a_\ep)}^{1/2}dy_2\frac{1}{\left[(x_1-y_1)^2+(x_2-y_2)^2\right]^{1+\gamma}}.
\]
By making appropriate change of variables, we get that $J$ can be written as
\[
\iint_{T\times T}dx_1 dy_1\int_{-f(x_1,a_\ep)}^{a_\ep} dx_2\int_{a_\ep}^{1/2-f(y_1,a_\ep)}dy_2 \frac{1}{\left[(x_1-y_1)^2+(x_2+f(x_1,a_\ep)-y_2-f(y_1,a_\ep))^2\right]^{1+\gamma}}.
\]
Fixed $\ep$, taking $\sigma$ sufficiently small, we have
\begin{align*}
(x_1-y_1)^2&+(x_2-y_2+f(x_1,a_\ep)-f(y_1,a_\ep))^2\\
&\geq (x_1-y_1)^2+(x_2-y_2)^2+2(x_2-y_2)(f(x_1,a_\ep)-f(y_1,a_\ep))\\
&\geq (x_1-y_1)^2\left[1-2\left(\frac{f(x_1,a_\ep)-f(y_1,a_\ep)}{x_1-y_1}\right)^2\right]+\frac{1}{2}(x_2-y_2)^2\\
&\geq (x_1-y_1)^2 \left(1-2||f||_{C^{1+\alpha}(T\times\sup(\varpi'))}^2\right)+\frac{1}{2}(x_2-y_2)^2\\
&\geq c\left((x_1-y_1)^2+(x_2-y_2)^2\right),
\end{align*}
for some positive constant $c$.

Therefore,
\begin{align*}
J&\leq C\int_{T}dx_1\int_{T}dy_1\int_{-f(x_1,a_\ep)}^{a_\ep}dx_2\int_{a_\ep}^{1/2-f(y_1,a_\ep)}dy_2 \frac{1}{\left[(x_1-y_1)^2+(x_2-y_2)^2\right]^{1+\gamma}}\\
&\leq C\int_{T}dx_1\int_{T}dy_1\int_{-||f||_{L^\infty}}^{a_\ep}dx_2\int_{a_\ep}^{1/2+||f||_{L^\infty}}dy_2\frac{1}{\left[(x_1-y_1)^2+(x_2-y_2)^2\right]^{1+\gamma}}.
\end{align*}
Now we compute
\begin{align*}
\int_{T}\int_{T}\frac{dx_1dy_1}{\left[(x_1-y_1)^2+(x_2-y_2)^2\right]^{1+\gamma}}&=\frac{1}{|x_2-y_2|^{2\gamma}}\int_{-\frac{2\pi}{|x_2-y_2|}}
^{\frac{2\pi}{|x_2-y_2|}}\int_{-\frac{2\pi}{|x_2-y_2|}}
^{\frac{2\pi}{|x_2-y_2|}}\frac{dx_1dy_1}{\left[1+(x_1-y_1)^2\right]^{1+\gamma}}\\
&\leq \frac{1}{|x_2-y_2|^{2\gamma}} \int_{-\frac{2\pi}{|x_2-y_2|}}
^{\frac{2\pi}{|x_2-y_2|}}dy_1 \int_{-\infty}
^{\infty}dx_1\frac{1}{\left[1+(x_1-y_1)^2\right]^{1+\gamma}}\\
&\leq \frac{C}{|x_2-y_2|^{2\gamma}}\int_{-\frac{2\pi}{|x_2-y_2|}}
^{\frac{2\pi}{|x_2-y_2|}}dy_1\leq \frac{C}{|x_2-y_2|^{1+2\gamma}}.
\end{align*}
Thus, we have
\begin{align*}
J&\leq C\int_{-||f||_{L^\infty}}^{a_\ep}\int_{a_\ep}^{1/2+||f||_{L^\infty}}\frac{1}{(y_2-x_2)^{1+2\gamma}}dy_2 dx_2\\
&\leq \frac{C}{\gamma} \int_{-||f||_{L^\infty}}^{a_\ep} \left(\frac{1}{(a_\ep-x_2)^{2\gamma}}-\frac{1}{(1/2+||f||_{L^\infty}-x_2)^{2\gamma}}\right)dx_2\\
&\leq \frac{C}{\gamma(1-2\gamma)}\left[ (a_\ep+||f||_{L^\infty})^{1-2\gamma}-(1/2+2||f||_{L^\infty})^{1-2\gamma}+(1/2-a_\ep+||f||_{L^\infty})^{1-2\gamma}\right].
\end{align*}
Since $0<\gamma<1/2$, $1/2-a_\ep=\ep/2$, and we can make $||f||_{L^\infty}$ as small as we want by taking $\sigma$ small enough, we have proved that $J$ is of order $o(1)$.

To sum up, combining all we have proved that
\[
||\nabla(\omega +2 y) ||^2_{\dot{H}^{1/2^-}(\Omega)} =o(1).
\]

Consequently, for any $\epsilon > 0$ and for any $0 \leq \gamma < 3/2$, taking $\ep$ and $\sigma$ small enough, we find a
traveling wave such that its vorticity satisfies $||\omega^\sigma_{\ep}+2y||_{H^{\gamma}(\T\times[-1,1])}<\epsilon.$
\end{proof}

\textbf{Acknowledgments:}  A.C. acknowledge financial support from the Severo Ochoa Programme for Centres of Excellence Grant CEX2019-000904-S funded by MCIN/AEI/10.13039/501100011033, Grant PID2020-114703GB-I00 funded by MCIN/AEI/10.13039/501100011033, Grants RED2022-134784-T and RED2018-102650-T funded by MCIN/AEI/10.13039/501100011033
and by a 2023 Leonardo Grant for Researchers and Cultural Creators, BBVA Foundation. The BBVA Foundation accepts no responsibility for the opinions, statements, and contents included in the project and/or the results thereof, which are entirely the responsibility of the authors.

D.L. is  supported by RYC2021-030970-I funded by MCIN/AEI/10.13039/501100011033 and the NextGenerationEU. D.L also acknowledge financial support from Grant PID2020-114703GB-I00 and Grant PID2022-141187NB-I00 funded by MCIN/AEI/10.13039/501100011033.

%\bibliography{references}
\bibliographystyle{abbrv}

\end{document}